\documentclass[11pt]{amsart}

\usepackage{amssymb,verbatim}
\usepackage{amsmath,amsfonts}
\usepackage[mathscr]{euscript} % make a nice \Net
\usepackage{amsthm}
\usepackage{url}
\usepackage{graphicx} %to include figure
\usepackage{enumitem} %to use nice enumeration with changed indent
\usepackage{hyperref} % to make links for lemmas and theorems
\usepackage{cleveref}
\hypersetup{colorlinks} %so the links look nicer
\usepackage{bbm} % for the characteristic function 1
\usepackage{bm} % for bold greek letter
\usepackage{cancel} %to make sout to cross things out

\usepackage[colorinlistoftodos]{todonotes}

\RequirePackage{cleveref}
\usepackage{hypcap}
\hypersetup{colorlinks=true, citecolor=darkblue, linkcolor=darkblue}
\definecolor{darkblue}{rgb}{0.0,0,0.7}
\newcommand{\darkblue}{\color{darkblue}}

\definecolor{darkred}{rgb}{0.68,0,0}

\definecolor{darkgreen}{rgb}{0,.38,0}

\newcommand{\defn}[1]{\emph{\darkblue #1}}
\newcommand{\defnb}[1]{\emph{\darkblue #1}}

\def\defna{\defn}
\def\defng{\defn}

\setlist[enumerate]{
	label=\textnormal{({\roman*})},
	ref={\roman*}}

\makeatletter
\def\th@plain{%
	\thm@notefont{}% same as heading font
	\itshape % body font
}
\def\th@definition{%
	\thm@notefont{}% same as heading font
	\normalfont % body font
}
\makeatother

\newtheorem{thm}{Theorem}[section]
\newtheorem{lemma}[thm]{Lemma}

\newtheorem*{claim*}{Claim}
\newtheorem{cor}[thm]{Corollary}
\newtheorem{prop}[thm]{Proposition}
\newtheorem{conj}[thm]{Conjecture}
\newtheorem{conv}[thm]{Convention}
\newtheorem{question}[thm]{Question}

\newtheorem{op}[thm]{Open Problem}

\theoremstyle{definition}

\newtheorem{definition}[thm]{Definition}

\numberwithin{figure}{section}
\numberwithin{equation}{section}

\newtheorem*{reptheoreminner}{\reptheoremheading}

\newenvironment{reptheorem}[2][]{%
	\def\reptheoremheading{Theorem~#2%
		\if\relax\detokenize{#1}\relax
		\else\space(#1)%
		\fi
	}%
	\begin{reptheoreminner}%
	}{%
	\end{reptheoreminner}%
}

\newtheorem*{replemmainner}{\replemmaheading}

\newenvironment{replemma}[2][]{%
	\def\replemmaheading{Lemma~#2%
		\if\relax\detokenize{#1}\relax
		\else\space(#1)%
		\fi
	}%
	\begin{replemmainner}%
	}{%
	\end{replemmainner}%
}
\def\emp{\nothing}

\def\zz{\mathbb Z}
\def\nn{\mathbb N}

\def\rr{\mathbb R}
\def\qqq{\mathbb Q}

\def\ff{\mathbb F}

\def\fq{{\mathbb F}_q}

\def\sm{\smallsetminus}

\def\la{\lambda}
\def\ga{\gamma}

\def\al{\alpha}
\def\be{\beta}

\def\vp{\varphi}

\def\cB{\mathcal B}
\def\cI{\mathcal I}

\def\cF{\mathcal F}

\def\ssu{\subset}

\def\<{\langle}
\def\>{\rangle}

\def\ups{\upsilon}
\def\vt{\ze}

\def\0{{\mathbf 0}}

\def\nothing{\varnothing}

\def\.{\hskip.06cm}
\def\ts{\hskip.03cm}

\def\ba{\textbf{\textbf{a}}}

\def\bc{\textbf{\textbf{c}}}

\def\bA{\textbf{\textit{A}}}
\def\bO{\textbf{\textbf{0}}}

\def\kK{{\Bbbk}}
\def\ze{{\zeta}}

\newcommand{\girth}{\mathrm{girth}}

\def\ah{\textrm{h}}
\def\aI{{ \textrm{I} } }
\def\aIr{\textrm{\em I}}

\def\aM{\textrm{M}}

\def\aN{\textrm{N}}

\def\ag{\textrm{g}}

\def\ah{\textrm{h}}

\def\.{\hskip.06cm}
\def\ts{\hskip.03cm}
\def\nin{\noindent}

\newcommand{\textsu}[1]{\textup{\textsf{#1}}}

\newcommand{\ComCla}[1]{\textup{\textsu{#1}}}
\newcommand{\sharpP}{\ComCla{\#P}}
\newcommand{\SP}{\ComCla{\#P}}

\newcommand{\Sigmap}{\ensuremath{\Sigma^{{\textup{p}}}}}

\newcommand{\NP}{\ComCla{NP}}

\newcommand{\coNP}{\ComCla{coNP}}
\renewcommand{\P}{\ComCla{P}}
\newcommand{\CeqP}{\ComCla{C$_=$P}}

\newcommand{\PH}{\ComCla{PH}}
\newcommand{\FP}{\ComCla{FP}}
\newcommand{\PP}{\ComCla{PP}}
\newcommand{\PSPACE}{\ComCla{PSPACE}}

\def\SP{\sharpP}
\def\poly{{\P}}
\def\CEP{{\CeqP}}

\newcommand{\iB}{\textnormal{B}} %% Macro for convex bodies
\DeclareMathOperator{\cb}{\mathbf{c}} % characteristic vector
\DeclareMathOperator{\eb}{\mathbf{e}} % standard basis vector
\DeclareMathOperator{\Ec}{\mathcal{E}} %The set of all linear extensions
\DeclareMathOperator{\hb}{\mathbf{h}} % characteristic vector
\DeclareMathOperator{\Rb}{\mathbb{R}} %Real numbers
\DeclareMathOperator{\vb}{\mathbf{v}} % vector
\DeclareMathOperator{\wb}{\mathbf{w}} % vector
\DeclareMathOperator{\zb}{\mathbf{z}} % a vector
\DeclareMathOperator{\rk}{\textnormal{rk}}

\DeclareMathOperator{\Ef}{\Theta} %  Edge set
\DeclareMathOperator{\ef}{\text{\it e}} %  edges
\DeclareMathOperator{\Qf}{ {\Gamma} } %  Quiver
\DeclareMathOperator{\Vf}{\Omega} %  Vertex set
\DeclareMathOperator{\Vfm}{\Omega^0} %  Vertex set
\DeclareMathOperator{\Vfp}{\Omega^+} %  Vertex set

\DeclareMathOperator{\vf}{{\text{\it \sts v}}} %  vertex
\newcommand{\vfs}{D_v} %  vertex
\DeclareMathOperator{\wf}{{\text{\it \sts w}}} %  vertex

\def\sts{\hskip.015cm}

\newcommand\ar{{d}} % I put textnormal mode so I remember that it is a macro
\def\arr{d}

\def\av{\textrm{v}}

\def\AA{\mathbb A}
\def\bM{\textbf{\textrm{M}}\hskip-0.03cm{}}
\def\bMr{\textbf{\textrm{\em M}}\hskip-0.03cm{}}
\def\bT{\mathbf{T}}

\def\af{\textrm{f}}
\DeclareMathOperator{\fb}{\mathbf{f}}
\DeclareMathOperator{\gb}{\mathbf{g}}

\def\asr{\textrm{\em s}}

\def\aC{\textrm C}
\DeclareMathOperator{\bC}{\textbf{\textrm{C}}\hskip-0.03cm{}}

\newcommand{\Mf}{\mathscr{M}} % a matroid
\newcommand{\Nf}{\mathscr{N}} % a matroid
\DeclareMathOperator{\bS}{\mathbf{S}}
\DeclareMathOperator{\bSc}{\mathbf{Sc}}
\DeclareMathOperator{\bScp}{\mathbf{S^\prime c^\prime}}

\DeclareMathOperator{\bzc}{\mathbf{z \ts c}}
\DeclareMathOperator{\psc}{\textnormal{p}} % Integers

\DeclareMathOperator{\EBULC}{\textsc{EqualitySY}}
\DeclareMathOperator{\EMason}{\textsc{EqualityMason}}

\DeclareMathOperator{\CDC}{\textsc{CoincidenceDC}}
\DeclareMathOperator{\CDCR}{\textsc{BasisRatioCoincidence}}

\DeclareMathOperator{\VDCR}{\textsc{BasisRatioVerification}}

\DeclareMathOperator{\BBM}{\#\textsc{Bases}}
\DeclareMathOperator{\NDCR}{\#\textsc{BasisRatio}}
\DeclareMathOperator{\aP}{\textnormal{P}}
\DeclareMathOperator{\at}{\tau} %number of spanning trees
\def\cJ{\mathcal J}

\newcommand{\supp}{\textnormal{supp}}
\DeclareMathOperator{\as}{\textnormal{s}}
\DeclareMathOperator{\Par}{\textnormal{Par}}
\DeclareMathOperator{\Comp}{\textnormal{Comp}}%Set of compatible letters
\DeclareMathOperator{\NL}{\textnormal{NL}}%Set of non-loops

\title
[Equality cases of the Stanley--Yan inequality]
{Equality cases of the Stanley--Yan \\
 log-concave matroid inequality}
\date{\today}

\author{Swee Hong Chan}
\address[Swee Hong Chan]{Department of Mathematics, Rutgers University,  Piscataway, NJ 08854.}
\email{\texttt{sc2518@rutgers.edu}}

\author[\ts Igor Pak]{Igor Pak}
\address[Igor Pak]{Department of Mathematics, UCLA,  Los Angeles, CA 90095.}
\email{\texttt{pak@math.ucla.edu}}

\begin{document}

\begin{abstract}
The \emph{Stanley--Yan} (SY) \emph{inequality} \ts gives the ultra-log-concavity
for the numbers of bases of a matroid which have given sizes of intersections
with $k$ fixed disjoint sets.  The inequality was proved by Stanley (1981)
for regular matroids, and by Yan (2023) in full generality.  In the original
paper, Stanley asked for equality conditions of the SY~inequality, and proved
total equality conditions for regular matroids in the case $k=0$.
%Yan conjectured that Stanley's conditions hold for general matroids.

In this paper, we completely resolve Stanley's problem. First, we obtain
an explicit description of the equality cases of the SY inequality for
$k=0$, extending Stanley's results to general
matroids and removing the ``total equality'' assumption.
Second, for $k\ge 1$, we prove that the equality cases of the SY inequality
for binary matroids cannot be described in a sense that they are not in the
polynomial hierarchy unless the polynomial hierarchy collapses to a finite level.
\end{abstract}
	
\maketitle
	
\vskip-.5cm

%\newpage

\section{Introduction}\label{s:intro}

\subsection{Foreword}\label{ss:intro-for}
Among combinatorial objects, \emph{matroids} \ts are fundamental and
have been extensively studied in both combinatorics % (see e.g.\ \cite{Ox}),
and their applications to other fields (see e.g.\ \cite{Ox,Schr03}).  In recent years, a remarkable
progress has been made towards understanding  \emph{log-concave} \ts
matroid inequalities for various matroid parameters (see e.g.\ \cite{Huh,Kalai}).
Much less is known about their equality conditions as they remain inaccessible
by algebraic techniques (see Section~\ref{s:hist}).

The \emph{Stanley--Yan} (SY) \emph{inequality} is a very general log-concave
inequality for the numbers of bases of a matroid which have given sizes of
intersections with $k$ fixed sets.  In this paper we completely resolve
Stanley's open problem \cite[p.~60]{Sta-AF}, asking for equality conditions
for the SY inequality, although probably not in the
way Stanley had expected: we give a positive result for $k=0$
and a negative result for $k\ge 1$.

Since known proofs of the SY inequality are independent of~$k$,
it came as a surprise that the equality conditions have a completely
different nature for different~$k$.  Curiously, our negative result
is formalized and proved in the language of computational complexity.
Even as a conjecture this was inconceivable until our recent work
(cf.~$\S$\ref{ss:finrem-quote}).

% For \ts $k=0$, we give an explicit description of equality cases of the
% Stanley--Yan inequality.  In a break with much of the literature, we use
% the \emph{combinatorial atlas} technology to prove the result.
% This is a technical linear algebraic approach we
% developed in \cite{CP-intro,CP} to prove both the inequalities and
% the equality cases of matroid and poset inequalities.
%
% On the other hand, for \ts $k>0$ \ts we show that no such explicit description
% exists.  The latter result is formalized in the language of computational complexity
% and uses standard assumptions in the area.  This is a rare of application of computational
% complexity to a problem in combinatorics.  To obtain the result we use two
% tools we previously developed: the \emph{combinatorial coincidences} \ts approach
% developed in \cite{CP-coinc,CP-AF}, and the analysis of continued fractions
% we used in \cite{CP-AF,CP-CF}.  We postpone a detailed historical
% overview until the next section, and proceed to state the results. % Section~\ref{s:hist}.

\smallskip

\subsection{Stanley's problem}\label{ss:intro-main}
Let \ts $\Mf$ \ts be a matroid of rank \ts $r=\rk(\Mf)$, with a ground set~$X$
of size \ts $|X|=n$.  Denote by \ts $\cB(\Mf)$ \ts the set of bases of~$\Mf$.
This is a collection of $r$-subsets of~$X$.  Fix integers \ts $k \geq 0$ \ts and
\. $0\le a, c_1,\ldots,c_k \le r$. Additionally, fix disjoint subsets \ts
$R, S_1,\ldots, S_k \ssu X$.
%
% $\bS:=(S_1,\ldots, S_k)$ \ts be disjoint subsets of $X-R$, and let \ts $\cb:=(c_1,\ldots, c_k) \in \nn^k$\ts.
Define
\[
\cB_{\bSc}(\Mf,R,a)  \ := \ \big\{ \ts A \in \cB(\Mf)  \, : \,  |A \cap R|=a, \.  |A \cap S_1|=c_1, \. \ldots \., \. |A \cap S_k|=c_k  \ts \big\}\ts,
\]
where \ts $\bS = (S_1,\ldots,S_k)$ \ts and \ts $\bc = (c_1,\ldots,c_k)$.
Denote \. $\iB_{\bSc}(\Mf,R,a):=|\cB_{\bSc}(\Mf,R,a)|$, and let
\[ \aP_{\bSc}(\Mf,R,a) \ := \  \iB_{\bSc}(\Mf,R,a) \. \tbinom{r}{a \ts , \ts c_1 \ts , \ldots , \ts c_k \ts , \ts \ups }^{-1},
%\quad \text{where \quad $\ups = r-a-c_1-\ldots-c_k$\ts.}
\]
where \. $\ups = r-a-c_1-\ldots-c_k$\ts.
In words, $\iB_{\bSc}(\Mf,R,a)$ counts the bases containing exactly $a$ elements of $R$ and exactly $c_i$ elements of $S_i$ for each $i$, whereas $\aP_{\bSc}(\Mf,R,a)$ is the probability that a random set satisfying these same cardinality constraints is a basis.
 See $\S$\ref{ss:def-matroids} for the
matroid definition and notation.

\smallskip

\begin{thm}[{\rm \defn{Stanley--Yan inequality}, \cite[Thm~2.1]{Sta-AF} and \cite[Cor.~3.47]{Yan23}}{}]\label{t:SY}
\begin{equation*}\label{eq:SY}\tag{SY}
{\aP_{\bSc}(\Mf,R,a)}^2 \ \geq \ \aP_{\bSc}(\Mf,R,a+1) \. \aP_{\bSc}(\Mf,R,a-1).
\end{equation*}
\end{thm}

\smallskip

This inequality was discovered by Stanley who proved it for regular (unimodular)
matroids using the \emph{Alexandrov--Fenchel inequality}.  The inequality was
extended to general matroids by Yan \cite{Yan23}, using \emph{Lorentzian polynomials}.
Both proofs are independent of~$k$.

To motivate the result, Stanley showed in \cite[Thm~2.9]{Sta-AF} (see also \cite[Thm~3.48]{Yan23}),
that the  Stanley--Yan (SY) inequality for \ts $k=0$ \ts implies the \emph{Mason--Welsh conjecture} (1971, 1972),
see \eqref{eq:Mason-1} in~$\S$\ref{ss:hist-LP}.
This is a log-concave inequality for the number of independent sets of a matroid (see~$\S$\ref{ss:hist-LP}),
which remained a conjecture until Adiprasito, Huh and Katz \cite{AHK}
famously proved it in full generality using \emph{combinatorial Hodge theory}.

% In other words, the SY inequality says that the sequence \. $\{\aP_{\bSc}(\Mf,R,a)\. : \. 0< a <r \}$ \.
% is log-concave.

In~\cite[$\S$2]{Sta-AF}, Stanley asked for equality conditions for \eqref{eq:SY}
and proved partial results in this direction (see below).  Despite major developments
on matroid inequalities, no progress on this problem has been made until
now.  We give a mixture of both positive and negative results which completely resolve
Stanley's problem.  We start with the latter.

\smallskip

\subsection{Negative results} \label{ss:intro-negative}
Let \ts $\Mf$ be a \emph{binary matroid} \ts given by its representation over $\mathbb{F}_2\ts$, and
let $k\ge 0$, \ts $R\subseteq X$, \ts $\bS\in X^k$, \ts $a\in \nn$, \ts $\bc \in \nn^k$ \ts be as above.
Denote by \. $\EBULC_k$ \.
the decision problem
\[
\EBULC_k \ := \  \big\{\ts \aP_{\bSc}(\Mf,R,a)^2 \. =^? \. \aP_{\bSc}(\Mf,R,a+1) \. \aP_{\bSc}(\Mf,R,a-1) \ts\big\}.
\]
Here the input to the problem consists of a representation of the matroid \ts $\Mf$, subsets \ts $R, S_1,\ldots,S_k$ \ts
of the ground set, and integers \ts $a,c_1,\ldots,c_k$.  The integer $k$ is the only fixed parameter.

\smallskip

\begin{thm}[{\rm \defn{\ts $k\ge 1$ \ts case}}{}]\label{t:main-negative}
For all \ts $k \geq 1$\ts, we have:
	$$\EBULC_k \ts \in \PH \ \ \Longrightarrow \ \ \PH=\Sigmap_m  \quad \ \text{for some} \ \, m,
	$$
for binary matroids.
Moreover, the result holds for \ts $a=1$ and $c_1=r-2$\ts.
\end{thm}

\smallskip

This gives a negative solution to Stanley's problem for \ts $k\ge 1$.
Informally, the theorem states that equality cases of the Stanley--Yan
inequality \eqref{eq:SY} cannot be \emph{described} \ts using a finite
number of alternating quantifiers \ts $\exists$ \ts and \ts $\forall$,
unless a standard complexity assumptions fails (namely, that the
\emph{polynomial hierarchy} \ts $\PH$ \ts collapses to a finite
level\footnote{This is a standard assumption in theoretical
computer science that is similar to \ts $\poly \ne \NP$ (stronger, in fact),
and is widely believed by the experts.  If false it
would bring revolutionary changes to the field, see e.g.\ \cite{Aar16,Wig23}.}).
This is an unusual application of computational complexity to a problem in
combinatorics (cf.~$\S$\ref{ss:finrem-quote}).
The proof of Theorem~\ref{t:main-negative} is given in Section~\ref{s:main-proof},
and uses technical lemmas developed in  Section~\ref{s:red}--\ref{sec:verify}.

The theorem does not say that no \emph{geometric} \ts description of \eqref{eq:SY}
can be obtained, or that some large family of equality cases cannot be described.
In fact, the vanishing cases we present below (see Theorem~\ref{t:vanish}), is
an example of the latter.

The proof of Theorem~\ref{t:main-negative} uses the \emph{combinatorial
coincidences} \ts approach developed in \cite{CP-coinc,CP-AF}.
We also use the analysis of the \emph{spanning tree counting function} \ts
through continued fractions (see $\S$\ref{ss:intro-st} below).
Paper \cite{CP-AF} is especially notable, as it can be viewed both a
philosophical and (to a lesser extent) a technical prequel to this paper.
There, we prove that the equality cases of the Alexandrov--Fenchel
inequality are not in $\PH$ for \emph{order polytopes} (under the same assumptions).
See~$\S$\ref{ss:finrem-binary} for possible variations of the theorem
for  other classes of matroids.

\smallskip

\subsection{Positive results} \label{ss:intro-positive}
For \ts $k=0$, we omit the subscripts:
$$
\iB(\Mf,R,a) \. := \. \big|\big\{A\in \cB(\Mf) \. : \. |A \cap R|=a\big\}\big| \quad
\text{and} \quad
\aP(\Mf,R,a) \. := \. \iB(\Mf,R,a) \. \tbinom{r}{a}^{-1}.
$$
%, to simplify the notation.
Denote by \ts $\NL(\Mf)$ \ts the set of \emph{non-loops} in~$\Mf$, i.e. elements \ts $x \in X$ \ts such that \. $\{x\}$ \. is an independent set.
For a non-loop \ts $x\in \NL(\Mf)$, denote by \ts
$\Par_{\Mf}(x)\subseteq X$ \ts the set of elements  of \ts $\Mf$ \ts
that are \emph{parallel} to~$x$, i.e. elements \ts $y \in X$ \ts such that \. $\{x,y\}$ \. is not an independent set.
The following two results (Theorem~\ref{t:main-positive} and Proposition~\ref{p:vanish})  give a positive solution to Stanley's problem for \ts $k=0$\ts.

\smallskip

\begin{thm}[{\rm \defn{\ts $k=0$ \ts case}}{}]\label{t:main-positive}
Let \ts $\Mf$ \ts be a matroid of rank~$r\geq 2$ with a ground set~$X$.
 Let \ts $R \ssu X$,
and let \ts $1\le a \le r-1$.  Suppose that \. $\aP(\Mf,R,a)>0$. Then the equality
	\begin{equation}\label{eq:comb-1}
			 {\aP(\Mf,R,a)}^2 \ = \ \aP(\Mf,R,a+1) \. \aP(\Mf,R,a-1)
	\end{equation}
holds \.	\underline{if and only if} \. there exists $\as >0$ such that,
for every independent set \ts $A \ssu X$ \ts s.t.\ \. $|A| =  r-2$ \ts and \ts  $|A \cap R| =  a-1$,
and every non-loop \ts $x \in \NL(\Mf/A)$,
we have:	\begin{equation}\label{eq:comb-2}
	 |\Par_{\Mf/A}(x)  \cap R| \ = \  \as \ts |\Par_{\Mf/A}(x)  \cap (X-R)|.
	\end{equation}	
\end{thm}

%\com{SH}{The condition in Theorem~\ref{t:main-positive}  might be co-NP complete.}
%\com{SH}{If the input of $\EBULC_0$ is changed to general matroids and the size of the input is the number of bases of $\bM$, then Theorem~\ref{t:main-positive} shows that  $\EBULC_0$ is in fact  in $\P$. This raises the next question, what is the standard way of encoding general finite matroids?}

\smallskip

In particular, Theorem~\ref{t:main-positive} resolves a conjecture of Yan (Conjecture~\ref{conj:Yan} in \S\ref{ss:intro-total} below).
We prove the theorem in Section~\ref{s:main-positive} using the
\emph{combinatorial atlas} \ts technology, see~$\S$\ref{ss:hist-atlas}.
This is a technical linear algebraic
approach we developed in \cite{CP-intro,CP} to prove both the inequalities
and the equality cases of matroid and poset inequalities, as well as their
generalizations.  Notably, we obtain the equality cases of Mason's
ultra-log-concave inequality \cite[$\S$1.6]{CP}, and we use a closely
related setup in this case.

\smallskip

\subsection{Vanishing conditions}\label{ss:intro-vanish}
Note that when \ts $\aP(\Mf,R,a)=0$, we always have equality in~\eqref{eq:comb-1}.
The following nonvanishing conditions give a description to such equality cases:

\smallskip

\begin{prop}[{\rm \defn{nonvanishing conditions for $k=0\ts$}}{}]
\label{p:vanish}
Let \. $\Mf$ \. be a matroid of rank \ts $r=\rk(\Mf)$ \ts with a ground set~$X$,
and let \ts  $R \subseteq X$.
Then, for every \ts $0\le a \le r$, we have:
	\.  $\aP(\Mf,R,a)>0$ \, \underline{if and only if}
	\[   r \. - \. \rk(X \sm R)  \, \leq \, a \, \leq \,  \rk(R). \]
\end{prop}

\smallskip

The proposition is completely straightforward and is a special case of a more general
Rado's~Theorem~\ref{t:vanish}, see below.
Combined, Theorem~\ref{t:main-positive} and  Proposition~\ref{p:vanish} give a
complete description of equality cases of the Stanley--Yan~inequality~\eqref{eq:SY}
for \ts $k=0$.

\smallskip

It is natural to compare our positive and negative results, in the complexity
language.  In particular, Theorem~\ref{t:main-negative} shows that \ts $\EBULC_k \not\in \coNP$,
for all \ts $k\ge 1$ (unless $\PH$ collapses).  In other words, it is \emph{very unlikely}
that there is a \emph{witness} \ts for \eqref{eq:SY} being strict that can be verified
in polynomial time.  This is in sharp contrast with the case \ts $k=0$~:

\smallskip

 \begin{cor}\label{cor:k=0}
Let \ts $\Mf$ \ts be a matroid given by a succinct presentation.
Then$:$
$$
\EBULC_0 \. \in \ts \coNP \ts.	
$$
 \end{cor}

\smallskip

Here by \defn{succinct} \ts we mean a presentation of a matroid
with an oracle which computes the rank function (of a subset of
the ground set) in polynomial time, see e.g.\ \cite[$\S$5.1]{KM22}.
Matroids with succinct presentation include graphical, transversal
and bicircular matroids (see e.g.\ \cite{Ox,Welsh-matroids}),
% \cite{GN06},
certain paving matroids based on Hamiltonian cycles \cite[$\S$3]{Jer},
and matroids given by their representation over fields~$\fq$ or~$\qqq$.
By a mild abuse of notation, we use \ts $\EBULC_k$ \ts to denote the
equality decision problem of the Stanley--Yan inequality \eqref{eq:SY} for general matroids
given by a succinct presentation.

Corollary~\ref{cor:k=0} follows from the
explicit description of the equality cases given in Theorem~\ref{t:main-positive}
and Proposition~\ref{p:vanish}.  See also Section~\ref{s:ex} for several examples,
and $\S$\ref{ss:finrem-NP} for further discussion of computational hardness
of \ts $\EBULC_0.$

\smallskip

The following theorem of Rado~\cite{Rado} gives the nonvanishing condition for $\aP_{\bSc}(\Mf,R,a)$ for all $k \geq 0$.

\smallskip

\begin{thm}[{\rm \defn{nonvanishing conditions for all $k\ge 0\ts$}}{},~\cite{Rado}]
\label{t:vanish}
Let \. $\Mf=(X,\cI)$ \. be a matroid with a ground set \ts $X$ \ts and \ts
independent sets \ts $\cI \ssu 2^X$.  Let \ts $r=\rk(\Mf)$
\ts be the rank of~$\Mf$.  Let \ts $\bS=(S_1,\ldots,S_\ell)$ \ts be a  set partition of~$X$,
i.e.\  we have \ts
$X = \cup_i \ts S_i$ \ts and \ts $S_i \cap S_j = \emp$ \ts for all \ts $1\le i < j \le \ell$.
Finally, let \ts $\cb=(c_1,\ldots, c_\ell) \in \nn^\ell$.  Then,
there exists an independent set \ts $A \in \cI$ \ts such that
	\begin{equation*}
	 |A \cap S_i|  \ = \  c_i \ \quad \text{for all} \ \quad i \in [\ell]
	\end{equation*}
	\underline{if and only if}
$$
\rk \big(\cup_{i\in L} \ts S_i\big) \ \ge \ \sum_{i\in L} \. c_i \ \quad \text{for all} \ \quad L\subseteq [\ell]\ts, \  L \ne \emp\ts.
$$
\end{thm}

\smallskip

We refer to \cite[Thm.~11.2.2]{Ox} for a proof of Theorem~\ref{t:vanish} in modern matroid terminology.
One can think of this result as a positive counterpart to the (negative)
Theorem~\ref{t:main-negative}.  In the language of Shenfeld and van Handel
\cite{SvH-duke,SvH-acta}, the \emph{vanishing conditions} are ``trivial''
equality cases of the SY~inequality, in a sense of having a simple geometric
meaning rather than ease of the proof.

%For completeness, we present a proof of Theorem~\ref{t:vanish} in Section~\ref{s:vanish} using the
%\emph{discrete polymatroid theory}.

Note that Proposition~\ref{p:vanish} follows from Theorem~\ref{t:vanish},
by taking \ts $S_1\gets R$, \ts $S_2 \gets X \sm R$, \ts $c_1 \gets a$,
and \ts $c_2 \gets (r-a)$.
More generally, the complexity of the vanishing for all \ts $k\ge 0$ \ts
follows immediately from Theorem~\ref{t:vanish}, and is worth emphasizing:

\smallskip

 \begin{cor}\label{c:vanish-gen}
Let \ts $\Mf$ \ts be a matroid given by a succinct presentation.
Then, for all fixed~$k\ge 0$, the problem \ts $\big\{\iB_{\bSc}(\Mf,R,a)>^?0\big\}$ \ts is in~$\ts\poly$.
 \end{cor}

\smallskip

\subsection{Total equality cases} \label{ss:intro-total}
%
%\com{SH}{Move this section to after Theorem~\ref{t:main-positive} as an application of the theorem?}
Throughout this section, we let \ts $k=0$.
We begin with a simple observation that follows from Theorem~\ref{t:SY} by a well-known argument applicable to all positive log-concave sequences.
%
% We start with a simple
%observation whose proof is well-known and applies to
%all positive log-concave sequences.

\smallskip

\begin{cor}%[{\rm Stanley}{}]
\label{c:total-Stanley}
Let \. $\Mf$ \. be a matroid of rank \ts $r\geq 2$ \ts with a ground set~$X$,
and let \ts  $R \subseteq X$.  Suppose  \ts $\aP(\Mf,R,0)>0$ \ts and \ts $\aP(\Mf,R,r)>0$.
Then$:$
\begin{equation}\label{eq:Stanley-cor}
\aP(\Mf,R,1)^r \, \ge \, \aP(\Mf,R,0)^{r-1} \aP(\Mf,R,r).
\end{equation}
Moreover, the equality in~\eqref{eq:Stanley-cor} holds \. \underline{if and only if}
\. \eqref{eq:SY} is an equality for all \. $1\le a \le r-1$.
\end{cor}
\smallskip

For completeness, we include a short proof in~$\S$\ref{ss:total-cor}.
The equality condition in Corollary~1.8, which we refer to as \defn{total equality}, admits the following explicit structural characterization (stated as condition~{\small $(iv)$} in the following theorem) for \emph{regular} matroids.

%This motivates the following result that is more surprising
%than it may seem at first:

\smallskip

\begin{thm}[{\rm \defn{total equality conditions}, \cite{Sta-AF} and \cite{Yan23}}{}]
\label{t:total-SY}
Let \. $\Mf$ \. be a loopless regular matroid of rank \ts $r\geq 2$ \ts with a
ground set~$X$, and let \ts  $R \subseteq X$.  Suppose that \ts $\aP(\Mf,R,0)>0$ \ts
and \ts $\aP(\Mf,R,r)>0$.  Then \. \underline{the following are equivalent}~$:$

\smallskip

% \nin
{\small $(i)$} \. $\aP(\Mf,R,1)^r = \aP(\Mf,R,0)^{r-1} \aP(\Mf,R,r)$,

\smallskip

% \nin
{\small $(ii)$} \. $\aP(\Mf,R,a)^2 = \aP(\Mf,R,a+1) \ts \aP(\Mf,R,a-1)$ \, for all \, $a \in \{1, \ldots, r-1\}$,

\smallskip

% \nin
{\small $(iii)$} \. $\aP(\Mf,R,a)^2 = \aP(\Mf,R,a+1) \ts \aP(\Mf,R,a-1)$ \, for some \, $a \in \{1, \ldots, r-1\}$,

% \nin
{\small $(iv)$} \. There exists $\as >0$ such that  $|\Par_{\Mf}(x)  \cap R| \. = \.  \as |\Par_{\Mf}(x)  \cap (X-R)|$ \,
for all \. $x \in X$\..
\end{thm}

\smallskip

\begin{conj}[{\rm  \cite[Conj.~3.40]{Yan23}}{}] \label{conj:Yan}
The conclusion of Theorem~\ref{t:total-SY} holds for all loopless matroids.
\end{conj}

\smallskip

The equivalence \. {\small $(i)$} \ts  $\Leftrightarrow$ \ts {\small $(ii)$} \.
is the second part of Corollary~\ref{c:total-Stanley} and holds for all matroids.
The implication \. {\small $(ii)$} \ts  $\Rightarrow$ \ts {\small $(iii)$} \. is trivial.
The equivalence \. {\small $(i)$} \ts  $\Leftrightarrow$ \ts {\small $(iv)$} \.
was proved by Stanley for regular matroids \cite[Thm~2.8]{Sta-AF} (see also \cite[Thm~3.34]{Yan23}).
Similarly, the implication \. {\small $(iii)$} \ts  $\Rightarrow$ \ts {\small $(ii)$} \.
was proved in \cite[Lem~3.39]{Yan23} for regular matroids.
This is the most striking part of the theorem:  the apparently much weaker
condition of equality in \eqref{eq:SY} for a \emph{single value} of $a$ already
implies equality for \emph{every} $a\in\{1,\ldots,r-1\}$.
The implication \. {\small $(iv)$} \ts  $\Rightarrow$ \ts {\small $(ii)$} \. was proved
in \cite[Thm~3.41]{Yan23} for all matroids.
% Turns out, of the remaining implications
% \. {\small $(ii)$} \ts  $\Rightarrow$ \ts {\small $(iv)$} \. and
% \. {\small $(iii)$} \ts  $\Rightarrow$ \ts {\small $(ii)$}, the former is true
% and latter is false.
The following result completely resolves the remaining implications of Yan's
Conjecture~\ref{conj:Yan}.

\smallskip

\begin{thm}\label{t:Yan-conj}
In the  notation of Theorem~\ref{t:total-SY}, we have:

\smallskip
$(1)$ \ \ {\small $(i)$} \ts  $\Leftrightarrow$ \ts {\small $(ii)$} \ts  $\Leftrightarrow$
\ts {\small $(iv)$} \. for all loopless matroids, \ts and

\smallskip
$(2)$ \ \ there exists a loopless binary matroid \ts $\Mf$ \ts s.t.\ \. {\small $(iii)$} \ts  holds but not \ts {\small $(ii)$}.
\end{thm}

\smallskip

The theorem is another example of the phenomenon that regular matroids satisfy certain
matroid inequalities that general binary matroids do not (see e.g.\ \cite{HSW22} and $\S$\ref{ss:finrem-NP}).
The proof of Theorem~\ref{t:Yan-conj} is given in Section~\ref{s:total}, and is
based on Theorem~\ref{t:main-positive}.
The example in Theorem~\ref{t:Yan-conj} part $(2)$ can be found in \S\ref{ss:ex-combin}.

\smallskip

\subsection{Counting spanning trees}\label{ss:intro-st}
At a crucial step in the proof of Theorem~\ref{t:main-negative}, we
give bounds for the relative version of the tree counting function, see below.
This surprising obstacle occupies a substantial part of the proof
(Sections~\ref{s:graph-CF} and~\ref{s:st}).  It is also of independent
interest and closely related to the following combinatorial problem.

Let \ts $G=(V,E)$ \ts be a connected  simple graph.  Denote by \ts $\tau(G)$ \ts
the number of spanning trees in~$G$.  Sedl\'{a}\v{c}ek \cite{Sed70} considered
the smallest number of vertices $\al(N)$  of a planar graph~$G$ with exactly $N$ spanning
trees:  \ts $\tau(G)=N$.\footnote{The original problem considered general rather
than planar graphs, see~$\S$\ref{ss:hist-st}.}

\smallskip

\begin{thm}[{\rm Stong \cite[Cor.~7.3.1]{Stong}}{}] \label{t:Stong}
For all $N\ge 3$, there is a simple planar graph $G$ with
\ts $O\big((\log N)^{3/2}/(\log\log N)\big)$ \ts vertices and exactly
\ts $\tau(G)=N$ \ts spanning trees.
\end{thm}

\smallskip

Until this breakthrough, even \ts
$\al(N) = o(N)$ \ts remained out of reach, see \cite{AS13}.
As a warmup, Stong first proves this bound in \cite[Cor.~5.2.2]{Stong},
and this proof already involves a delicate number theoretic argument.
Naturally, Stong's Theorem~\ref{t:Stong} is much stronger.
The following result is a variation of Stong's theorem,
and has the advantage of having an elementary proof.

\smallskip

\begin{thm}\label{t:st}
For all \ts $N\ge 3$, there is a planar graph $G$ with
\ts $O(\log N \ts \log\log N)$ \ts edges and exactly
\ts $\tau(G)=N$ \ts spanning trees.
\end{thm}

\smallskip
Compared to Theorem~\ref{t:Stong}, note that graphs are not required to
be simple, but the upper bound in terms of edges is much sharper
(in fact, it is nearly optimal, see~$\S$\ref{ss:intro-st}).
Indeed, the planarity implies the same asymptotic bound for the number
of vertices and edges in~$G$.\footnote{Although Stong does not
explicitly mention planarity in \cite{Stong},
his construction involves only planar graphs.}
Theorem~\ref{t:st} is a byproduct of the proof of the following lemma
that is an intermediate result in the proof of
Theorem~\ref{t:main-negative}.

\smallskip

For an edge \ts $e\in E$, denote by \ts $G-e$ \ts and \ts $G/e$ \ts
the deletion of~$e$ and the contraction along~$e$.  Define the
\defn{spanning tree ratio} \ts as follows:
\begin{equation*}
\rho(G,e) \, := \, \frac{\tau(G-e)}{\tau(G/e)}\..
\end{equation*}

\begin{lemma}[{\rm spanning tree ratios lemma}{}]\label{lem:st-ratio}
Let \ts $A,B \in \nn$ \ts such that \ts $1 \le B \le A \le 2 B \le N$.
Then there is a planar graph $G$ with \ts $O\big((\log N) \ts (\log \log N)^2\big)$ \ts edges and
\ts $\rho(G,e)=A/B$ \ts spanning tree ratio.
\end{lemma}

% This result is of independent interest.
Note that the spanning tree ratios are not attainable by the tools
in~\cite{Stong}.  This is why we need a new approach to the analysis of
the spanning tree counting function, giving the proof of both
Theorem~\ref{t:st} and Lemma~\ref{lem:st-ratio} in Section~\ref{s:st}.

Our approach follows general outlines in \cite{CP-AF,CP-CF},
although technical details are largely different.
Here we use a variation on the celebrated \emph{Haj\'{o}s construction}
\cite{Haj61} (see also \cite{Urq97}), introduced in the context of
graph colorings.  Also, in place of the Yao--Knuth \cite{YK75} ``average case''
asymptotics for continued fractions used in~\cite{CP-AF},
we use more delicate ``best case'' bounds by Larcher \cite{Lar86}.

Finally, note that the lemma gives the spanning tree ratio
\ts $\rho(G,e)$ \ts in the interval \ts $[1,2]$.
In the proof of Theorem~\ref{t:main-negative},
we consider more general ratios.  We are able to avoid
extending Lemma~\ref{lem:st-ratio} by combining
combinatorial recurrences and complexity ideas.

\smallskip

\subsection{Paper structure} \label{ss:intro-structure}
In Section~\ref{s:hist}, we give an extensive historical background of many strains
leading to the two main  results (Theorems~\ref{t:main-negative}
and~\ref{t:main-positive}).  The material is much too rich to give a proper review
in one section, so we tried to highlight the results that are most relevant to
our work, leaving unmentioned many major developments.

In Section~\ref{s:def}, we give basic definitions and notation, covering both
matroid theory and computational complexity.  In a short Section~\ref{s:red},
we give a key reduction to the SY equality problem from the matroid basis
coincidence problem.  In Sections~\ref{s:graph-CF} and~\ref{s:st} we relate
spanning trees in planar graphs to continuous fractions, and prove
Theorem~\ref{t:st} and  Lemma~\ref{lem:st-ratio} in~$\S$\ref{ss:st-proof-st-thm} along the way.

Sections~\ref{sec:verify} and~\ref{s:main-proof} contain the proof of
Theorem~\ref{t:main-negative}.  Here we start with the proof of the
Verification lemma (Lemma~\ref{lem:verify}), which uses our spanning
tree results and number theoretic estimates, and prove the theorem
using a complexity theoretic argument.

In Sections~\ref{s:atlas} and~\ref{s:main-positive}, we give a proof
of Theorem~\ref{t:main-positive}.  We start with an overview of our
combinatorial atlas technology (Sections~\ref{s:atlas}).  We give a
construction of the atlas for this problem in~$\S$\ref{ss:main-atlas-construction}
and proceed to prove the theorem.  These two sections are completely
independent from the rest of the paper.

%In Section~\ref{s:vanish}, we discuss vanishing conditions and prove
%Theorem~\ref{t:vanish}.

In Section~\ref{s:ex},
 we give  examples and  counterexamples to
equality conditions of \eqref{eq:SY}.  In a short Section~\ref{s:indep},
we present the \emph{generalized Mason inequality}, a natural variation of the
Stanley--Yan inequality for the independent sets.  We conclude with final remarks and
several open problems in Section~\ref{s:finrem}, all of them in  connection
with matroid inequalities and computational complexity.

\medskip

%%%%%%%%%%%%%%%%%%%%%%%%%%%%%%%%%%%%%%%%%%%%%%%%%%%%%%%%%%%%%%%%%%%%%%%%%%%%%%%
%%%%%%%%%%%%%%%%%%%%%%%%%%%%%%%%%%%%%%%%%%%%%%%%%%%%%%%%%%%%%%%%%%%%%%%%%%%%%%%
%%%%%%%%%%%%%%%%%%%%%%%%%%%%%%%%%%%%%%%%%%%%%%%%%%%%%%%%%%%%%%%%%%%%%%%%%%%%%%%
%%%%%%%%%%%%%%%%%%%%%%%%%%%%%%%%%%%%%%%%%%%%%%%%%%%%%%%%%%%%%%%%%%%%%%%%%%%%%%%
%%%%%%%%%%%%%%%%%%%%%%%%%%%%%%%%%%%%%%%%%%%%%%%%%%%%%%%%%%%%%%%%%%%%%%%%%%%%%%%
%%%%%%%%%%%%%%%%%%%%%%%%%%%%%%%%%%%%%%%%%%%%%%%%%%%%%%%%%%%%%%%%%%%%%%%%%%%%%%%
%%%%%%%%%%%%%%%%%%%%%%%%%%%%%%%%%%%%%%%%%%%%%%%%%%%%%%%%%%%%%%%%%%%%%%%%%%%%%%%

\section{Background}\label{s:hist}

\subsection{Log-concave inequalities}\label{ss:hist-log}
Log-concavity is a classical analytic property going back
to Maclaurin (1929) and Newton (1732).   Log-concavity is closely related
to negative correlations, which also has a long history going back to Rayleigh
and Kirchhoff, see e.g.\ \cite{BBL09}.
Log-concave inequalities for matroids
and their generalizations (morphisms of matroids, antimatroids, greedoids)
is an emerging area in its own right, see \cite{CP,Yan23} for detailed overview.

Stanley was a pioneer in the area of unimodal and log-concave inequalities in
combinatorics, as he introduced both algebraic and geometric techniques
\cite{Sta-Log}, see also \cite{Bra15,Bre89}.
In \cite{Sta-AF}, he gave two applications of the \emph{Alexandrov--Fenchel} (AF)
\emph{inequality} for mixed volumes of convex bodies, to log-concavity of combinatorial
sequences.  One is the \eqref{eq:SY} for regular matroids.
The other is \emph{Stanley's poset inequality} for the number of linear extensions
\cite[Thm~3.1]{Sta-AF} that is extremely well studied in recent years, see a
survey in \cite{CP23-survey}.  Among many variations, we note the \emph{Kahn--Saks
inequality} which was used to prove the first major breakthrough towards the \ts
\emph{$\frac13-\frac23$ conjecture} \ts \cite{KS84}.

Formally, let \ts $P=(X,\prec)$ \ts be a poset with \ts $|X|=n$ \ts elements.
A \defn{linear extension} of $P$ is an order-preserving bijection \ts
$f: X \to [n]$. Denote by \ts $\Ec(P)$ \ts the set of linear extensions of~$P$.
Fix \. $x,z_1, \ldots, z_k\in X$ \. and \. $a,c_1,\ldots,c_k\in [n]$.
Let \ts
$\Ec_{\bzc}(P,x,a)$ \ts be the set of linear extensions \ts $f\in \Ec(P)$,
s.t.\ \. $f(x)=a$ \. and \. $f(z_i)=c_i$ \. for all \. $1\le i \le k$.
\defn{Stanley's poset inequality} is the log-concavity of numbers \ts
$\aN_{\bzc}(P, x,a):=|\Ec_{\bzc}(P, x,a)|$~:
\begin{equation*}\label{eq:Sta} \tag{Sta}
\aN_{\bzc}(P, x,a)^2 \, \ge \, \aN_{\bzc}(P, x,a+1) \.\cdot \.  \aN_{\bzc}(P, x,a-1).
\end{equation*}

These Stanley's inequalities \eqref{eq:SY} and \eqref{eq:Sta} have
superficial similarities as
they were obtained in the same manner, via construction of combinatorial
polytopes whose volumes and mixed volumes have a combinatorial interpretation.
For regular matroids, Stanley used \emph{zonotopes} spanned by the vectors
of a unimodular representation, while for posets he used \emph{order polytopes}
\cite{Sta-two}.  Partly motivated by Stanley's paper, Schneider~\cite{Schn-zono}
gives equality conditions for the AF~inequality in case of zonotopes.
In principle, one should be able to derive Theorem~\ref{t:main-positive}
in the case of regular matroids from Schneider's result as well.

\smallskip
\nin
\emph{A word of caution}: \ts Although it may seem that inequalities \eqref{eq:Sta}
and \eqref{eq:SY} are both consequences of the AF~inequality, and that this paper
and \cite{CP-AF} cover the same or similar ground, in fact the opposite is true.
While \eqref{eq:Sta} is a direct consequence of the AF~inequality, only the
\emph{computationally easy} \ts part of the
\eqref{eq:SY} follows from the AF~inequality.  It took Lorentzian polynomials
to prove the \emph{computationally hard} \ts part of~\eqref{eq:SY}.
See Proposition~\ref{p:SY-P} for the formal statement.

\smallskip

In \cite{Sta-AF}, Stanley asked for equality conditions for both matroid
and poset inequalities that he studied.  He noted that the AF~inequality
has equality conditions known only in a few special cases.
He used one such known special case (dating back to Alexandrov), to
describe equality cases of his matroid log-concave inequality for
regular matroids (Theorem~\ref{t:total-SY}).  The equality conditions
for \eqref{eq:Sta} are now largely understood, see below.

Stanley's inequality \eqref{eq:SY} led to many subsequent developments.
Notably, Godsil \cite{God84} resolved Stanley's question to show that the generating
polynomial
$$
\sum_a \. \aP_{\bSc}(\Mf,R,a) \ts t^a
$$
has only real nonpositive roots (this easily implies log-concavity).
Choe and Wagner \cite{CW06} proved that $\{\aP_{\bSc}(\Mf,R,a)\. : \. 0< a <r \}$ \.
is log-concave for a larger family of matroids with the \emph{half-plane property} (HPP),
see also \cite[$\S$9.1]{Bra15}.

In a remarkable series of papers, Huh and coauthors developed a
highly technical algebraic approach to log-concave inequalities for various
classes of matroids, see an overview in \cite{Huh,Kalai}.
% The idea was to establish Hodge--Riemann
% bilinear relations and use the hard Lefschetz property for cohomology
% rings of algebraic varieties.
Most famously, Adiprasito, Huh and Katz \cite{AHK}, proved a number of
log-concave inequalities for \emph{general matroids}, some of which were
conjectured many decades earlier.  These results established log-concavity
for the number of independent sets of a matroid according to the size
(\emph{Mason--Welsh conjecture} implied by the \eqref{eq:SY} inequality, see below),
and of the coefficients of the characteristic polynomial
(\emph{Heron--Rota--Welsh conjecture}).
After that much progress followed, eventually leading to the proof
of a host of other matroid inequalities.

\smallskip

\subsection{Lorentzian polynomials} \label{ss:hist-LP}
\emph{Lorentzian polynomials} were introduced by Br\"and\'en and Huh \cite{BH},
and independently by Anari, Oveis Gharan and Vinzant \cite{ALOV-FOCS18}.
%\com{SH}{I removed Liu because he is not a co-author in AOV18}.
This approach led to a substantial extension of earlier algebraic and analytic notions,
as well as a major simplification of the earlier proofs.  Specifically, they
showed that the homogeneous \emph{multivariate Tutte polynomial} \ts
of a matroid is a Lorentzian polynomial (see \cite[Thm~4.1]{ALOV-III}
and \cite[Thm~4.10]{BH}). This implied the ultra-log-concave inequality
conjectured by Mason, i.e.\ the strongest of the Mason's conjectures.

Formally, let \ts $\aI(\Mf,a)$ \ts denotes the number of independent
sets in matroid~$\Mf$ of size~$a$.  Mason's weakest conjecture
(the Mason--Welsh conjecture mentioned above) is the log-concave inequality
\begin{equation*}\label{eq:Mason-1}\tag{M1}
\aI(\Mf,a)^2 \ \ge \ \aI(\Mf,a+1) \. \aI(\Mf,a-1) \quad \text{for all} \quad 1\le a \le (r-1).
\end{equation*}
Similarly, Mason's strongest conjecture (we skip the intermediate one), is the
ultra-log-concave inequality
\begin{equation*}\label{eq:Mason-2}\tag{M2}
\aI(\Mf,a)^2 \ \ge \  \left(1+\tfrac{1}{a}\right) \.
\left(1+\tfrac{1}{n-a}\right) \. \aI(\Mf,a+1) \. \aI(\Mf,a-1) \quad \text{for all} \quad 1\le a \le (r-1),
\end{equation*}
where \ts $n=|X|$ \ts is the size of the
ground set (see Section~\ref{s:indep}).

Most recently, Yan \cite{Yan23} used Lorentzian polynomials to extend Stanley's result
from regular to general matroids (Theorem~\ref{t:SY}).
The resulting
Stanley--Yan inequality \eqref{eq:SY} is one of the most general matroid results
proved by a direct application of Lorentzian polynomials, and it easily implies \eqref{eq:Mason-1} (we include this argument in \S\ref{s:indep} for the reader's convenience).

\smallskip

\subsection{Later developments}\label{ss:hist-atlas}
Recently, the authors introduced a linear algebra based
\emph{combinatorial atlas} technology in \cite{CP}, which includes
Lorentzian polynomials as a special case \cite[$\S$5]{CP-intro}.
The authors proved equality conditions and various extensions of both Mason's
ultra-log-concave inequality (for the number of independent sets of matroids),
and for Stanley's poset inequality \eqref{eq:Sta}.
Most recently, the authors used combinatorial atlases to establish
correlation inequalities for the numbers of linear extensions  \cite{CP-corr}.
These results parallel earlier correlation inequalities by Huh, Schr\"oter
and Wang \cite{HSW22}.

In a separate development, Br\"and\'en and Leake introduced
\emph{Lorentzian polynomials on cones} \cite{BL23}.  They were
able to give an elementary proof of the Heron--Rota--Welsh conjecture.
Note that both combinatorial atlas and this new technology give
new proofs of the Alexandrov--Fenchel inequality, see \cite[$\S$6]{CP-intro} and
\cite[$\S$6]{BL23}. This is a central and most general inequality in convex
geometry, with many proofs none of which are truly simple, see e.g.\ \cite[$\S$20]{BZ-book}.

Shenfeld and van Handel \cite{SvH-acta} undertook a major study of equality
cases of the AF~inequality.  They obtained a (very technical)
geometric characterization in the case of convex polytopes, making a progress
on a long-standing open problem in convex geometry, see \cite{Schn1}.
They gave a complete description of equality cases for \eqref{eq:Sta}
in the $k=0$ case and showed  that the equality decision problem is in~$\poly$.

The $k=0$ equality cases of \eqref{eq:Sta} were rederived in~\cite{CP}
using combinatorial atlas, where the result was further extended to
\emph{weighted linear extensions}.  Shenfeld and van Handel's approach
was further extended is in~\cite{vHYZ} to the {Kahn--Saks inequality} (a diagonal
slice of the $k=1$ case), and in~\cite{CP-AF} to
the full $k=1$ case of \eqref{eq:Sta}.
For general \ts $k\ge 2$, Ma and Shenfeld \cite{MS24} gave a technical
combinatorial description of the equality cases of \eqref{eq:Sta}.
Notably, this description involves $\SP$ oracles, and therefore is  not naturally  in~$\PH$.

\smallskip

\subsection{Negative results} \label{ss:hist-CS}
In a separate development,  the authors in~\cite{CP-AF} showed that for
$k\ge 2$, the equality conditions of Stanley's poset inequality are not
in~$\PH$ unless $\PH$ collapses to a finite level.  In particular, this
implied that the equality cases of the AF inequality for $H$-polytopes
with a concise description, are also not in~$\PH$, unless $\PH$ collapses
to a finite level.

Prior to \cite{CP-AF}, there were very few results on computational
complexity of (equality cases) of combinatorial inequalities.
The approach was introduced by the second author in~\cite{Pak},
as a way to show that certain combinatorial numbers do not have a
\emph{combinatorial interpretation}.  This was formalized as counting
functions not being in $\SP$, see survey \cite{Pak-OPAC}.  Various
examples of functions not in~$\SP$ were given in \cite{IP22},
based on an assortment of complexity theoretic assumptions.

It was shown by Ikenmeyer, Panova and the second author in \cite{IPP24},
that vanishing of $S_n$ characters problem \ts $[\chi^\la(\mu)=^?0]$ \ts is \ts
$\CEP$-complete.  This implies that this problem is not in $\PH$
unless $\PH$ collapses to the  second level (ibid.).  Finally, a key
technical lemma in \cite{CP-AF} is based on the analysis of the
\emph{combinatorial coincidence problem}.  This is a family of
decision problems  introduced and studied in \cite{CP-coinc}.
They are also characterised by a collapse of~$\PH$.

\smallskip

\subsection{Spanning trees}\label{ss:hist-st}
%
% We oversimplify the history of the spanning tree counting function.
%
Sedl\'{a}\v{c}ek \cite{Sed70} and Azarija--\v{S}krekovski \cite{AS13}
considered two closely related functions \ts $\al'(N)$ \ts and \ts $\be'(N)$,
defined to be the minimal number of vertices and edges, respectively,
over \emph{all} (i.e., not necessarily planar) graphs $G$ with \ts $\tau(G)=N$ \ts
spanning trees.  For connected planar graphs, the number of edges is linear
in the number of vertices, so this distinction disappears.

In recent years, there were several applications of continued
fractions to problems in combinatorics.  Notably, Kravitz and Sah \cite{KS21}
used continuous fractions to study a similar problem for the number
\ts $|\Ec(P)|$ \ts of linear extensions of a poset, see also \cite{CP-CF}.
An earlier construction by Schiffler \cite{Sch19} which appeared
in connection with cluster algebras, related continued fractions
and perfect matchings.  We also mention a large literature on
enumeration of lattice paths via continued fractions,
see e.g.\ \cite[Ch.~5]{GJ83}.

\smallskip

\subsection{Counting complexity}\label{ss:hist-count}
The problem of counting the number of bases and more generally,
the number of independent sets of a given size, have been heavily studied
for various classes of matroids.  Even more generally, both problems
are evaluations of the Tutte polynomial, and other evaluations have
also been considered.  We refer to \cite{Welsh-knots} for both an
introduction to the subject and a detailed, though  somewhat dated, survey of
known results.

Among more recent work, let us mention $\SP$-completeness for
the number of trees (of all sizes) in a graph \cite{Jer94},
the number of bases in bicircular matroids \cite{GN06},
in balanced paving matroids \cite{Jer}, rational matroids \cite{Snook},
and most recently in binary matroids\footnote{There is a mild controversy
over priority of this result, see a short discussion in \cite[$\S$6.3]{CP-coinc}.}
\cite{KN23}.  We also note that the volumes of both order polytopes and
zonotopes are $\SP$-hard, see \cite[$\S$3]{BW} and \cite[Thm~1]{DGH}.
See~$\S$\ref{ss:finrem-NP} for further results and applications.

Finally, in a major breakthrough, Anari, Liu, Oveis Gharan and Vinzant \cite{ALOV-RW},
used Lorentzian polynomials to prove that the basis exchange random walk mixes
in polynomial time.  This gave a FPRAS for the number of bases of a matroid,
making a fast probabilistic algorithm for approximate counting of bases.
This resolved an open problem by Mihail and Vazirani (1989). Previously,
FPRAS for the number of bases was known for regular matroids \cite{FM92},
paving matroids~\cite{CW96}, and bicircular matroids~\cite{GJ21}.

% In followup papers \cite{AD20,ALOV-IV} the authors
% extended the results and further improved bounds on the mixing time.

%%%%%%%%%%%%%%%%%%%%%%%%%%%%%%%%%%%%%%%%%%%%%%%%%%%%%%%%%%%%%%%%%%%%%%%%%%%%%%%%%%%%%%%%%%%%%%
%%%%%%%%%%%%%%%%%%%%%%%%%%%%%%%%%%%%%%%%%%%%%%%%%%%%%%%%%%%%%%%%%%%%%%%%%%%%%%%%%%%%%%%%%%%%%%
%%%%%%%%%%%%%%%%%%%%%%%%%%%%%%%%%%%%%%%%%%%%%%%%%%%%%%%%%%%%%%%%%%%%%%%%%%%%%%%%%%%%%%%%%%%%%%
%%%%%%%%%%%%%%%%%%%%%%%%%%%%%%%%%%%%%%%%%%%%%%%%%%%%%%%%%%%%%%%%%%%%%%%%%%%%%%%%%%%%%%%%%%%%%%
%%%%%%%%%%%%%%%%%%%%%%%%%%%%%%%%%%%%%%%%%%%%%%%%%%%%%%%%%%%%%%%%%%%%%%%%%%%%%%%%%%%%%%%%%%%%%%
%%%%%%%%%%%%%%%%%%%%%%%%%%%%%%%%%%%%%%%%%%%%%%%%%%%%%%%%%%%%%%%%%%%%%%%%%%%%%%%%%%%%%%%%%%%%%%
%%%%%%%%%%%%%%%%%%%%%%%%%%%%%%%%%%%%%%%%%%%%%%%%%%%%%%%%%%%%%%%%%%%%%%%%%%%%%%%%%%%%%%%%%%%%%%

\medskip

\section{Notations and definitions}\label{s:def}

\subsection{Basic notation} \label{ss:def-basic}
Let \ts $\nn=\{0,1,2,\ldots\}$ \ts and \ts $[n]=\{1,\ldots,n\}$.
For a set \ts $A$ \ts and an element \ts $x \notin A$,
we write \ts $A+x:=A \cup \{x\}$.  Similarly, for an element
\ts $x \in A$, we write \ts $A-x:= A \setminus \{x\}$.

We use bold letters \ts $\ba=(a_1,a_2,\ldots)$ \ts to denote a sequence of
integers, and \ts $\bA=(A_1,\ldots,A_n)$ \ts to denote a sequence of sets.
We write \ts $\vb=(v_1,\ldots,v_d)\in \Bbbk^d$ \ts
to denote vectors over the field \ts $\Bbbk$.
Let \ts $\eb_1,\ldots, \eb_d$ \ts denote the standard basis in \ts $\Bbbk^d$,
and let \ts $\bO = (0,\ldots,0)\in \Bbbk^d$.
Let \ts $\fq$ \ts to denote the finite field with~$q$ elements.

We say that \ts $\vb\in \rr^d$  \ts is
\defnb{strictly positive} \ts if $v_i>0$ \ts for all \ts $1\le i \le d$.
For \ts $\ba=(a_1,\ldots,a_d) \in \nn^d$, denote \ts $|\ba|:= a_1+ \ldots + a_d$.
The support of  a vector \. $\vb=(v_1,\ldots, v_d) \in \rr^d$  \.  is the set of indices \ts $i \in [d]$ \ts such that \ts $v_i \neq 0$\ts.
 The support of a symmetric matrix \. $\bM=(\aM_{i,j})_{i,j \in [d]}$  \.  is the set of indices \ts $i \in [d]$ \ts such that  \ts $\aM_{ij} \neq 0$ \ts for some \ts $j \in [d]$\ts.
Let $G(\bM)$ be the graph with the vertex set $[d]$, in which distinct vertices $i,j$ are adjacent if $\aM_{i,j} >0$.
We say that $\bM$ is \defnb{irreducible} if $G(\bM)$ is connected.

\smallskip

\subsection{Matroids}\label{ss:def-matroids}
A (finite) \defn{matroid}~$\Mf$ \ts is a pair \ts $(X,\cI)$ \ts of a
\emph{ground set} \ts $X$ \ts with \ts $|X|=n$ \ts elements,
and a nonempty collection of \defn{independent sets}
\ts $\cI \subseteq 2^X$ \ts that satisfies the following:
\begin{itemize}
	%	\item \, $\emp\in \cI$\ts,
	\item (\defn{hereditary property}) \, $A\subset B$, \. $B\in \cI$ \, $\Rightarrow$ \, $A\in \cI$\ts, and
	\item (\defn{exchange property}) \,  $A, \ts B\in \cI$, \.  $|A|<|B|$ \, $\Rightarrow$ \,
	$\exists \ x \in B \setminus A$ \. s.t.\ $A+x \in \cI$\ts.
\end{itemize}
%We will allow $\Mf$ to have  loops and parallel elements in $\Mf$, i.e. $\Mf$ is not simple.
The
\defn{rank} of a matroid  is the maximal size of an  independent set, i.e.,  $\rk(\Mf) := \max_{A\in \cI} \ts |A|$.
More generally, the \defn{rank} \ts $\rk(A)$ \ts of a subset \ts $A \subseteq X$ \ts is the size of the largest independent set contained in $A$.
A \defnb{basis} of $\Mf$  is a maximal independent set of $\Mf$, or equivalently an independent set with size $\rk(\Mf)$.
We denote by \. $\cB(\Mf)$ \. the set of bases of $\Mf$, and by \. $\iB(\Mf):=|\cB(\Mf)|$ \. the number of bases of $\Mf$.

An element $x \in X$ is a \defnb{loop} if $\{x\} \notin \cI$, and is a \defnb{non-loop}
otherwise.  A matroid without loops is called \defn{loopless}.\footnote{Unless stated otherwise,
we allow matroids to have loops and parallel elements.}
We denote by \ts $\NL(\Mf)$ \ts the set of non-loops of $\Mf$.
Two non-loops $x,y \in \NL(\Mf)$ are \defnb{parallel} if $\{x,y\} \notin \cI$.
Note that the parallelship relation between non-loops is an equivalence relation~(see e.g. \cite[Prop~4.1]{CP}).
The equivalence classes of this relation are called \defnb{parallel classes}.

%By removing loop elements if necessary, we will without loss of generality assume that all matroids in this paper do not have  loops.

Given matroids \ts $\Mf=(X,\cI)$\ts, \ts $\Mf':=(X',\cI')$\ts, the \defnb{direct sum} \ts $\Mf \oplus \Mf':=(Y,\cJ)$ \ts is a matroid with ground set \ts $Y = X\sqcup X'$, and whose independent sets \ts $A \in \cJ$ \ts are  disjoint unions of  independent sets: \ts $A=I \sqcup I'$, \ts
$I \in \cI$, \ts $I' \in \cI'$.
%The \defnb{direct sum} \ts $ \Mf + \Mf'$ is the  matroid with ground set the disjoint union of $X$ and $X'$, and whose independent sets are the disjoint unions
%When $\Mf'$ is a matroid with \ts $X'=\{z\}$ \ts (i.e., matroid of size 1) and \ts $z \notin X$\ts,
%we write \ts $\Mf\cup \{z\}$ \ts

Let  \ts $x\in \NL(\Mf)$.
The \defnb{deletion} \ts $\Mf - x$ \ts  is the matroid with ground set $X$ and with independent sets
\. $\{  A \subseteq X -x \. : \.  \ts A  \ts \in \cI \}$.
The \defnb{contraction} \ts $\Mf/x$ \ts is the matroid with ground set $X$ and with independent sets
\.  $\{  A \subseteq X -x \. : \. A +x\in \cI \}$.
Note that both $\Mf/x$ and $\Mf-x$ share the same ground set as $\Mf$.
This is slightly different than the usual convention, and is adopted here
for technical reasons that will be apparent in Section~\ref{s:main-positive}.

More generally, for $B \subseteq \NL(\Mf)$,  the
 \defnb{contraction} \ts $\Mf/B$ \ts is the matroid with ground set $X$ and with independent sets
\.  $\{  A \subseteq X -B \. : \.  A \cup B \ts \in \cI \}$\..
Recall the \defnb{deletion--contraction recurrence} \ts for the number of bases of matroids:
\[  \iB(\Mf) \ = \ \iB(\Mf-x) \. + \. \iB(\Mf/x).\]

A \defn{representation} of a matroid $\Mf$ over the field~$\kK$ \ts is a map
\ts $\phi:X \to \Bbbk^d$, such that
$$A\in \cI \quad \Longleftrightarrow \quad \phi(x_1),\ldots, \phi(x_m) \ \text{are linearly
independent over~$\Bbbk$}\ts,
$$
for every subset \ts $A=\{x_1,\ldots, x_m \} \subseteq X$.
A matroid is \defnb{binary} if it has a representation
over~$\ff_2$. A matroid  is \defnb{rational} if it has a representation
over~$\qqq$.
Throughout this paper, whenever a binary matroid with ground set~$X$
is given as input to a computational problem, it is specified by a
binary matrix whose columns are indexed by~$X$, or equivalently by a
representation $\phi:X\to\ff_2^d$.

A matroid is \defn{regular} (also called \emph{unimodular}),
if it has a representation over every field~$\Bbbk$.  A representation
\ts $\phi: X \to \zz^d$ is called \defn{unimodular} \ts if \ts
$\det\big(\phi(x_1),\ldots,\phi(x_r)\big) =\pm 1$, for every basis \ts
$\{x_1,\ldots, x_r \}\in \cB(\Mf)$.  Regular matroids are known to
have a unimodular representation (see e.g.\ \cite[Lem.~2.2.21]{Ox}).

Let $G=(V,E)$ be a finite connected graph, and let $\cF$ be the set of \emph{forests}
in~$G$ (subsets $F\subseteq E$ with no cycles).  Then \ts $\Mf_G =(E,\cF)$ \ts
is a \defn{graphical matroid} corresponding to~$G$.  Bases of the
graphical matroid \ts $\Mf_G$ \ts are the spanning trees in~$G$,
so \ts $\iB(\Mf_G)=\tau(G)$.  Recall that graphical matroids are regular.

%Let \ts $d \geq 1$, and let \ts $\Bbbk:=\ff_2\ts$.

\smallskip

\subsection{Complexity} \label{ss:def-CS}
We refer to \cite{AB,Gold,Pap} for definitions and standard results
in computational complexity, and to \cite{Aar16,Wig19} for a modern
overview.
We briefly recall the notions used in this paper.

We regard every input of a complexity problem as a finite binary string, using a standard binary
encoding of the underlying mathematical objects. A \defn{decision problem} is a problem whose answer is either \emph{yes} or \emph{no}.
We identify such a problem with a language
\ts ${\sf L}\subseteq \{0,1\}^*$\ts,
where ${\sf L}$ is the set of input strings representing instances for which
the correct answer is \emph{yes}.

A decision problem belongs to \ts $\P$ \ts  if it can be solved by a
deterministic algorithm in polynomial time. A function problem belongs to
\ts $\FP$ \ts  if it can be computed by a deterministic algorithm in polynomial
time.
A decision problem $\sf L$ belongs to $\NP$ if every yes-instance has a certificate of polynomial length whose validity can be verified in polynomial time.
That is,  there exist a
polynomial $q$ and a polynomial-time algorithm $V$ such that
\[
x\in {\sf L}
\quad\Longleftrightarrow\quad
V(x,y)=1
\quad\text{for some $y$ with } |y|\leq q(|x|),
\]
where $|x|$ denotes the length of the binary encoding of $x$, and the string $y$ is a \defn{certificate}.
A decision problem belongs to $\coNP$ if the complement of its associated language belongs to $\NP$.
Whereas an $\NP$ decision problem asks whether a polynomial-time verifier
accepts at least one polynomial-length certificate, the corresponding
$\SP$ problem counts the number of certificates accepted by that
verifier.

For two complexity classes \ts {\sf R}, {\sf S}\ts,
we use the \emph{oracle notation} \ts
{\sf R}$^{\text{\sf S}}$ \ts
for the class obtained by allowing an
{\sf R}-algorithm to make oracle queries to {\sf S}, with each
oracle query counted as a single computational step.
The levels of the polynomial hierarchy are defined recursively by
\[
\Sigmap_0:=\P,
\qquad
\Sigma_{m+1}^{\mathrm p}:=\NP^{\Sigmap_m}
\quad\text{for }m\geq 0.
\]
In particular, $\Sigmap_1=\NP$. The
\emph{polynomial hierarchy} is
\[
\PH:=\bigcup_{m\geq 0}\Sigmap_m.
\]
We say that $\PH$ \emph{collapses to a finite level} if
\. $\PH=\Sigmap_m$ \.  for some fixed $m$.

We write
\ts $\PSPACE$ \ts to denote the class of decision problems that can be solved by a
deterministic algorithm using polynomially bounded space.
For a problem \ts ${\sf A} \in \PSPACE$\ts, the \emph{polynomial closure}
\ts $\langle {\sf A}\rangle$ \ts  is the class of all problems that can be solved in
polynomial time using an oracle for ${\sf A}$.

%We assume that the reader is familiar with basic notions and results in
%computational complexity and only recall a few definitions.  We use standard
%complexity classes: \. $\P$, \. $\FP$, \. $\NP$,\. $\coNP$, \. $\SP$,
%\. $\Sigmap_m$, \. $\PH$ \. and \. $\PSPACE$.
%
%
%
%
%
% We use the \emph{oracle notation} \ts
%{\sf R}$^{\text{\sf S}}$ \ts for two complexity classes \ts {\sf R}, {\sf S} $\subseteq \PH$,
%and the polynomial closure \ts $\<${\sc A}$\>$ for a problem \ts {\sc A} $\in \PSPACE$.

% We will also use less common classes \.
% $$
% \GapP:= \{f-g \mid f,g\in \SP\} \quad \text{and} \quad
% \CEP:=\{f(x)=^?g(y) \mid f,g\in \SP\}.
% $$
% Note that \ts $\coNP \subseteq \CEP$.

% We also assume that the reader is familiar with standard decision and
% counting problems: \ts {\sc 3SAT}, \ts {\sc \#3SAT} \ts and
% \ts {\sc PERMANENT}.  Denote by \ts {\sc \#LE} \ts the problem of
% computing the number \ts $e(P)$ \ts of linear extensions.
% Occasionally, we conflate counting functions $f$ and the problems
% of computing~$f$.  We hope this does not lead to a confusion.
%

The notation \. $\{a =^? b\}$ \. is used to denote the
decision problem whether \ts $a=b$.
For a counting function \ts $f\in \SP$,
the \defn{coincidence problem} \ts is defined as \.
$\big\{\ts f(x) \. = ^? \ts f(y) \ts \big\}$.
Note the difference with the \defn{equality verification problem} \.
$\big\{ \ts f(x) \. =^? \. g(x) \big\}$.
%
% Clearly, we have both \ts $\text{\sc E}_{f-g}\in \CEP$ \. and \. $\text{\sc C}_f \in \CEP$.
% Note also that \.$\text{\sc C}_\text{\#3SAT}$ \ts
% is both \ts $\CEP$-complete \ts and \ts $\coNP$-hard.
% The distinction between \emph{binary} \ts and \emph{unary} \ts presentation
% will also be important.  We refer to \cite{GJ78} and \cite[$\S$4.2]{GJ79}
% for the corresponding notions of $\NP$-completeness and \emph{strong} \ts $\NP$-completeness.
Unless stated otherwise, we use \defn{reduction} \ts to mean
a polynomial Turing reduction.

% \newpage

%%%%%%%%%%%%%%%%%%%%%%%%%%%%%%%%%%%%%%%%%%%%%%%%%%%%%%%%%%%%%%%%%%%%%%%%%%%%%%%%%%%%%%%%%%%%%%
%%%%%%%%%%%%%%%%%%%%%%%%%%%%%%%%%%%%%%%%%%%%%%%%%%%%%%%%%%%%%%%%%%%%%%%%%%%%%%%%%%%%%%%%%%%%%%
%%%%%%%%%%%%%%%%%%%%%%%%%%%%%%%%%%%%%%%%%%%%%%%%%%%%%%%%%%%%%%%%%%%%%%%%%%%%%%%%%%%%%%%%%%%%%%
%%%%%%%%%%%%%%%%%%%%%%%%%%%%%%%%%%%%%%%%%%%%%%%%%%%%%%%%%%%%%%%%%%%%%%%%%%%%%%%%%%%%%%%%%%%%%%
%%%%%%%%%%%%%%%%%%%%%%%%%%%%%%%%%%%%%%%%%%%%%%%%%%%%%%%%%%%%%%%%%%%%%%%%%%%%%%%%%%%%%%%%%%%%%%
%%%%%%%%%%%%%%%%%%%%%%%%%%%%%%%%%%%%%%%%%%%%%%%%%%%%%%%%%%%%%%%%%%%%%%%%%%%%%%%%%%%%%%%%%%%%%%
%%%%%%%%%%%%%%%%%%%%%%%%%%%%%%%%%%%%%%%%%%%%%%%%%%%%%%%%%%%%%%%%%%%%%%%%%%%%%%%%%%%%%%%%%%%%%%

\medskip

\section{Reduction from coincidences}\label{s:red}

In this section, we introduce a coincidence problem for matroid basis
ratios and prove, through two intermediate reductions, that it admits a
polynomial-time Turing reduction to \ts $\EBULC_{1}$.
Throughout Sections~\ref{s:red}--\ref{s:main-proof}, we assume that all
matroids are binary and are given by their binary representations.
Although some of the arguments in these sections may not intrinsically
require the matroid to be binary, we adopt this convention for consistency, since these
sections are devoted to proving Theorem~\ref{t:main-negative}.

\subsection{Setup} \label{ss:red-setup}
In this section we present and prove the main technical lemma to be used in the proof of Theorem~\ref{t:main-negative}.

Let \ts $\Mf=(X,\cI)$ \ts be a binary matroid, let \ts $x\in X$ \ts be a non-loop: \ts $x\in \NL(\Mf)$.
Define the \defn{basis ratio}
$$
\rho(\Mf,x) \, := \, \frac{\iB(\Mf-x)}{\iB(\Mf/x)}\..
$$
Note that the spanning-tree ratio $\rho(G,e)$ introduced in Section~\ref{ss:intro-st} is a special case of $\rho(\Mf,x)$, obtained by taking $\Mf$ to be the graphic matroid of $G$ and setting $x=e$.
Denote by  \. $\BBM$ \. the problem of computing the number of bases \ts $\iB(\Mf)$ \ts in~$\Mf$.
Similarly, denote by \. $\NDCR$ \. the problem of computing the basis ratio \ts $\rho(\Mf,x)$.

Let \ts $\Mf=(X,\cI)$ \ts and \ts $\Nf=(Y,\cI')$ \ts be binary matroids,
let \ts $x\in X$ \ts and \ts $y\in Y$ \ts be non-loop elements: \ts $x\in \NL(\Mf)$, \ts
$y\in \NL(\Nf)$.  Consider the following decision problem:
\[
\CDCR \ := \  \big\{\rho(\Mf,x)  \, =^? \,  \rho(\Nf,y) \big\},
\]
where $\Mf,\Nf,x,y$ are inputs to the problem.
The following is the main technical lemma in the proof of Theorem~\ref{t:main-negative}.

\smallskip

\begin{lemma}% [{\rm main lemma}{}]
\label{lem:CDCR-to-EBULC}
	$\CDCR$ \. reduces to \. $\EBULC_{1}$\ts.
\end{lemma}

\smallskip

The lemma follows from two parsimonious reductions presented below.

\smallskip

\subsection{Deletion-contraction coincidences}
Let \ts $\Mf=(X,\cI)$ \ts be a binary matroid, and let \ts $x,y\in X$ \ts be non-parallel
and non-loop elements. Consider the following decision problem:
\[ \CDC \ := \ \big\{\iB(\Mf/x \ts - \ts y) \, =^?  \, \iB(\Mf/y  \ts - \ts x) \big\}\.,
\]
where $\Mf,x,y$ are inputs to the problem.
%\smallskip

\begin{lemma}\label{lem:CDC-to-EBULC}
	 $\CDC$ \. reduces to \. $\EBULC_{1}\ts$.
\end{lemma}

%\smallskip

\begin{proof}
%	Let \ts $\Mf=(X,\cI), \ts x,y$ \ts be the input of  \ts $\CDC$\ts.
%	We can assume that $x$ and $y$ are distinct, as otherwise the decision problem
%	\ts $\CDC$ \ts is trivial.
Let \. $\phi : X\to \ff_2^d$ \. be the (given) binary representation of $\Mf$.
Consider the matroid \ts $\Mf'=(X',\cI')$ \ts  defined by its binary representation
\. $\phi':X' \to \mathbb{F}_2^{d+1}\ts$,
where \. $X':=X \cup \{u,v\}$ and  \ts $u,v$ \ts are two distinct new elements, and
	\[ \phi'(z) \ := \
	\begin{cases}
	(\phi(z),0 ) & \text{ for } \ z \in X,\\
	\quad \,\. (\bO,1) & \text{ for }  \ z \in \{u,v\}.
	\end{cases} \]
That is, we append a zero to the vector representation of all \ts $z \in X$,
and we represent $u,v$ by the basis vector~$\eb_{d+1}\ts$.

Let \ts $r:=\rk(\Mf)$ \ts be the rank of~$\Mf$, and let \ts $n:=|X|$ \ts be the number of elements.
Note that $\Mf'$ is a matroid of rank \ts $r+1$ \ts and with \ts $n+2$ \ts elements.
Note also that the bases of $\Mf'$ are of the form \ts $A+u$ \ts and \ts $A+v$, where \ts $A\in \cB(\Mf)$ \ts
is a basis of $\Mf$.

To define the reduction in the lemma, let
\begin{align*}
&R \, := \ \{x,u\}, \qquad  a \. := \ 1, \\
&S  \, := \ X-\{x,y\}, \qquad  c \. := \. r -1.
\end{align*}
%In words, the constraints select exactly one element from $\{x,u\}$, exactly $r−1$ elements from $X-\{x,y\}$, and leave the remaining choice to $\{y,v\}$.
It then follows that
\[ 		\iB_{Sc}(\Mf',R,a+1)  \, = \,   \iB(\Mf/x-y). \\ \]
Indeed, $\iB_{S c}(\Mf',R,a+1)$ are subsets \. $A\subseteq X\cup\{u,v\}$ \. that are of the form
\[ A \cap R \ = \ \{x,u\}, \quad   \quad A \cap \{y,v\}  \ = \ \varnothing, \quad A-\{u\} \. \in \. \cB(\Mf).\]
%and that \ts $A-\{u\}$ \ts is a basis of $\Mf$.
It then follows that \ts $A-\{x,u\}$ \ts is a basis of  \ts $\Mf/x-y$\ts,
and that this correspondence is a bijection, proving our claim.  By the argument as above, we have:
%	Let \ts $R:=\{x,u\}$ \ts and \ts $S:=X-\{x,y,u,v\}$\ts, and let $c=\rk_{\Mf}-1$.
%	It then follows that
	\begin{alignat*}{2}
		& \iB_{Sc}(\Mf',R,a) \, &&= \,  \iB(\Mf/x-y) \. + \. \iB(\Mf/y-x),\\
	& \iB_{Sc}(\Mf',R,a-1) \, &&= \,  \iB(\Mf/y-x).
	\end{alignat*}

We have:
	 \begin{align*}
	 	&\aP_{Sc}(\Mf',R,a)^2 \. - \. \aP_{Sc}(\Mf',R,a+1) \. \aP_{Sc}(\Mf',R,a-1) \\
	 	 & \quad = \ \tfrac{1}{r^2(r+1)^2} \,  \big[ \iB_{Sc}(\Mf',R,a)^2 \. -\. 4 \. \iB_{Sc}(\Mf',R,a+1) \. \iB_{Sc}(\Mf',R,a-1)  \big] \\
	 	 & \quad = \ \tfrac{1}{r^2(r+1)^2} \,  \big[ \big(\iB(\Mf/x-y) \. + \. \iB(\Mf/y-x)\big)^2 \. - \. 4 \. \iB(\Mf/x-y) \. \iB(\Mf/y-x) \big] \\
	 	 	 	 & \quad = \ \tfrac{1}{r^2(r+1)^2}  \, \big( \iB(\Mf/x-y) \. - \. \iB(\Mf/y-x) \big)^2.
	 \end{align*}
	 Therefore, we have:
	 \[ \aP_{Sc}(\Mf',R,a)^2  =  \aP_{Sc}(\Mf',R,a+1) \. \aP_{Sc}(\Mf',R,a-1)  \ \  \Longleftrightarrow \ \  \iB(\Mf/x-y)  =  \iB(\Mf/y-x),  \]
	 which completes the proof of the reduction.
%	Since $\Mf$
\end{proof}

%\medskip

\smallskip

\subsection{Back to ratio coincidences}
% Main
Lemma~\ref{lem:CDCR-to-EBULC} now follows from the following reduction.

\begin{lemma}\label{lem:CDCR-to-CDC}
$\CDCR$ \. reduces to \. $\CDC$.
\end{lemma}

\begin{proof}
Let \ts $\Mf,\Nf, x,y$ \ts be the input of $\CDCR$.
Let \ts $\Mf':=\Mf \oplus \Nf$ \ts be the direct sum of
matroids \ts $\Mf$ and $\Nf$.  Note that \ts $\Mf'$ \ts is also binary.
We have:	
\begin{align*}
		\iB(\Mf'/x \ts - \ts y) \. = \.  \iB(\Mf/x) \. \iB(\Nf-y) \quad \text{and}
\quad \iB(\Mf'/y \ts - \ts x) \. = \.  \iB(\Mf-x) \. \iB(\Nf/y),
\end{align*}
as bases of the direct sum arise from pairs of bases of the summands.
This proves proves the reduction.
\end{proof}

\medskip

\section{Planar graphs and  continued fractions}\label{s:graph-CF}

In this section and the next, we construct planar graphs with prescribed
spanning-tree counts and spanning-tree ratios, while keeping the number
of edges small.  These size bounds are essential for the polynomial-time
reductions used in the proof of Theorem~\ref{t:main-negative}.

\subsection{Graph theoretic definitions}\label{ss:graph-CF-def}
Throughout this paper  $G=(V(G),E(G))$ will be a graph with vertex set $V(G)$
and edge set $E(G)$, possibly with loops and parallel edges.
We will write $V$ and $E$ when the  underlying graph $G$ is clear from the context.

For an edge \ts $e=(v,w) \in E$, the \defnb{deletion} $G-e$ is the graph obtained by deleting
the edge~$e$ from the graph, and the \defnb{contraction} $G/e$ is the graph obtained
by identifying $v$ and $w$, and removing the resulting loops.
Recall that \ts $\at(G)$ \ts  denotes  the number of spanning trees in~$G$.  Note that
\ts $\at(G)$ \ts satisfies the \defnb{deletion-contraction} recurrence for every non-loop $e \in E$:
\[ \at(G) \, = \,  \at(G-e) \. + \. \at(G/e). \]

Let \ts $G=(V,E)$ \ts be a planar graph.  For every planar embedding of~$G$, the  \defnb{dual graph} \ts
$G^\ast=(V^\ast,E)$ \ts is defined as follows.
The vertices of $G^\ast$ are the faces of the embedding of $G$ (including the unbounded face).
For every edge $e\in E$, the corresponding edge of
$G^\ast$ joins the two faces incident to $e$; if the same face is incident
to both sides of $e$, then the corresponding edge in $G^\ast$ is a loop.
Thus, the edges of $G$ and $G^\ast$ are naturally identified.
% and each edge is  incident to  faces of $G$ that are separated from each other by the edge in the planar embedding.
While the dual graph $G^\ast$ can depend on the given planar embedding of~$G$, we will
not emphasize that as our proof is constructive and the embedding will be clear
from the context.

Note that deletion and contraction for dual graphs swap their meaning.  Formally,
for an edge \ts $e \in E$ \ts that is neither a bridge nor a loop, we have:
\begin{equation}\label{eq:tree-dual}
	\at(G-e) \, = \,  \at(G^\ast/e) \quad \text{and} \quad \at(G/e) \, = \, \at(G^\ast-e).
\end{equation}
Consequently, we get for the spanning tree ratios that  \ts $\rho(G^\ast,e) = \rho(G,e)^{-1}$.

%\com{SH}{Find a citation to this result?}

\smallskip

\subsection{Continued fraction representation}\label{ss:st-CF-rep}
Given integers \. $a_0\geq 0$\., \. $a_1, \ldots, a_s \geq 1 $\., where  \ts $s \geq 0$,
the corresponding \defnb{continued fraction} \ts is defined as follows:
\[ [a_0\ts ; \ts a_1,\ldots, a_s] \ := \ a_0  \. +  \.  \cfrac{1}{a_1  \. +  \. \cfrac{1}{\ddots  \ +  \.  \frac{1}{a_s}}}\,.
\]

The integers \ts $a_i$ \ts are called \defn{quotients} \ts or \defn{partial quotients},
see e.g.\ \cite[$\S$10.1]{HardyWright}.  We refer to \cite[$\S$4.5.3]{Knuth98}
for a detailed asymptotic analysis of the quotients in connection with
the Euclidean algorithm, and further references.

\smallskip

The following result gives a connection between spanning trees and
continued fractions.  It is inspired by a similar construction for
perfect matchings given in \cite[Thm~3.2]{Sch19}.

\smallskip

\begin{thm}\label{thm:graph-cf}
	Let \. $a_0,  \ldots, a_s \ \geq 1 $\. be positive integers.
	Then there exists a  connected loopless bridgeless planar graph \ts $G=(V,E)$ \ts and an edge  \ts $e \in E$,
	such that
	\[ \frac{\at(G-e)}{\at(G/e)}
 \ = \   [a_0 \ts ; \ts a_1,\ldots,a_s]
	\]
	and \. $|E| \. = \. a_0+\ldots +a_s+1$\..
\end{thm}

\smallskip

We start with the following lemma.

\smallskip

\begin{lemma}\label{lem:graph-cf-1}
	Let $G=(V,E)$ be a connected loopless bridgeless planar graph, and let $e \in E$.
	Then there exists a connected loopless bridgeless planar graph $G'=(V',E')$ and  $e'\in E'$ such that
	\[  \frac{\at(G'-e')}{\at(G'/e')}  \ = \  1 \. + \. \frac{\at(G-e)}{\at(G/e)}
\]
	and  \. $|E'| \ts = \ts |E|+1$.
\end{lemma}

\begin{proof}
	Let $G'$ be obtained from $G$ by
	adding an edge $e'$ that is parallel to $e$.
	Note that $G'-e'$ is isomorphic to $G$, and $G'/e'$ is isomorphic to $G/e$,
	and it follows that
%	subdividing $e=\{u,v\}$ into two edges, i.e.
%	the edge $e$ is deleted, and two new edges \ts $e'=\{u,w\}$ \ts and $f'=\{v,w\}$ \ts along with a new vertex $w$ are added.
	\begin{align*}
& \at(G'/e') \, = \, \at(G/e) \quad \text{and} \quad
\at(G'-e') \, = \,  \at(G) \, = \,  \at(G-e) \. + \. \at(G/e). 	
	\end{align*}
This implies the lemma.
\end{proof}

\smallskip

\subsection{Proof of Theorem~\ref{thm:graph-cf}}
We use induction on~$s$.
For \ts $s=0$,
let $H$ be the  graph with two vertices and with $a_0+1$ parallel edges connecting the two vertices, and let $f$ be any edge of~$H$.
Note that $H-f$ is the  same graph but with  $a_0$ edges instead, while $H/f$ is the graph with one vertex and $a_0$ loops.
Thus we have
\[ \at(H-f) \, = \, a_0 \quad \text{and} \quad \at(H/f) \, = \,  1.
\]
We also have \. $|E(H)|=a_0+1$.
It then follows that
\[
 \frac{\at(H-f)}{\at(H/f)} \ = \  a_0\ts,
\]
and the claim follows by taking \ts $G \gets H$ \ts and \ts $e \gets f$.

For $s\ge 1$, by induction there exists a connected loopless
bridgeless planar graph \ts $H$ \ts and \ts $f \in E(H)$ \ts such that
\[ \frac{\at(H-f)}{\at(H/f)}  \ = \ [a_1;a_2,\ldots,a_s]\ts,
\]
and with \. $|E(H)|=a_1+\ldots + a_s+1$.
%Note that $f$ is a non-loop of $H^\ast$,  since \ts $\at(H-f)=\at(H^\ast/f) \geq 1$ \ts

Now, by applying Lemma~\ref{lem:graph-cf-1}  $a_0$ many times  to $H^\ast$, there exists a graph $G$ and a  $e \in E(G)$ such that
\[ \frac{\at(G-e)}{\at(G/e)} \ = \  a_0 + \frac{\at(H^\ast-f)}{\at(H^\ast/f)} \ = \ a_0 + \frac{\at(H/f)}{\at(H-f)}  \ = \  a_0  \. + \. \frac{1}{[a_1;a_2,\ldots,a_s]} \ = \  [a_0;a_1,\ldots,a_s],\]
and with \. $|E(G)| \. = \. a_0  +  |E(H^\ast)| \. = \.  a_0+\ldots+a_s+1$.
This completes the proof. \qed
% \end{proof}
\smallskip

\subsection{Sums of continued fractions}\label{ss:st-sums-CF}
We now extend Theorem~\ref{thm:graph-cf} to sums of two continued fractions:
\smallskip

\begin{thm}\label{thm:graph-cf-sum}
	Let \. $a_0, \ldots, a_s,   b_0,\ldots, b_t \geq 1 $ \. be positive integers.
Then there exists a  connected loopless bridgeless planar graph \ts $G=(V,E)$ \ts and an edge \ts $e \in E$,
such that
\[ \frac{\at(G-e)}{\at(G/e)} \ = \   \frac{1}{[a_0 \ts ; \ts a_1,\ldots,a_s]} \. + \. \frac{1}{[b_0 \ts ; \ts b_1,\ldots,b_t]}
\]
and \. $|E| \. = \. a_0+\ldots +a_s + b_0+\ldots +b_t+1$\..
\end{thm}

\smallskip

We start with the following lemma.

\smallskip

\begin{lemma}\label{lem:graph-sum}
	Let $G,H$ be connected loopless bridgeless planar graphs, and let \ts $e\in E(G)$\ts, \ts $f\in E(H)$\ts.
	Then there exists a connected loopless bridgeless planar graph \ts $G'$ \ts  and an edge \ts $e' \in E(G')$,
	such that
\[ \frac{\at(G'-e')}{\at(G'/e')} \ = \  \frac{\at(G-e)}{\at(G/e)} \. + \. \frac{\at(H-f)}{\at(H/f)}
\]
	and \. $|E(G')| \. = \. |E(G)|+|E(H)|-1$.
\end{lemma}

%\smallskip

\begin{proof}
Let \ts $e=(x,y)\in E(G)$ \ts and let \ts $f=(u,v)\in E(H)$.
Consider the graph
$$G' \, := \, G\oplus H \ts / \ts (x,u), \ts (y,v)
$$
obtained by taking the disjoint
union of $G$ and $H$, then identifying $e$ and $f$.
Denote by \ts $e'\in E(G')$ \ts the edge resulted from
identifying $e$ and $f$.
	
First, note that
	\begin{align}\label{eq:graph-sum-1}
		\at(G'/e') \  = \  \at(G/e) \. \at(H/f).
	\end{align}
This is because \ts $G'/e' = (G/e \oplus H/f)/(x,u)$, i.e.\ can be obtained by identifying $x$ with $u$
in the disjoint union of $G/e$ and $H/f$.

Second, note that
		\begin{align}\label{eq:graph-sum-2}
		\at(G'-e') \  = \  \at(G-e) \. \at(H/f) \. + \. \at(G/e) \. \at(H-f).
	\end{align}
Indeed, let $T$ be a spanning tree of $G'-e'$.
There are two possibilities.
First, $x$ and $y$ are connected in $T$ through a path in $G$.
Then, restricting $T$ to edges of $G$ gives us a spanning tree in $G-e$, while restricting $T$ to edges of $H$ gives us a spanning tree of $H/f$.
This bijection gives us the first term in the RHS of \eqref{eq:graph-sum-2}.

Second, suppose that $x$ and $y$ are connected in $T$ through a path in $H$.
Then, restricting~$T$ to edges of $G$ gives us a spanning tree in $G/e$, while restricting $T$ to edges of $H$ gives us a spanning tree of $H-f$.
This bijection gives us the second term in the RHS of \eqref{eq:graph-sum-2}.

The lemma now follows from dividing  \eqref{eq:graph-sum-2} by  \eqref{eq:graph-sum-1}.
\end{proof}

\smallskip

\subsection{Proof of Theorem~\ref{thm:graph-cf-sum}}
%\begin{proof}[]
	By Theorem~\ref{thm:graph-cf}, there exists connected loopless bridgeless planar graphs $G,H$ and edges  \ts $e\in E(G)$, \ts $f \in E(H)$ \ts such that
	\[ \frac{\at(G-e)}{\at(G/e)}
\ = \   [a_0 \ts ; \ts a_1,\ldots,a_s]\ts, \qquad \frac{\at(H-f)}{\at(H/f)}
\ = \   [b_0 \ts ; \ts b_1,\ldots,b_t]\ts,
\]
and with \. $|E(G)| \. = \. a_0+\ldots +a_s+1$, \. $|E(H)| \. = \. b_0+\ldots +b_t+1$.
Applying Lemma~\ref{lem:graph-sum} to \ts $(G^\ast,e)$ \ts and \ts $(H^\ast,f)$,
gives a planar graph \ts $G'$ \ts  and \ts $e' \in E(G')$,
	such that
	\begin{align*}
	 \frac{\at(G'-e')}{\at(G'/e')} \ &= \  \frac{\at(G^\ast-e)}{\at(G^\ast/e)} \. + \. \frac{\at(H^\ast-f)}{\at(H^\ast/f)} \ = \ \frac{\at(G/e)}{\at(G-e)} \. + \. \frac{\at(H/f)}{\at(H-f)} \\
	 & = \  \frac{1}{[a_0 \ts ; \ts a_1,\ldots,a_s]} \. + \. \frac{1}{[b_0 \ts ; \ts b_1,\ldots,b_t]}
	\end{align*}
	and \. $|E(G')| = |E(G^\ast)|+|E(H^\ast)|-1 =  a_0+\ldots+a_s \ts + \ts b_0+\ldots+b_t \ts + \ts 1$, as desired.  \qed
%\end{proof}

\medskip

\section{Counting spanning trees}\label{s:st}

In this section, we prove Theorem~\ref{t:st} and Lemma~\ref{lem:st-ratio}.

%\smallskip

\subsection{Proof of Theorem~\ref{t:st}}\label{ss:st-proof-st-thm}
We restate Theorem~\ref{t:st} here for the convenience of the reader.
\begin{reptheorem}{\ref{t:st}}
For all \ts $N\ge 3$, there is a planar graph $G$ with
\ts $O(\log N \ts \log\log N)$ \ts edges and exactly
\ts $\tau(G)=N$ \ts spanning trees.
\end{reptheorem}

\smallskip

For \ts $\alpha\in \qqq_{>0}$, consider the sum of the quotients of~$\alpha$~:
\[
\as(\alpha) \ := \ a_0 \. + \. \ldots \. + \. a_s \quad \text{where} \quad  \alpha \ = \ [a_0 \ts ; \ts a_1,\ldots,a_s].
\]
% Note that $\as(\alpha)$ is well defined, as this sum is invariant under different continued fraction expansions of $\alpha$,
% The earliest proof of the theorem below is due to Larcher~\cite{Lar86}, and was discovered independently in different contexts, see e.g. .
We will need the following theorem from number theory.
\smallskip

\begin{thm}[{\rm Larcher \cite[Cor.~2]{Lar86}}{}]\label{thm:Lar}
For  \ts $m\geq 9$ \ts and  $L \geq 2$, the set
\[
\Big\{ d \in [m] \  : \ \gcd(d,m)=1 \ \ \text{and} \ \  \as\left(\tfrac{d}{m}\right) \ \leq \  L \.  \tfrac{m}{\phi(m)} \. \log m \. \log \log m \Big\}
\]
contains at least \. $\left(1-\tfrac{16}{\sqrt{L}}\right) \. \phi(m)$ \. many elements, where $\phi$ is the Euler's totient function.
\end{thm}

\smallskip

First, assume that $N$ is prime and note that \ts $\phi(N)=N-1$.
By Larcher's Theorem~\ref{thm:Lar}, there exists \ts $d<N$ \ts such that
\[
\as\big(\tfrac{d}{N}\big) \ \leq  \   C \. \log N \. \log \log N \quad \text{for some} \quad C \ts >\ts 0\ts .
\]
By Theorem~\ref{thm:graph-cf} and planar duality \eqref{eq:tree-dual},
there exists a planar graph \ts $G=(V,E)$ \ts and edge \ts $e\in E$, such that
	\[ \frac{\at(G-e)}{\at(G/e)}
\ = \   \frac{N}{d} \quad \text{ and } \quad |E(G)| \ \leq  \ 1 \. + \.  C \. \log N \. \log \log N\ts.
\]
The conclusion follows by taking  \ts $(G-e)$.

\smallskip
In full generality, let \. $N \ts = \ts p_1^{b_1} \ts \cdots \ts p_\ell^{b_\ell}$ \.
be the prime factorization of~$N$.  Let \ts $G_i=(V_i,E_i)$, \ts $1\le i \le \ell$,
be the planar graphs constructed above:
$$
\tau(G_i) \. = \. p_i \quad \text{and} \quad
|E_i|  \, \le \,  C \. \log p_i \. \log \log p_i\,.
$$
Finally, let \ts $G=(V,E)$ \ts be a union of \ts $b_i$ \ts copies of \ts $G_i$ \ts
attached at vertices, so that \ts $G$ \ts is planar and connected.  Clearly,
\ts $\tau(G) = N$ \ts and
\begin{align*}
 |E| \ \leq \  \sum_{i=1}^\ell  \.   b_i \.  C \ts \log p_i \ts \log \log p_i
  \ \leq \ \left(\sum_{i=1}^\ell    b_i \. \log p_i \right) \. C\. \log \log N \   = \  C \. \big(\log N\big) \. \log \log N\ts,
\end{align*}
as desired.  \qed

\smallskip

\subsection{Number theoretic estimates}\label{ss:st-NT}
We start with the following number theoretic estimates that is  based on Larcher's Theorem~\ref{thm:Lar}.
%Note that the results in this section bear many similarities to those in \cite[$\S$8.2]{CP-AF},
%but with enough subtle differences that the proofs need to be treated differently.

%Recall that \ts $[n]=\{1,\ldots,n\}$.
%For \ts $A\in \nn_{\geq 1}$ \ts and \ts $m \in [A]$,
%consider the quotients in the (canonical) continued fraction of \ts $m/A$ \ts and their sum:
%\[\frac{m}{A} \, = \, [0 \ts ; \ts a_1(m),\ldots, a_s(m)] \quad \text{and} \quad S_A(m) \, := \, \sum_{i=1}^s \. a_i(m)\..\]
%%Note that every rational number can be represented by continued fractions in exactly two ways, and   \. $S_{A}(m)$ \. are equal for both representations.
%Also note that
%\begin{equation}\label{eq:cf-irr}
%	 S_{A}(m) \ = \  S_{A'}(m'),
%\end{equation}
%where \. $A':=\frac{A}{\gcd(A,m)}$ \. and \. $m':=\frac{m}{\gcd(A,m)}$ \. are normalized to be coprime integers.
%The following is the main result of this section.

\smallskip

\begin{prop}\label{p:NTD}
	There exists   constants \ts $C, K >0$, such that for all coprime integers \ts $A,B$ \ts which satisfy \. $C < B \leq A \leq 2B$,
	there exists a positive integer \. $m:=m(A,B)$ \. such that $m <  B$,
	and
	\[ \as\left(\tfrac{m}{A}\right) \, \leq \,  K \ts (\log A) \ts (\log \log A)^2  \quad \text{ and } \quad   \as\left(\tfrac{B-m}{A}\right) \, \leq \,  K \ts (\log A) \ts (\log \log A)^2.  \]
\end{prop}

\smallskip
%We now build toward the proof of this result.
%We need the following technical result.

Thus Proposition~\ref{p:NTD} supplies an integer $m$ for which both  $\tfrac{m}{A}$ and $\tfrac{B-m}{A}$ have small quotient-sums.

%\smallskip
%
%\begin{lemma}[{\rm Yao--Knuth \cite{YK75}}{}]\label{lem:Knuth}
%	We have:
%	\[  \frac{1}{n} \. \sum_{m \ts \in \ts [n]}  S_n(m) \ = \
%	\frac{6}{\pi^2} \. (\log n)^2 \. + \. O\big((\log n) (\log \log n)^2\big) \quad \text{as \ \  $n\to \infty$}.\]
%\end{lemma}
%
%\smallskip
%
%By the Markov inequality, it follows from Lemma~\ref{lem:Knuth} that
%\begin{equation}\label{eq:Knuth}
%	\begin{split}
%		\big|\big\{ m \in [n] \, : \,  S_n(m) \. > \. 3 \. (\log n)^2 \.  \big\}\big|
%		%   \ &\leq \ \frac{1}{6 \. (\log n)^2 \. \log \log n} \. \varphi(n) \. \frac{6}{\pi^2} \. (\log n)^2  \. (1+o(1))  \\
%		& \  \leq  \  \frac{2}{\pi^2} \, n \. (1+o(1)).
%	\end{split}
%\end{equation}
%
%
%\smallskip

\begin{proof}[Proof of Proposition~\ref{p:NTD}]
	Define
	\begin{align*}
		\vt(A,B) \, := \, \big|\big\{ \ts m \in [B] \, : \, \as\left(\tfrac{m}{A}\right) \. \leq \.  K \ts (\log A) \ts (\log \log A)^2, \ \as\left(\tfrac{B-m}{A}\right) \. \leq \.  K \ts (\log A) \ts (\log \log A)^2 \ts \big\} \big|\ts.
	\end{align*}
We will prove a stronger claim, that \. $\vt(A,B) = \Omega\big(B\big)$ \.  as \. $C\to \infty$. % This implies the result.
	
	It follows from the de Morgan's law and the union bound that
	\begin{align*}
	 \vt(A,B) \, \ge & \, B \. - \.   \big|\big\{\ts m \in [B] \, : \,   \as\left(\tfrac{m}{A}\right) \ > \  K \ts (\log A) \ts (\log \log A)^2 \ts  \big\} \big| \\
	 & \ \ \. - \.  \big|\big\{ \ts m \in [B] \, : \,   \as\left(\tfrac{B-m}{A}\right) \ > \  K \ts (\log A) \ts (\log \log A)^2  \ts  \big\} \big|.
	 \end{align*}
	On the other hand, we have
	\begin{align*}
		& \big|\{ m \in [B]  \, : \,   \as\left(\tfrac{m}{A}\right) \ > \  K \ts (\log A) \ts (\log \log A)^2  \  \} \big|  \\
& \quad\leq \ 	\big|\{ m \in [A] \, : \,  \as\left(\tfrac{m}{A}\right) \ > \  K \ts (\log A) \ts (\log \log A)^2  \  \} \big| \\
		&  \quad\quad \leq  \    	\big|\{ m \in [A] \, : \,  \as\left(\tfrac{m}{A}\right) \ > \  \tfrac{K}{2} \. \tfrac{A}{\phi(A)} \. \log A \. \log \log A  \  \} \big| \    \leq \  \ 0.2 \ts A ,
	\end{align*}
	where the second inequality is because \. $\tfrac{A}{\phi(A)} < 2 \log \log A$ \. for sufficiently large $A$\footnote{Recall that the Euler's totient function $\phi$ satisfies $\liminf_{A \to \infty} \frac{\phi(A)}{A} \log \log A \. = \. 0.561...$}, and the third inequality is because of Larcher's Theorem~\ref{thm:Lar}.
Similarly, we have
	\begin{align*}
		\big|\{ m \in [B]  \, : \,      \as\left(\tfrac{B-m}{A}\right) \ > \  K \. \log A \. (\log \log A)^2 \  \} \big| \ \leq \ 0.2 \ts  A.
	\end{align*}
	
	Combining these inequalities, we get
	\begin{align*}
		\vt(A,B) \ \geq \ B  \.  - \. 0.4 \. A  \  \geq  \ 0.2 \ts B,
	\end{align*}
	and the result follows.
\end{proof}

\smallskip

\subsection{Proof of Lemma~\ref{lem:st-ratio}}\label{ss:st-proof}
We restate Lemma~\ref{lem:st-ratio} here for the convenience of the reader.

\smallskip

\begin{replemma}{\ref{lem:st-ratio}}
Let \ts $A,B \in \nn$ \ts such that \ts $1 \le B \le A \le 2 B \le N$.
Then there is a planar graph $G$ with \ts $O\big((\log N) \ts (\log \log N)^2\big)$ \ts edges and
\ts $\rho(G,e)=A/B$\ts.
\end{replemma}

\smallskip

Let $C>0$ be the absolute constant from Proposition~\ref{p:NTD}.  We may assume
without loss of generality that $B>C$.  Indeed, if $B\leq C$, then
$A\leq 2B\leq 2C$,
and  there are only finitely many possible reduced fractions $A/B$.
Then,
applying Theorem~\ref{thm:graph-cf} directly to the continued fraction
expansion of $A/B$ gives a planar graph realizing
this ratio.  By increasing the implicit constant in the statement of
Lemma~\ref{lem:st-ratio}, these cases are covered by the claimed bound.

We now assume that $B>C$.
It follows from Proposition~\ref{p:NTD}, that there exists fixed \ts $K>0$ \ts and  an integer \ts $m < B$,  such that
	\[ \as\left(\tfrac{m}{A}\right) \ \leq \  K \ts (\log A) \ts (\log \log A)^2  \quad \text{ and } \quad   \as\left(\tfrac{B-m}{A}\right) \ \leq \  K \ts (\log A) \ts (\log \log A)^2.  \]	
%\[
%S_{A}(m) \, \leq \, 3 \. (\log A)^2 \, \leq \, 3 \. (\log N)^2 \quad \text{ and } \quad  S_{A}(B-m) \, \leq \, 3 \. (\log A )^2 \, \leq \, 3 \. (\log N)^2.
%\]
Let \ts $[a_0,\ldots,a_s]$ \ts and \ts $[b_0,\ldots, b_t]$ \ts be a continued fraction representation of \ts $A/m$ \ts and \ts $A/(B-m)$, respectively.
By Theorem~\ref{thm:graph-cf-sum}, there exists a connected loopless bridgeless planar graph $G$ and an edge $e \in E(G)$, such that
\[ \frac{\at(G-e)}{\at(G/e)} \ = \   \frac{1}{[a_0 \ts ; \ts a_1,\ldots,a_s]} \. + \.
\frac{1}{[b_0 \ts ; \ts b_1,\ldots,b_t]} \ = \  \frac{B}{A}
\]
and
\[ |E(G)| \ = \as\left(\tfrac{m}{A}\right)  \. + \. \as\left(\tfrac{B-m}{A}\right) +1 \ \leq  \  2 \ts K \ts (\log A) \ts (\log \log A)^2  \ts + \ts 1 \ = \  O\big((\log N) \ts (\log \log N)^2\big).
\]
Taking the dual graph \ts $G^\ast$ \ts gives the result.  \qed

\medskip

\section{Verification of matroid basis ratios}\label{sec:verify}
In this section, we continue toward the proof of Theorem~\ref{t:main-negative} by
establishing a verification lemma for matroid basis ratios; the proof of
the theorem will be completed in Section~\ref{s:main-proof}.

\subsection{Setup}
Let \ts $\Mf=(X,\cI)$ \ts be a binary matroid, let \ts $x\in \NL(\Mf)$,
and let \ts $A, B \in \nn$, where \ts $B>0$.
 Consider the following decision problem:
\[
\VDCR \ := \ \left\{ \rho(\Mf,x) \, =^? \, \tfrac{A}{B}  \right\},
\]
where $\Mf,x,A,B$ are inputs to the problem, and
$\rho(\Mf,x)$ is the ratio $\tfrac{\iB(\Mf-x)}{\iB(\Mf/x)}$ (defined in Section~\ref{ss:red-setup}).

\smallskip

\begin{lemma}[{\rm verification lemma}{}]
	\label{lem:verify}
	\.  $\NP^{\<\VDCR\>} \. \subseteq \. \NP^{\<\CDCR\>}.$
\end{lemma}

\smallskip

The proof is broadly similar to that in \cite{CP-AF}, but with
major technical differences.
We start with the following simple result.

\smallskip

\begin{lemma}\label{lem:r-bound}
Let \ts $\Mf=(X,\cI)$ \ts be a matroid on \ts $n=|X|$ \ts elements,
and let \ts $x\in X$ \ts be a non-loop  of~$\Mf$.
	Then \. $\rho(\Mf,x) \le n$.
\end{lemma}

\begin{proof}
To prove that
\[   \rho(\Mf,x) \, = \, \frac{\iB(\Mf-x)}{\iB(\Mf/x)} \, \leq \, n,
\]
we construct an explicit injection  \. $\ga : \cB(\Mf-x) \to   \cB(\Mf/x) \times X$.
Fix a basis \ts $A \in \cB(\Mf)$ \ts such that \ts $x \in A$.  Such basis~$A$
exists since $x$ is not a loop.
By the  symmetric basis exchange property, for every basis \ts $B\in \cB(\Mf)$ \ts
such that \ts $x \notin B$, there exists \ts $y \in B$, such that
\. $B':=B-y + x$ \. is a basis of~$\Mf$. Now take the lex-smallest of such~$y$, and
define \. $\ga(B) := (B'-x,y)$.  Note that map~$\ga$ is an injection because $B$
can be recovered from $(B',y)$ by taking \. $B=B'-x+y$. This completes
the proof.  \end{proof}

\smallskip

\subsection{Proof of Lemma~\ref{lem:verify}}
We now simulate \. $\VDCR$ \. with an oracle for \. $\CDCR$ \. as follows.
% Let \ts $\Mf, x, A,B$ \ts be the input of \ts $\VDCR$\ts.

Let \ts $\Mf=(X,\cI)$ \ts be a binary matroid of rank \ts $\rk(\Mf)=r$ \ts
on \ts $n=|X|$ elements.  Let \ts $x\in \NL(\Mf)$ \ts and \ts $A, B \in \nn$,
where \ts $B>0$.  We can assume that \. $A\geq 1$, as otherwise \ts $A=0$ \ts  and \ts $\VDCR$ \ts is then
equivalent with checking if \. $\iB(\Mf-x) = 0$\., i.e. checking if \ts $\Mf-x$ \ts has rank~$(r-1)$.
We can also assume that $A\geq B$,
as otherwise $A/B<1$ and there exists  $(\Mf,x)$ for which $\rho(\Mf,x)=A/B<1$.
%$\iB(\Mf-x)$ is equal to $0$, which can be done in $|X|$ steps\com{SH}{Do we need citation for this?}.

Without loss of generality we can assume that integers $A$ and $B$ are coprime.
Since we have \. $\iB(\Mf-x) \leq \binom{n}{r}$ \. and \. $\iB(\Mf/x) \leq \binom{n}{r}$,
we can also assume that
\begin{equation}\label{eq:verify-1}
 1 \, \leq \, A, \. B  \, \leq \, \binom{n}{r},
\end{equation}
as otherwise \.  $\VDCR$ \. fails.
Similarly, by Lemma~\ref{lem:r-bound} we can also assume that
\begin{equation}\label{eq:verify-2}
   \frac{A}{B} \, \leq  \, n\ts.
\end{equation}

%We now split the algorithm into two cases.
%First suppose that $A\geq B$.
Let $A'$ be the positive integer given by
\[  A' \ := \   B \. + \. A \. - \.  \bigg\lfloor \frac{A}{B} \bigg\rfloor B.   \]
That is, $A'$ is the unique integer  for which $\tfrac{A'}{B}$ equals 1 plus the fractional part of $\tfrac{A}{B}$.
Note that \. $B \. \leq \. A'\.  \leq 2 B$.  From this point on we proceed
following the proof of Lemma~\ref{lem:st-ratio}.

It follows from Proposition~\ref{p:NTD} that there exists fixed $K>0$ and  an integer \. $m < B$ \.  such that
	\[ \as\left(\tfrac{m}{A'}\right) \ \leq \  K \. \log A' \. (\log \log A')^2  \quad \text{ and } \quad   \as\left(\tfrac{B-m}{A'}\right) \ \leq \  K \. \log A' \. (\log \log A')^2.  \]	
%\[S_{A'}(m) \ \leq \ 3 \. (\log A')^2 \quad \text{ and } \quad  S_{A'}(B-m) \ \leq \ 3 \. (\log A')^2.
%\]
We now use the nondeterminism available in Lemma~\ref{lem:verify}. Nondeterministically choose an integer $m \in \{1,\ldots,B-1\}$,
encoded in binary, and reject the computation branch unless $m$ satisfies the two inequalities above.  Using the Euclidean algorithm, we can compute the continued fractions of $m/A'$ and $(B-m)/A'$, together with their quotient sums, in polynomial time (in $n$).
Proposition~\ref{p:NTD} then guarantees that at least one computation branch passes this verification.

%
%At this point we \emph{guess} \ts such~$m$.  Since computing the quotients of \ts $m/A'$ \ts
%can be done in polynomial time, we can verify in polynomial time that \ts $m$ \ts satisfies
%the inequalities above.

Let \ts $[a_0,\ldots,a_s]$ \ts and \ts $[b_0,\ldots, b_t]$ \ts be a continued fraction representation of \ts $A'/m$ \ts and \ts $A'/(B-m)$, respectively.
By Theorem~\ref{thm:graph-cf-sum}, there exists a connected loopless bridgeless planar graph $G$ and an edge $e \in E(G)$ such that
\[ \frac{\at(G-e)}{\at(G/e)} \ = \   \frac{1}{[a_0 \ts ; \ts a_1,\ldots,a_s]} \. + \. \frac{1}{[b_0 \ts ; \ts b_1,\ldots,b_t]} \ = \  \frac{B}{A'}
\]
and
\begin{align*}
 |E(G)| \ & = \ \as\left(\tfrac{m}{A'}\right)  \. + \. \as\left(\tfrac{B-m}{A'}\right) \. + \. 1 \ \leq  \  2 \ts K \ts \log A' \. (\log \log A')^2  \ts + \ts 1  \\
  \, & \le_{\eqref{eq:verify-1}}  \  2\ts K \ts \log 2 \tbinom{n}{r} \. \Big(\log \log 2  \tbinom{n}{r}\Big)^2 \. + \. 1 \ = \   O\big(n \. (\log n)^2\big).
\end{align*}

Let $G'$ and $e'$ be the graph and edge obtained by applying Lemma~\ref{lem:graph-cf-1} for \. $\lfloor A/B\rfloor-1$ \. many times to the planar dual $G^\ast$ of $G$.
Then we have
\[  \frac{\at(G'-e')}{\at(G'/e')} \ = \  \left\lfloor \frac{A}{B} \right\rfloor  \. - \. 1 \. + \. \frac{\at(G^\ast-e)}{\at(G^\ast/e)} \ = \  \left\lfloor \frac{A}{B} \right\rfloor \. - \. 1 \. + \. \frac{A'}{B} \ = \ \frac{A}{B} \]
and
\[ E(G') \ = \  \left\lfloor \tfrac{A}{B} \right\rfloor  \. - \. 1  \. + \. |E(G)| \ \leq_{\eqref{eq:verify-2}} \ n \. - \. 1  \. + \. |E(G)| \ = \ O(n \. (\log n)^2).  \]
Now, let \ts $\Nf:=(E(G'),\cJ)$ \ts be the graphical matroid corresponding to $G'$,
where $\cJ$ are spanning forests in~$G'$, and let $y =e'$.  Then we have
\[ \rho(\Nf,y) \, = \, \rho(G',e') \, = \,   \frac{\at(G'-e')}{\at(G'/e')} \, = \, \frac{A}{B} \..
\]
Thus, the decision problem \ts $\VDCR$ \ts  with input \ts $\Mf, x, A,B$ \ts can be simulated
by \ts $\CDCR$ \ts with input \ts $\Mf,x,\Nf,y$. This completes the proof. \qed

\medskip

\section{Proof of Theorem~\ref{t:main-negative}} \label{s:main-proof}

In this section we complete the proof of Theorem~\ref{t:main-negative}.

\subsection{Two more reductions}\label{ss:main-proofs-two-reductions}
We also need two minor technical lemmas:

\begin{lemma} \label{lem:more}
	For all \ts $k> \ell$, \.
	$\EBULC_\ell$ \. reduces to \. $\EBULC_{k}$\..
\end{lemma}

\begin{proof}[Proof of Lemma~\ref{lem:more}]
	Let $\Mf,R,a, \bS=(S_1,\ldots,S_\ell), \cb=(c_1,\ldots,c_\ell)$ be an input of 	$\EBULC_\ell$.
	Let \ts $S_{\ell+1}=\ldots=S_{k}=\varnothing$ \ts and \ts $c_{\ell+1}=\ldots=c_k=0$.
	Let \. $\bS':=(S_1,\ldots,S_k)$ \ts and \ts $\bc':=(c_1,\ldots,c_k)\ts$.
	It then follows that
	\[ \aP_{\bScp}(\Mf,R,a) \, = \,   \aP_{\bSc}(\Mf,R,a). \]
	We conclude that the decision problem \. $\EBULC_\ell$ \.
	with input $\Mf,R,a,\bS,\cb$, \.
	 is equivalent to the decision problem
	\. $\EBULC_k$ \. with input $\Mf,R,a,\bS',\cb'$.  \end{proof}

\smallskip

\begin{lemma}\label{lem:BBM}
$\BBM$ \. is polynomial time equivalent to \. $\NDCR$\..
\end{lemma}

\begin{proof}[Proof of Lemma~\ref{lem:BBM}]
Note that \ts $\NDCR$ \ts reduces to \ts $\BBM$ \ts by definition.
In the opposite direction, let $\Mf$ be a binary matroid of rank \ts $r=\rk(\Mf)$.
Construct a basis  $\{x_1,\ldots, x_{r}\}$  of $\Mf$ by the Gaussian elimination.
Denote by  \. $\Mf_i$  \. the contraction of $\Mf$ by  $\{x_1,\ldots, x_i\}$  (note that $\Mf_0=\Mf$ by definition).
We have
\[ \iB(\Mf) \ = \ \prod_{i=1}^r  \frac{\iB(\Mf_{i-1})}{\iB(\Mf_i)}     \ = \  \prod_{i=1}^r 1 + \frac{\iB(\Mf_{i-1}-x_{i})}{\iB(\Mf_{i-1}/x_i)}  \ = \
\prod_{i=1}^r 1 + \rho(\Mf_{i-1},x_i), \]
where the first equality follows from a telescoping argument and the observation that $\iB(\Mf_{r})=1$.
%\[ \iB(\Mf) \ = \ \prod_{i=1}^r  \frac{\iB(\Mf_0)}{\iB(\Mf_1)} \. \cdot \. \frac{\iB(\Mf_1)}{\iB(\Mf_2)} \cdots    \ = \  \left(1 + \frac{\iB(\Mf_0-x_1)}{\iB(\Mf_0/x_1)}  \right) \. \left(1 + \frac{\iB(\Mf_1-x_2)}{\iB(\Mf_1/x_2)}  \right) \cdots, \]
This gives the desired reduction.
\end{proof}

\smallskip

\subsection{Putting everything together}
First, we need the following recent result$\ts:$

\smallskip

\begin{thm}[{\rm Knapp--Noble \cite[Thm~53]{KN23}}{}] \label{t:KN}
$\BBM$ is \ts $\SP$-complete for binary matroids.
\end{thm}

\smallskip

By Lemma~\ref{lem:BBM}, we conclude that \ts $\NDCR$ \ts is \ts $\SP$-hard.  We then have:
\begin{equation}\label{eq:together-LE}
	\ComCla{PH} \. \subseteq \. \P^{\SP} \. \subseteq \. \P^{\<\NDCR\>} \. \subseteq \. \NP^{\<\VDCR\>}\.,
\end{equation}
where the first inclusion is Toda's theorem \cite{Toda}, the second inclusion is because \ts  $\NDCR$ \ts
is $\SP$-hard, and the third inclusion is because one can simulate \ts $\NDCR$ \ts by first
guessing and then verifying the answer.

% Combining Lemma~\ref{l:Sta-Sta-more}, \ts  Lemma~\ref{l:Flat-Sta} \ts and \ts  Lemma~\ref{l:Quart-Flat},
%we conclude that \. {\sc QuadRLE} \. reduces to \. {\sc EqualityStanley}$_k$\ts.

  We now have:
\begin{equation}\label{eq:together-verify}
\aligned
&	\NP^{\<\VDCR\>} \. \subseteq \. \NP^{\<\CDCR\>} \\
& \hskip2.cm \subseteq \. \NP^{\<\EBULC_{1}\>} \. \subseteq \. \NP^{\<\EBULC_{k}\>}
\endaligned
\end{equation}
where the first inclusion is the Verification Lemma~\ref{lem:verify},
the second inclusion is Lemma~\ref{lem:CDCR-to-EBULC},
and the third inclusion is Lemma~\ref{lem:more}.

Now, suppose \ts $\EBULC_{k} \in \PH$.  Then \ts $\EBULC_{k} \in \ts \Sigmap_m$ \ts
for some~$m$.  Combining \eqref{eq:together-LE} and \eqref{eq:together-verify}, this implies:
\begin{equation}\label{eq:together-collapse}
	\PH \. \subseteq \. \NP^{\<\EBULC_{1}\>} \. \subseteq \.  \NP^{\Sigmap_m} \. \subseteq \. \Sigmap_{m+1}\,,
\end{equation}
as desired. \qed

\medskip

\section{Combinatorial atlas}\label{s:atlas}

In this section we give a brief review of the theory of combinatorial atlas.
This is the main tool used to prove Theorem~\ref{t:main-positive}.
We will be concise in our explanation of this tool for the sake of brevity and refer the reader to \cite[$\S3,\S4${}]{CP-intro}  for more detailed introduction,
and to \cite{CP} for an even more in-depth discussion on this topic.

%\com{SH}{The stuffs below are copied directly from IntroAtlas, so it might  trigger arxiv's plagiarism detector.}

%Let \ts $\cP=(\Vf,\prec)$ \ts be a  (possibly infinite) poset of bounded height, i.e. every chain in the poset is of finite size.

\smallskip

\subsection{Informal overview}\label{ss:atlas-quick}
On a basic level, an inequality \ts $a^2\ge bc$ \ts can be viewed as a result
that
$$\det(M) \. \le \. 0 \quad \text{where} \quad M \. := \. \begin{pmatrix}
a & b \\ c & a
\end{pmatrix}.
$$
In other words, the inequality says that the eigenvalues of \ts $M$ \ts
can be zero, but cannot be of the same sign.
In a similar vein, the log-concavity of a sequence can be reformulated
as a claim that a certain matrix has one positive eigenvalue (OPE).
Constructing such a matrix is rather technical, and for the sequence
\ts $\{\aP(\Mf,R,a)\}$ \ts  this will be done in the next section.

Proving that a matrix has (OPE) is difficult, and is done essentially
by a strong induction.  In this section, we present the setup
(combinatorial atlas) which allows us to formalize this induction
and obtain (OPE) in Theorem~\ref{t:Hyp} (proved in our prior work).
An inelegant but  unavoidable feature of the setup is the one-parameter
deformation which allows us  to prove the result for matrices with no
zero eigenvalues, and taking the limit later.

Note that we do not need to reprove the Stanley--Yan inequality \eqref{eq:SY},
even though our construction does give an independent proof, see Proposition~\ref{prop:C-hyp}.
An important feature of the combinatorial atlas is its ability to also give equality
conditions in certain cases.  This requires further assumptions and a technical
Theorem~\ref{t:equality Hyp} (also proved in our prior work).  We include this
setup in this section before using it to prove
Theorem~\ref{t:main-positive} in the next section.

To give an intuitive picture, it is useful to regard a combinatorial atlas as an
induction diagram.  At a vertex $v$, the associated matrix $\bM_v$ is
the object whose hyperbolicity we wish to establish, while the matrices
at the out-neighbors of $v$ typically represent simpler instances of
the same problem.  The edge transformations compare the corresponding
linear-algebraic data, and the associated vector $\mathbf h_v$ records
the local coefficients relating the matrix at $v$ to the matrices at
its out-neighbors.

The machinery below has two principal outputs, which propagate information
in opposite directions.  For the inequality, the \defn{local--global principle}~Theorem~\ref{t:Hyp}, deduces property \eqref{eq:Hyp} at a  vertex from
property \eqref{eq:Hyp} at all of its out-neighbors.  Thus, hyperbolicity
propagates against the orientation of the edges.
For equality, the direction is reversed.  The \defn{local--global equality
principle}~Theorem~\ref{t:equality Hyp}, propagates an equality condition from a
vertex to \emph{some} of its out-neighbors.
The equality propagation is considerably more delicate, since equality
need not propagate to \emph{every} out-neighbor.  The notions of a global pair,
a functional source, a functional edge, and a functional target are
introduced precisely to identify and keep track of the vertices and
directed paths along which equality does propagate.  In particular, the
associated vector $\mathbf h_v$ helps determine which outgoing edges are
functional.  This necessary bookkeeping accounts for much of the
technical setup below.

\smallskip

\subsection{The setup}\label{ss:atlas-setup}
Let \. $\Qf=(\Vf, \Ef)$ \.  be a (possibly infinite) acyclic digraph.
%
% given by the
%Hasse diagram $\Hc_\cP$ of~$\cP$, i.e.  an edge \ts $(v,v') \in \Ef$ \ts is oriented from $v$ to $v'$ whenever \ts $v\prec v'$ \ts and there is no $v''$ such that \ts $v \prec v'' \prec v'$\ts.
We denote by  \ts $\Vfm\subseteq\Vf$ \ts  the set of  \defn{sink vertices} in~$\Qf$ (i.e. vertices without outgoing edges).  Similarly, denote by
\ts $\Vfp:=\Vf\sm \Vfm$ \ts the \defn{non-sink vertices}.
 We denote by \ts $\vfs$ \ts
 the set of \defnb{out-neighbors} of $v$, i.e.  vertices \ts $v'\in \Vf$ such that \ts $(v,v')\in \Ef$\footnote{The set $\vfs$ was denoted by ${\text{\it \sts v}}^\ast$ in \cite{CP-intro,CP}.}.

\smallskip

\begin{definition}\label{def:atlas}
	{\rm
		Let \ts $\ar$ \ts be a positive integer.
		A \defna{combinatorial atlas} \ts $\AA$ \ts of dimension \ts $\ar$ \ts
		is an acyclic digraph   \. $\Qf:=(\Vf, \Ef)$ \. with an additional structure:
		\begin{itemize}[leftmargin=*]
			\item Each vertex \ts $\vf \in \Vf$ \ts is associated with a  pair \ts $(\bM_{\vf}, \hb_{\vf})$\ts, where \ts
			$\bM_{\vf}$ \ts is a  nonnegative 		symmetric
			\quad   $\ar \times \ar$ \ts matrix,
			and \ts $\hb_{\vf}\in \rr_{\ge 0}^\ar$ \ts is a nonnegative vector.
			
		\item 	Each vertex \ts $\vf \in \Vfp$ \ts has out-degree equal to $\ar$, and the $i$-th edge is labeled \ts
		$\ef^{\<i\>}=(\vf,\vf^{\<i\>})$ \ts for \ts  $1 \leq i \leq \ar$.
		
		\item Each edge $\ef^{\<i\>}$ is associated to a linear transformation \.
		$\bT^{\langle i \rangle}_{\vf}: \ts \Rb^{\ar} \to \Rb^{\ar}$.
		\end{itemize}
		\smallskip

		We call \. $\bM_v=(\aM_{ij})_{i,j \in [\ar]}$ \. the \defn{associated matrix} of $\vf$,
		and \. $\hb =
		\hb_v= (\ah_i)_{i \in [\ar]}$  \. the \defn{associated vector} of~$\vf$.
		In notation above, we have \ts $\vf^{\<i\>}\in \vfs$, for all \ts $1\le i \le \ar$.
		The associated vector $\mathbf h_v$ should be viewed as local coupling
		data between the matrix at $v$ and the matrices at its out-neighbors.
		Its role becomes explicit in the properties introduced in Section~\ref{ss:atlas-ineq}.
		
%				Whenever clear, we drop the subscript~$\vf$ to avoid cluttering.
	}
\end{definition}

\smallskip

A common objective of the setup above is to demonstrate that the matrices in the atlas satisfy the following property.
A matrix $\bM$ is called \defn{hyperbolic}, if
\begin{equation*}\label{eq:Hyp}
	\langle \vb, \bM \wb \rangle^2 \  \geq \  \langle \vb, \bM \vb \rangle  \langle \wb, \bM \wb \rangle \quad \text{for every \ \, $\vb, \wb \in \Rb^{\ar}$, \ \  such that \ \ $\langle \wb, \bM \wb \rangle > 0$}. \tag{Hyp}
\end{equation*}
For the atlas \ts $\AA$, we say that \ts $\vf\in \Vf$ \ts is \defn{hyperbolic},
if the associated matrix \ts $\bM_{\vf}$ \ts is hyperbolic, i.e. it satisfies \eqref{eq:Hyp}.
We say that  \ts $\AA$ \ts satisfies \ts\defng{hyperbolic property} \ts if every
\ts $\vf\in \Vf$ \ts is hyperbolic.

\smallskip

For symmetric matrices,
property \eqref{eq:Hyp} is equivalent to the following property:
\begin{equation*}\label{eq:OPE}\tag{OPE}
	\text{$\bM$ \. has at most \. \defng{one positive eigenvalue} \. (counting multiplicity).}
\end{equation*}
The equivalence between these two properties is well-known in the literature, see e.g.\
\cite{Gre}, \cite[Thm~5.3]{COSW}, \cite[Lem.~2.9]{SvH19}  and \cite[Lem.~2.5]{BH}.

\smallskip

\begin{lemma}[{\cite[Lem.~5.3]{CP}}]\label{l:Hyp is OPE}
	Let \ts {\rm $\bM$} \ts be a symmetric  matrix.
	Then$:$
	\[  	\text{\ts {\rm $\bM$} \ts satisfies \eqref{eq:Hyp} \ts}
	\ \Longleftrightarrow  \  \text{{\rm $\bM$} \ts satisfies \eqref{eq:OPE}}. \]
\end{lemma}

\smallskip

\subsection{Sufficient conditions for hyperbolic property}\label{ss:atlas-ineq}
In practice, verifying that an atlas  satisfies the hyperbolic property can be a nontrivial task. In this subsection we present a set of conditions that are sufficient to imply the hyperbolic property.

We say that the atlas \ts $\AA$ \ts satisfies the \ts\defng{inheritance property} \ts if for
every non-sink vertex \ts $\vf\in \Vfp$, we have:
\begin{equation*}\label{eq:Inh}\tag{Inh}
	\begin{split}
		&	  (\bM \vb )_{i} \ = \ \big \langle   \bT^{\<i\>}\vb, \, \bM^{\<i\>} \,
		\bT^{\<i\>}\hb \big \rangle \quad \text{ for every} \ \ \, i \in \supp(\bM) \text{ \ \ and \ \. }  \vb \in \Rb^{\ar},
	\end{split}
\end{equation*}
where \ts $\bM:=\bM_v$\ts, \ts  $\bT^{\<i\>}=\bT^{\<i\>}_{\vf}$\., \ts $\hb=\hb_{\vf}$ \. and
\. $\bM^{\<i\>}:=\bM_{{\vf}^{\<i\>}}$ \. is the matrix associated with~$\vf^{\<i\>}$\..
Note that for the remainder of this section, we omit the subscript
$\vf$
 from some notations to prevent cluttering the equations.

\smallskip

We say that  \ts $\AA$ \ts satisfies the \ts\defng{pullback property} \ts if for
every non-sink vertex \ts $\vf\in \Vfp$, we have:
\begin{equation*}\label{eq:Pull}\tag{Pull}
	\sum_{i\ts\in\ts \supp(\bM)} \. \ah_{i} \, \big \langle   \bT^{\<i\>}\vb, \, \bM^{\<i\>} \,  \bT^{\<i\>}\vb  \big \rangle  \ \geq \ \langle \vb, \bM \vb \rangle \qquad \text{ for every } \vb \in \Rb^{\ar}.
\end{equation*}
We say that  \ts $\AA$ \ts satisfies the  \ts\defng{pullback equality property} \ts if for
every non-sink vertex \ts $\vf\in \Vfp$, we have:
\begin{equation*}\label{eq:PullEq}\tag{PullEq}
	\sum_{i\ts\in\ts \supp(\bM)} \. \ah_{i} \, \big \langle   \bT^{\<i\>}\vb, \, \bM^{\<i\>} \,  \bT^{\<i\>}\vb  \big \rangle  \ = \ \langle \vb, \bM \vb \rangle \qquad \text{ for every } \vb \in \Rb^{\ar}.
\end{equation*}
Clearly \eqref{eq:PullEq}  implies \eqref{eq:Pull}.

\smallskip

We say that a non-sink vertex \ts $\vf\in \Vfp$ \ts is \defn{regular} if the following positivity conditions are satisfied\footnote{In \cite{CP-intro}, there was an additional assumption that
	\ts $\bM_{\vf} \hb_{\vf}$ \ts given in
	\eqref{eq:hPos} is strictly positive
	when restricted to the support of \ts $\bM_{\vf}$\ts.
	Note that this additional assumption is redundant here because we assume that \ts $\bM_{\vf}$ \ts is a nonnegative matrix.}:
\begin{align*}\label{eq:Irr}\tag{Irr}
	& \text{The associated matrix \ts $\bM_{\vf}$ \ts restricted to its support is irreducible.} \\
	\label{eq:hPos} \tag{h-Pos} &\text{Vector  \ts $\hb_{\vf}$ \ts  is strictly positive when restricted to the support of \ts $\bM_{\vf}$\..}
\end{align*}
Here, irreducibility of $\bM_v$ is understood in the matrix-theoretic sense, and the support of $\bM_v$ is the set of indices corresponding to the nonzero rows of $\bM_v$; both notions are  defined in Section~\ref{ss:def-basic}.

The following theorem provides a method to establish the hyperbolic property for atlases by reducing the problem to checking the hyperbolic property only for the sink vertices of the atlases.
%The following theorem gives  log-concave inequalities for combinatorial objects that can be represented by atlases.

\smallskip

\begin{thm}[{\rm {\em \defna{local–global principle}, {\em see}~\cite[Thm~5.2]{CP}, \cite[Thm~3.4]{CP-intro}}}{}]
	\label{t:Hyp}
	Let $\AA$ be a combinatorial atlas that satisfies properties \eqref{eq:Inh} and \eqref{eq:Pull}, and
	let \ts $\vf\in \Vfp$ \ts be a non-sink regular vertex of~$\Qf$.
	Suppose every out-neighbor of \ts $\vf$ \ts is hyperbolic.
	Then \ts $\vf$ \ts is also hyperbolic.
\end{thm}

\smallskip

In our applications, the pullback property \eqref{eq:PullEq} is more involved than the inheritance
property~\eqref{eq:Inh}.  Below we give sufficient conditions for
\eqref{eq:PullEq} that are easier to establish.

\smallskip

We say that the atlas $\AA$ satisfies the
\defng{identity property}, if for every non-sink vertex \ts
$\vf\in \Vfp$ \ts and every \ts $i \in \supp(\bM)$, we have:
\begin{equation*}\label{eq:Iden}\tag{Iden}
	\bT^{\<i\>}: \Rb^{\ar} \to \Rb^{\ar} \ \text{ is the identity mapping.}
\end{equation*}	
We say that $\AA$ satisfies the \defng{transposition-invariant property},  if for every non-sink vertex \ts
$\vf\in \Vfp$,
\begin{equation*}\label{eq:TPInv}\tag{T-Inv}
	\aM_{jk}^{\<i\>}	\ = \  \aM_{ki}^{\<j\>}	 \ = \ \aM_{ij}^{\<k\>} \qquad 	\text{ for every  $i, j, k \in \supp(\bM)$},
\end{equation*}
where \ts $\aM_{jk}^{\<i\>}$ \ts is the $(j,k)$-th entry of the matrix $\bM^{\<i\>}$.

We say that $\AA$ has the \defng{decreasing support property}, if for every non-sink vertex \ts $\vf \in \Vfp$,
\begin{equation*}\label{eq:Decsupp}
	\tag{DecSupp}
	\supp(\bM) \ \supseteq \ \supp\big(\bM^{\<i\>}\big) \quad \text{ for every $i \in \supp(\bM)$.}
\end{equation*}

\smallskip

\begin{thm}[{\rm cf.~\cite[Thm~6.1]{CP},~\cite[Thm~3.8]{CP-intro}}{}]
	\label{t:Pull}
	Let \ts $\AA$ \ts be a combinatorial atlas that satisfies \eqref{eq:Inh}, \eqref{eq:Iden}, \eqref{eq:TPInv} and \eqref{eq:Decsupp}.
	Then  \ts $\AA$ \ts also satisfies  \eqref{eq:PullEq}.
\end{thm}

\smallskip

\subsection{Equality conditions for hyperbolic inequalities}\label{ss:atlas-equality}
In this subsection, we discuss a strengthening of the hyperbolic property for a given matrix, where the inequality in \eqref{eq:Hyp} is replaced with equality.

A \defn{global pair} \ts $\fb,\gb \in \Rb^{\ar}$ \ts
is a pair of nonnegative vectors, such that
\begin{equation*}
	\label{eq:PosGlob} \tag{Glob-Pos}  \text{$\fb \ts +\ts\gb$ \ is a strictly positive vector}\ts.
\end{equation*}
Here \ts $\fb$ \ts and \ts $\gb$ \ts are global in the sense that
are independent of the vertex
$v\in\Vf$: the same fixed pair is used together with each of the
vertex-dependent associated matrices $\bM_v$.

\smallskip

Fix a number \ts $\as>0$.
We say that a vertex \ts $\vf \in \Vf$ \ts satisfies~\eqref{eq:sEqu}, if
\begin{equation*}\label{eq:sEqu}
	\tag{s-Equ}  \langle \fb , \bM \fb \rangle  \ = \  \as \, \langle \gb , \bM \fb \rangle \ = \  \as^2 \, \langle \gb , \bM \gb \rangle,
\end{equation*}
where \ts $\bM=\bM_{\vf}$ \ts is the matrix associated to $\vf$.
Observe that \eqref{eq:sEqu} implies that equality occurs in \eqref{eq:Hyp} for substitutions \ts $\vb\gets \gb$ \ts
and \ts $\wb\gets \fb$, since
\[
	\langle \gb , \bM \fb \rangle^2  \ = \  \as \,\. \langle \gb , \bM \gb \rangle \ \as^{-1} \. \langle \fb , \bM \fb \rangle \ = \ \langle \gb , \bM \gb \rangle  \ \langle \fb , \bM \fb \rangle\..
\]
%We say that the atlas \ts $\AA$ \ts satisfies \defng{$\asr$-equality property}
%if~\eqref{eq:sEqu} holds for every \ts $\vf \in \Vf$.

\smallskip

The following lemma gives another equivalent condition to check \eqref{eq:sEqu}.

\smallskip

\begin{lemma}[{\cite[Lem~7.2]{CP}}]\label{l:equality kernel}
	Let $\bMr$ be a nonnegative symmetric hyperbolic $\arr \times \arr$ matrix.
	Let \ts $\fb,\gb \in \Rb^{\arr}$ \ts be nonnegative vectors, let \ts $\asr>0$,
	and let \ts $\zb := \fb \ts - \ts \asr  \gb$.
	Then \eqref{eq:sEqu} holds \  \underline{if and only if}  \ $\bMr \zb =0$.	
\end{lemma}

\smallskip

%We now present the first main result of this section,
%which is a \defng{local-global principle} for \eqref{eq:sEqu}.

%In practice, verifying that a given vertex of an atlas satisfies \eqref{eq:sEqu} can be a nontrivial task. Here we present a set of conditions that simplify this process.

A common objective of setting up an atlas is to identify vertices \ts $v\in \Omega$ \ts  that  satisfy
\eqref{eq:sEqu}.
As it turns out, the property \eqref{eq:sEqu} for a given vertex $v \in \Omega$
 is  equivalent to the statement that \emph{some}, but not \emph{all}, of the out-neighbors of $v$, satisfy \eqref{eq:sEqu}.
 The purpose of the definitions below is to identify and describe those neighbors that inherit \eqref{eq:sEqu}.

\smallskip

A vertex \ts $\vf \in \Vfp$ \ts is called a \defn{functional source} if
the following conditions are satisfied:
\begin{align*}
	\label{eq:ProjGlob} \tag{Glob-Proj}  &\af_j \ = \  \big(\bT^{\<i\>} \fb \big)_j \quad \text{ and }  \quad  \ag_j \ = \  \big(\bT^{\<i\>} \gb \big)_j  \quad  \forall  \ i \in \supp(\bM), \ j \in \supp(\bM^{\<i\>}), \\
	\label{eq:hGlob} \tag{h-Glob}  &\text{$\fb\.=\.\hb_{\vf}$\..} %\, for every \, $\vf\in \Vf$.}
\end{align*}	
Here condition~\eqref{eq:ProjGlob} means that \ts $\fb,\gb$ \ts are fixed points of the
projection \ts $\bT^{\<i\>}$ \ts when restricted to the support.
% Similarly, condition~\eqref{eq:hGlob} means that vector \ts $\hb_{\vf}$ \ts associated with vertex~$\vf$ \ts is equal to~$\fb$.
% To simplify the exposition, we replace \ts $\fb \gets \hb$ \ts from this point on.

We say that an edge \ts $\ef^{\<i\>}=(\vf, \vf^{\<i\>})\in \Ef$ \ts is \defn{functional} if
$\vf$ is a functional source and \. $i \in \supp(\bM) \, \cap \, \supp(\hb)$.
A vertex $\wf \in \Vf$ is a \defn{functional target}
of~$\vf$, if there exists a directed path \. $\vf\to \wf$ \. in $\Qf$ consisting
of only functional edges.  Note that a functional target is not necessarily
a functional source.

\smallskip

%The following result gives equality conditions to log-concave inequalities that are
%derived from Theorem~\ref{t:Hyp}.
The following theorem is the main result in this section and is a key
to the proof of Theorem~\ref{t:main-positive} we give in the next section.

\smallskip

\begin{thm}[{\rm {\em \defn{local–global equality principle},~\cite[Thm~7.1]{CP}}}{}]
	\label{t:equality Hyp}
	Let \ts $\AA$ \ts be a combinatorial atlas that satisfies properties
	\eqref{eq:Inh}, \eqref{eq:Pull}.  Suppose also \ts $\AA$ \ts
	satisfies property \eqref{eq:Hyp} for every vertex \ts $v\in \Vf$.
	Let \ts $\fb,\gb$ \ts be a global pair of~$\AA$.
	Suppose a non-sink vertex \ts $\vf \in \Vfp$ \ts satisfies \eqref{eq:sEqu}
	with a constant $\asr>0$.  Then every functional target of $\vf$ also
	satisfies \eqref{eq:sEqu} with the same constant~$\asr$.
\end{thm}

\smallskip
In words, once \eqref{eq:sEqu} holds at a functional source, it propagates along every path of functional edges.

\medskip

\section{Proof of Theorem~\ref{t:main-positive}}\label{s:main-positive}

We begin with a brief outline of the proof.  In Section~\ref{ss:main-atlas-construction}, given the
matroid $\Mf$ in Theorem~\ref{t:main-positive}, we construct a combinatorial atlas whose
associated matrices encode the relevant basis counts in the Stanley--Yan inequality for $k=0$.  In Section~\ref{ss:property-matrix}--\ref{ss:property-atlas}, we verify that the resulting atlas satisfies the structural
properties introduced in Section~\ref{s:atlas}.
In Section~\ref{ss:property-atlas-Hyp} we apply
Theorem~\ref{t:Hyp}  to establish the hyperbolicity of the
associated matrices.  Finally, in Section~\ref{ss:main-positive-proof}, we apply
Theorem~\ref{t:equality Hyp} to propagate the relevant equality condition
through selected vertices of the atlas and translate the resulting
linear-algebraic information into a condition on the combinatorial
structure of $\Mf$, thereby obtaining the characterization stated in
Theorem~\ref{t:main-positive}.

\medskip

\subsection{Combinatorial atlas construction}\label{ss:main-atlas-construction}
In this subsection, we first encode three consecutive terms of the
sequence
$\bigl\{\aP(\Mf,R,a)\bigr\}_{a=0}^r$
in a single nonnegative symmetric matrix.  We then organize a collection of these matrices into a  single combinatorial atlas.

\smallskip

Let \ts $\Mf:=(X,\cI)$ \ts be a matroid with rank \ts $r\geq 2$ \ts on \ts $n:=|X|$ \ts elements.
  Denote by \ts $X^*$ \ts the set of finite words in the alphabet~$X$.
%we define \ts $R^*$ \ts and \ts $Q^*$ analogously.
A word is called \defnb{simple} if it contains each letter at most once; we consider only simple words in this paper.
For a word \ts $\alpha \in X^*$\ts, the \defnb{length} \ts $|\alpha|$ \ts of $\alpha$ is  the number of letters in  $\alpha$.
For two words \ts $\alpha,\beta \in X^*$\ts, we denote by \ts $\alpha\beta \in X^*$ \ts the concatenation of \ts $\alpha$ \ts and \ts $\beta$\ts.

%\smallskip

%Let \ts $\alpha \in R^*$\ts, \ts $\beta \in S^*$\ts, and \ts $0 \leq a \leq r-|\alpha|-|\beta|$\ts.
Let \ts $a \in \{0,\ldots, r\}$ \ts and let \ts $R \subseteq X$\ts.
A word \. $\gamma = x_1\ldots x_{r}\in X^*$ \. of length $r$ is called \defnb{compatible} \ts
with a triple \ts $(\Mf,R,a)$, \ts if
\begin{itemize}
%	\item The length of $\gamma$ is equal to \. $r$\.,
	\item $\{x_1,\ldots, x_r \}$ forms a basis of $\Mf$, and
	\item $x_1,\ldots, x_a \in R$  \. and \. $x_{a+1}, \ldots, x_{r} \in X-R$.
\end{itemize}
We denote by \ts $\Comp(\Mf,R,a)$ \ts  the set of words compatible with
\ts $(\Mf,R,a)$.
Note that every such word  \. $\gamma \in \Comp(\Mf,R,a)$ \. is simple.
The use of ordered words, rather than unordered bases, is a useful bookkeeping
device, and we have
\begin{equation}\label{eq:Comp-P}
	|\Comp(\Mf,R,a)| \ = \  r! \.  \aP(\Mf,R,a),
\end{equation}
where $\aP(\Mf,R,a)$ is as defined in Section~\ref{ss:intro-positive}.

%We now define the matrix corresponding to the triple \ts $(\Mf,R, a)$ \ts

For every  \ts $a \in [r-1]$, we denote by
\. $\bC(\Mf,R, a) := \bigl(\aC_{x\ts y}\bigr)_{x,y \in X}$ \. the symmetric \ts $n \times n$ \ts
matrix given by
\begin{align*}\label{eq:C-1} \tag{DefC-1}
	\aC_{x\ts y} \ := \
	\begin{cases}
		\ \big|\Comp(\Mf/\{x, y\},R, a-1 )\big|, & \ \  \text{if } \ \ x\neq y \ \text{ and } \ \{x,y\} \in \cI\ts,\\
		\hskip4.6cm 0, & \ \  \text{if } \ \ x=y \ \text{ or } \ \{x,y\} \notin \cI\ts.
%		\big|\Comp(\Mf,R, a-1,\alpha xy,\beta)\big| & \text{for }  \ x,y \in R, \\
%		\big|\Comp(\Mf,R, a-1,\alpha,xy\beta)\big| & \text{for }  \ x,y \in S.
	\end{cases}
\end{align*}
Equivalently, \. $\aC_{x,y}$ \. is given by
\begin{align*}\label{eq:C-2} \tag{DefC-2}
	\aC_{x\ts y} \ := \
	\begin{cases}
				\ \big|\big\{\gamma \. : \. x\gamma y \in \Comp(\Mf,R,a)\big\}\big| & \text{for }  \ \ x \in R, \. y \in X-R, \\
				\ \big|\big\{\gamma \. : \. xy\gamma  \in \Comp(\Mf,R,a+1)\big\}\big| & \text{for } \ \ x,y \in R, \\
				\ \big|\big\{\gamma \. : \. \gamma xy  \in \Comp(\Mf,R,a-1)\big\}\big| & \text{for } \ \ x,y \in X-R. \\
		%		\big|\Comp(\Mf,R, a-1,\alpha xy,\beta)\big| & \text{for }  \ x,y \in R, \\
		%		\big|\Comp(\Mf,R, a-1,\alpha,xy\beta)\big| & \text{for }  \ x,y \in S.
	\end{cases}
\end{align*}
Both definitions will be frequently used throughout this section.
The three cases in \eqref{eq:C-2} are chosen so that the three blocks of
$\bC(\Mf,R,a)$ encode three consecutive terms of the Stanley--Yan
sequence, to be made precise in the forthcoming equation \eqref{eq:Cfg}.

It follows from the definition that \ts $\bC(\Mf,R, a)$ \ts is a nonnegative symmetric matrix,
and the diagonal entries of \ts $\bC(\Mf,R, a)$ \ts are equal to $0$.
%The matrices \.  $\bC(\Mf,R, a)$ \. will be the associated matrices of the atlas
Note that the Stanley--Yan inequality (Theorem~\ref{t:SY}) is a direct consequence of the preceding matrix satisfying \eqref{eq:Hyp}, as explained in the next paragraph.

Let \. $\fb,\gb \in \rr^n$ \. be the indicator vector of $R$ and $X-R$, respectively, which will be used as the global pair of this atlas.
It follows from \eqref{eq:C-2} and \eqref{eq:Comp-P} that
\begin{equation*}\label{eq:Cfg}\tag{Cfg}
\begin{split}
& \<\fb, \bC(\Mf,R,a) \gb \> \ = \  {r!} \. \aP(\Mf,R,a), \qquad
\<\fb, \bC(\Mf,R,a) \fb \> \ = \  {r!} \. \aP(\Mf,R,a+1),\\
&\<\gb, \bC(\Mf,R,a) \gb \> \ = \  {r!} \. \aP(\Mf,R,a-1).
\end{split}
\end{equation*}
Therefore, Stanley--Yan inequality will follow once we demonstrate that \ts $\bC(\Mf,R, a)$ \ts satisfies \eqref{eq:Hyp} (which we will prove in Proposition~\ref{prop:C-hyp}).
Moreover, when $\aP(\Mf,R,a)>0$, equality in the
Stanley--Yan inequality is equivalent to property \eqref{eq:sEqu} for
some $s>0$.  Thus, the matrix $\bC(\Mf,R,a)$ provides the link between the
log-concave inequality and the hyperbolicity and equality machinery
developed in Section~\ref{s:atlas}.

%The objective of this section is to recast the proof of the Stanley--Yan inequality within the framework of combinatorial atlases, enabling the use of involved tools to determine the corresponding equality conditions.
%
%The objective of this section is to recast the proof of Stanley--Yan inequality in the framework of combinatorial atlases, so that we can use the tools involved to find the corrresponding equality conditions.

%In particular, Stanley--Yan inequality (Theorem~\ref{thm:Lar})

\smallskip

We now arrange these matrices into the following combinatorial atlas \ts $\AA=\AA(\Mf,R,a)$ \ts of dimension \ts $\ar = n$ \ts, corresponding to the matroid \ts $\Mf$, the subset \ts $R \subseteq X$, and the integer \ts $a \in \{2,\ldots, r-1\}$.
Let \. $\Qf:=\Qf(\Mf,R,a)  :=  (\Vf,\Ef)$ \. be the acyclic graph with \.  $\Vf=   \Vf^0 \. \cup \. \Vf^1$,
where
\begin{align*}
	\Vf^1 \ & := \ \{t \in \rr \mid 0 \leq t \leq 1  \}, \qquad \Vfm  \  := \ X;  \\
	\Theta \  & := \ \bigl\{(t,x):t\in\Omega^1,\ x\in\Omega^0\bigr\}.
\end{align*}
Thus, $\Gamma$ is the complete bipartite digraph with all edges directed
from $\Omega^1$ to $\Omega^0$, and there are no other edges.  In
particular, every vertex in $\Omega^0$ is a sink, and, for
$v=t\in\Omega^1$ and $x\in X$, we write
\[
v^{\langle x\rangle}:=x
\qquad\text{and}\qquad
e_v^{\langle x\rangle}:=(v,x).
\]
Finally, for every $v=t\in\Omega^1$ and $x\in X$, let the linear
transformation
\[
\bT_v^{\langle x\rangle}:\mathbb R^d\longrightarrow\mathbb R^d
\]
associated with the edge
$e_v^{\langle x\rangle}=(v,v^{\langle x\rangle})=(t,x)$ be the identity
map.

%For a non-sink vertex  \ts $v=t \in \Omega^1$ \ts and $x \in X$, the corresponding outneighbor in $\Omega^0$ is  \ts $v^{\<x\>}:= x$\ts.
%\com{SH}{Draw a picture to visualize.}

For each vertex \ts $\vf=x \in \Vf^0$\ts, the associated matrix is
\[ \bM_v \ := \
\bC(\Mf/x,R,a-1) \quad  \text{ if  $x$ is a non-loop of $\Mf$},
\]
and is equal to the zero matrix otherwise. Note that the ground set of $\Mf/x$ is still $X$ (instead of $X-x$) under our convention.
For each vertex \ts $\vf=t \in \Vf^1$\ts,  the associated matrix is
\[  \ \bM_{\vf} \ := \  t \, \bC(\Mf,R,a)  \ +  \ (1-t) \, \bC(\Mf,R,a-1), \]
and the  associated vector \. $\hb_{\vf} :=(\ah_x)_{x \in X} \in \Rb^{\ar}$ \. is defined to have coordinates
\[  \ah_x \ := \
\begin{cases}
	t & \text{ if } \ \  x \in R\.,\\
	1-t & \text{ if } \ \ x \in X-R\..
\end{cases} \]
Equivalently, recalling that $\fb$ and $\gb$ are the
indicator vectors of $R$ and $X-R$, respectively, we have
\. $\hb_t \. = \. t\fb \. +  \. (1-t)\gb$\..
This choice  has two roles.  For $0<t<1$, the vector $\mathbf h_t$ is
strictly positive, as required when applying the local--global principle.
At the endpoints $t=0$ and $t=1$, we have $\mathbf h_0=\mathbf g$ and $\mathbf h_1=\mathbf f$, which identifies
the outgoing edges along which the equality condition will propagate.
The same coefficients $t$ and $1-t$ also provide the weights in the
inheritance and pullback identities.
See Figure~\ref{fig:atlas} for an illustration of this atlas.

%Finally, let
%the linear transformation \. $\bT^{\<x\>}_v: \Rb^{\ar} \to \Rb^{\ar}$ \.
%associated to the edge \ts $(\vf,\vf^{\<x\>})$ to be the identity map.

%	
% \begin{cases}
%	\bC(\Mf,R,a-1,x,\varnothing) & \text{ if } x \in R,\\
%		\bC(\Mf,R,a-1,\varnothing,x) & \text{ if } x \in S.
%\end{cases}

\begin{figure}[hbt]
	\centering
\includegraphics[width=\textwidth]{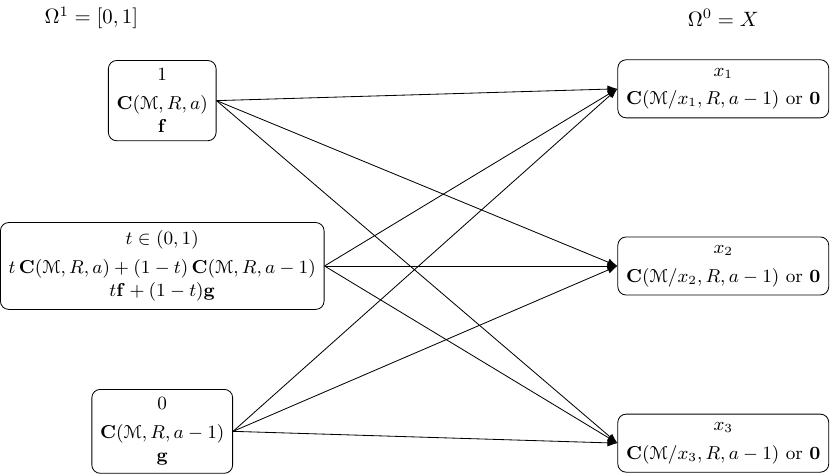}
    \caption{Schematic depiction of the combinatorial atlas
	$\mathbb A(\Mf,R,a)$.  Every vertex in $\Omega^1=[0,1]$ has a
	directed edge to every sink vertex in $\Omega^0=X$.
	The associated
	matrices and vectors are displayed at the corresponding vertices.
		Each edge $(v,x)$ carries the identity map $T_v^{\langle x\rangle}=\text{Id}$.
			For sink vertices $x\in X$, one may also indicate $\hb_x=\bO$.
	}
	\label{fig:atlas}
\end{figure}

\smallskip

\subsection{Properties of the constructed matrices}\label{ss:property-matrix}

In this subsection
we prove  two lemmas  that rule out the degeneracies that could obstruct the
application of the atlas machinery.
The first Lemma~\ref{lem:C-supp}
  identifies the
support of $\bC(\Mf,R,a)$ and proves its irreducibility; this will imply
the regularity of the relevant non-sink vertices.
Recall that \ts $\NL(\Mf)$ \ts is the set of non-loops of~$\Mf$.
\smallskip

\begin{lemma}\label{lem:C-supp}
	Let $\Mf$ be a matroid of rank $r\geq 2$, let $R \subseteq X$,
	and let $a\in [r-1]$ such that $\aP(\Mf,R,a)>0$.
	Then we have:
	\begin{itemize}
				\item \. the support of \. $\bC\.(\Mf,R,a)$ \. is equal to $\NL(\Mf)$, and
		\item \. matrix \. $\bC\.(\Mf,R,a)$ \. is irreducible.
	\end{itemize}
\end{lemma}

\begin{proof}
	Write
	\. $\bC(\Mf,R,a)=(\aC_{x,y})_{x,y\in X}$\..
	By \eqref{eq:C-1}, the row and column indexed by a loop of $\Mf$ are
	zero.  It therefore remains to show that every non-loop belongs to the
	support and that all non-loops lie in the same irreducible component of $\bC(\Mf,R,a)$.

	Since
$\aP(\Mf,R,a)>0$,
there exists a basis $B$ of $\Mf$ such that \. $|B \cap R|=a$\..
Since $a \in [r-1]$, this implies that there exists \ts $x \in R$  \ts and \ts $y \in X-R$ \ts such that   \ts $x,y\in B$\ts.
It then  follows from \eqref{eq:C-1} that
$\aC_{x,y}>0$, and hence
 $x$ and $y$ are contained in the same irreducible component.

Now, let $z$ be an arbitrary non-loop  of $\Mf$.
It suffices to show that either \. $\aC_{x,z}>0$ \. or \.  $\aC_{y,z}>0$.
 If $z\in B\cap R$, then
$B-\{y,z\}$ is a basis of $\Mf/\{y,z\}$ containing exactly $a-1$
elements of $R$, and hence $\aC_{y,z}>0$ by \eqref{eq:C-1}.  Similarly, if
$z\in B\cap(X- R)$, then $\aC_{x,z}>0$.

It remains to consider the case $z\notin B$.  Choose a basis $D$ containing $z$.
By symmetric basis exchange applied to the two bases $B$ and $D$, there
exists $z'\in B- D$ such that
$A:=B-z'+z$
is a basis of $\Mf$.
Suppose first that $z\in R$.  If $z'\in R$, then
$|A\cap R|=a$ and $y\in A$, so $\aC_{y,z}>0$ by \eqref{eq:C-2}.  If
$z'\in X- R$, then $|A\cap R|=a+1$ and $x\in A$, so
$\aC_{x,z}>0$ by \eqref{eq:C-2}.  The case $z\in X- R$ is analogous: if
$z'\in X- R$, then $C_{x,z}>0$, while if $z'\in R$, then
$C_{y,z}>0$.
This completes the proof.
%For the first claim it suffices to show that \. $z \in \supp \. (\bC(\Mf,R,a))$\., and for the second claim it suffices to show that  that $z$ is contained in the same irreducible component as $x$ and $y$.
%We will without loss of generality assume that $z \in R$, as the proof of the other case is analogous.
%
%There are now two possibilities.
%First suppose that $z \in B$.
%Then \ts $B':=B-y-z$ \ts is a basis of \. $\Mf/\{y,z\}$ \. such that $|B'\cap R|=a-1$.
%This implies that  \. $|\Comp(\Mf/\{y,z\},R,a-1)|>0$\., so it follows from \eqref{eq:C-1} that $z$ is contained in the support of \. $\bC(\Mf,R,a)$ \. and $z$ is contained in the same irreducible component as $y$.
%
%
%Now suppose that  \ts $z \notin B$\ts.
%By the symmetric basis exchange property, there exists $z' \in B$ such that \. $A:=B-z'+z$  \. is a basis of $\Mf$.
%Now, if  \. $|A\cap R|=a$ \. (i.e. $z' \in R$),
%then  \ts $A':=A-y-z$ \ts is a basis of \. $\Mf/\{y,z\}$ \. satisfying $|A'\cap R|=a-1$.
%This implies that  \. $|\Comp(\Mf/\{y,z\},R,a-1)|>0$, so it follows from \eqref{eq:C-1} that $z$ is contained in the support of \. $\bC(\Mf,R,a)$ \. and $z$ is contained in the same irreducible component as $y$.
%On the other hand,
%if \. $|A\cap R|=a+1$ \. (i.e. $z' \notin R$),
%then  \ts $A':=A-x-z$ \ts is a basis of \. $\Mf/\{x,z\}$ \. satisfying $|A'\cap R|=a-1$.
%This implies that  \. $|\Comp(\Mf/\{x,z\},R,a-1)|>0$\., so it follows from \eqref{eq:C-1} that $z$ is contained in the support of \. $\bC(\Mf,R,a)$ \. and $z$ is contained in the same irreducible component as $x$.
\end{proof}

%\smallskip
%
%\begin{lemma}\label{lem:C-Irr}
%	Let $\Mf$ be a matroid, let $R \subseteq X$,
%and let $1\leq a\leq r-1$ such that \. $\aP(\Mf,R,a)>0$\..
%Then the matrix
%$\bC\.(\Mf,R,a,\varnothing,\varnothing)$ is irreducible.
%\end{lemma}

\smallskip
The next
Lemma~\ref{lem:C-center}
shows that every nonzero contraction matrix appearing at a sink vertex
remains in the nonvanishing range required by the induction hypothesis.

\smallskip

\begin{lemma}\label{lem:C-center}
	Let $\Mf$ be a matroid of rank $r\geq 2$, let $R \subseteq X$,
and let \. $a \in \{2,\ldots,r-1\}$ \. such that \. $\aP(\Mf,R,a)>0$ \. and \. $\aP(\Mf,R,a-1)>0$\..
Then, for every $x \in X$ that is not a loop,
\[ \aP(\Mf/x,R,a-1)>0. \]
\end{lemma}

\begin{proof}
We will without loss of generality assume that $x\in R$,  as the proof of the other case is analogous.
	By the assumption, there exist bases $A$ and $B$ of $\Mf$, such that
	\. $|A \cap R|=a$ \. and \. $|B \cap R|=a-1$\..
	Since $x$ is not a loop, there exists a basis $D$ of $\Mf$
	containing $x$.
	
	We now claim that there exists a basis $A'$ of $\Mf$ such that $x \in A'$ and  \. $|A'\cap R| \in \{a,a+1\}$.
	Indeed, if $x \in A$, then let $A':=A$.
	 Otherwise, by applying  the symmetric basis exchange
	property to the bases $A$ and $D$,
	it follows that there exists $y \in A \setminus D$ such that $A':=A-y+x$
	is a basis of $\Mf$.
	Since $x\in R$, we then have
	$|A'\cap R|\in\{a,a+1\}$.
	This completes the proof of the claim.

%Applying the symmetric basis exchange property to  $A$, we get that

	Similarly,  there exists
	a basis $B'$ of $\Mf$ such that $x \in B'$ and  \. $|B'\cap R| \in \{a-1,a\}$.
	If either \. $|A'\cap R|=a$ \. or \. $|B'\cap R|=a$ \. then we are done,
	since $A'-x$ (resp. $B'-x$) is then  a basis of $\Mf/x$ for which \. $|(A'-x) \cap R| =a-1$ \. (resp. $|(B'-x) \cap R| =a-1$).
	So we assume that \. $|A'\cap R|=a+1 $ \. and \. $|B' \cap R|=a-1$\..
	Then by applying the basis exchange properties to $A'$ and $B'$ (possibly more than once),
	there exists a basis $C'$ of $\Mf$ such that
	\ts $|C'\cap R|=a$ \ts and \ts $x\in C'$\ts, and the claim follows by the same argument as before.
\end{proof}

\smallskip

\subsection{Properties of the atlas}\label{ss:property-atlas}
In this subsection, we show that the atlas \ts $\AA(\Mf,R,a)$\ts, constructed in Section~\ref{ss:main-atlas-construction},
satisfies the properties outlined in Section~\ref{ss:atlas-ineq}.
 This, in turn, enables us to apply tools and methods associated to atlases as described in Section~\ref{s:atlas}.

\smallskip

\begin{lemma}\label{lem:atlas-prop-1}
	Let $\Mf$ be a matroid of rank $r\geq 2$, let $R \subseteq X$,
	and let \. $a\in \{2,\ldots, r-1\}$.
	Then the atlas \. $\AA(\Mf,R,a)$ \. satisfies \eqref{eq:Inh}, \eqref{eq:Iden}, \eqref{eq:TPInv}.
\end{lemma}

\begin{proof}
	The condition \eqref{eq:Iden} follows directly from the definition.

		We now verify  \eqref{eq:Inh}, which has the following direct counting interpretation.  For
	fixed $x,y\in X$, the sum over $z\in R$ partitions the relevant
	compatible words according to their first letter, whereas the sum over
	$z\in X\setminus R$ partitions the corresponding compatible words
	according to their last letter.  The coefficients $t$ and $1-t$ are
	precisely the coordinates of the associated vector $\mathbf h_t$.
	Here is the formal proof:
	Let \. $\vf:=t \in \Vf^1$\.,  let \ts $x \in X$\ts, and let \ts $\bM:=\bM_v$ \ts be the associated matrix of $v$.
	By linearity, it suffices to prove that, for every $y \in X$,
	\begin{align*}
		\aM_{xy} \ = \  \big\< \bT^{\<x\>} \eb_y, \bM^{\<x\>} \,  \bT^{\<x\>} \hb \big\>.
	\end{align*}
%	Without loss of generality, assume that \ts $x \in R$,
%	as the proof of the other case is analogous.
%	We split the proof into two cases. First assume that $y \in R$. Then
Now we have
	\begin{align*}
		 &\big\< \bT^{\<x\>} \eb_y, \bM^{\<x\>} \,  \bT^{\<x\>} \hb \big\> \ = \   \big\< \eb_y, \bM^{\<x\>} \,   \hb \big\> \ = \   t \. \sum_{z \in R} \aM^{\<x\>}_{yz} \ +  \ (1-t) \. \sum_{z \in X-R} \aM^{\<x\>}_{yz}.
\end{align*}
Now note that, if either $x$ or $y$ is a loop of $\Mf$, then it follows from the definition~\eqref{eq:C-1}
that the sum above is $0$.
Also note that in this case we also have \ts $\aM_{xy}=0$ \ts by definition~\eqref{eq:C-1}.
Hence it suffices to consider the case when both $x$ and $y$ are non-loops of $\Mf$.
Now, continuing the equation above,
\begin{align*}
	\big\< \bT^{\<x\>} \eb_y, \bM^{\<x\>} \,  \bT^{\<x\>} \hb \big\>  \ 	 &=_{\eqref{eq:C-2}} \  t \. \sum_{z \in R}
		  \big|\big\{\gamma \. : \. z\gamma \in \Comp(\Mf/\{x,y\},R,a-1) \big\}\big| \\
		  &\qquad \ + \ (1-t) \.  \sum_{z \in X-R} \big|\big\{\gamma \. : \. \gamma z \in \Comp(\Mf/\{x,y\},R,a-2) \big\}\big| \\
		 &= \ t \. \big|\Comp(\Mf/\{x,y\},R,a-1)\big|  \ + \ (1-t)  \.  \big|\Comp(\Mf/\{x,y\},R,a-2)\big|  \\
		 &=_{\eqref{eq:C-1}} \  \aM_{xy}\..
	\end{align*}
%	as desired.
%	Now assume that $y \in S$. Then
%		\begin{align*}
%		&\big\< \bT^{\<x\>} \eb_y, \bM^{\<x\>} \,  \bT^{\<x\>} \hb \big\> \ = \   \big\< \eb_y, \bM^{\<x\>} \,   \hb \big\> \\
%		&  = \   t \. \sum_{z \in R} \aM^{\<x\>}_{yz} \ +  \ (1-t) \. \sum_{z \in S} \aM^{\<x\>}_{yz}   \\
%		&= \  t \. \sum_{z \in R} |\Comp(\Mf,R,a-2,xz,y)| \ + \ (1-t) \. t \. \sum_{z \in S} |\Comp(\Mf,R,a-2,x,zy)| \\
%		&= \ t \. |\Comp(\Mf,R,a-1,x,y)|  \ + \ (1-t)  \.  |\Comp(\Mf,R,a-2,x,y)|  \\
%		&= \  \aM_{xy},
%	\end{align*}
This completes the proof of \eqref{eq:Inh}.

We now verify \eqref{eq:TPInv}, which  reflects the fact that contracting several elements is
independent of the order in which they are selected.
It suffices to show that
\[  \aM^{\<x\>}_{yz} \ = \ \aM^{\<y\>}_{zx} \ = \  \aM^{\<z\>}_{xy}  \]
holds for all \ts $x,y,z \in X$.
Note that all three numbers are equal to $0$ if either one of $x$, $y$, or $z$ is a loop of $\Mf$,
so we assume that $x,y,z\in \NL(\Mf)$.  In this case, it follows from \eqref{eq:C-1} that
\begin{align*}
	\aM^{\<x\>}_{yz} \ = \ \aM^{\<y\>}_{zx} \ = \  \aM^{\<z\>}_{xy}  \ = \  \big|\Comp(\Mf/\{x,y,z\},R,a-2)\big|\ts.
\end{align*}
This completes the proof of  \eqref{eq:TPInv} and finishes the proof of the lemma.
\end{proof}

\smallskip

\begin{lemma}\label{lem:atlas-prop-2}
	Let $\Mf$ be a matroid of rank $r\geq 2$, let $R \subseteq X$,
and let \. $a \in \{2,\ldots,r-1\}$, such that \. $\aP(\Mf,R,a)>0$ \. and \. $\aP(\Mf,R,a-1)>0$.
	Then the atlas \. $\AA(\Mf,R,a)$ \. satisfies \eqref{eq:Decsupp}.
\end{lemma}

\begin{proof} The property
\eqref{eq:Decsupp}	follows from the elementary observation    that contraction may create additional loops,
	but cannot turn a loop of $\Mf$ into a non-loop.
	Here is the formal proof:
Let $v =t \in \Vf^1$.  In the notation above, we have:
$$\bM \. := \. \bM_v \. = \. t \. \bC(\Mf,R,a)
\. +  \. (1-t) \. \bC(\Mf,R,a-1).
$$
It follows from Lemma~\ref{lem:C-supp}, that the support of~$\bM$
is equal to the set \ts $\NL(\Mf)$ \ts of non-loop elements of~$\Mf$.
Let $x$ be an arbitrary element of $X$.
If $x$ is a loop of $\Mf$, then
\[  \supp(\bM^{\<x\>}) \ = \  \varnothing  \ \subseteq \  \supp(\bM).\]
If $x$ is not a loop of $\Mf$, then
\[  \supp(\bM^{\<x\>}) \ \subseteq \  \NL(\Mf/x)  \ \subseteq  \ \NL(\Mf) \ =  \supp(\bM),\]
as desired.
\end{proof}

\smallskip

\subsection{Hyperbolicity of the constructed atlas}\label{ss:property-atlas-Hyp}
The following proposition is the main technical result
we need in the proof of Theorem~\ref{t:main-positive}.
Also note that this proposition directly implies Stanley--Yan inequality (Theorem~\ref{t:SY}) as previously discussed in
Section~\ref{ss:main-atlas-construction}.

\smallskip

\begin{prop}\label{prop:C-hyp}
		Let $\Mf$ be a matroid of rank $r\geq 2$, let \ts $R \subseteq X$,
	and let \ts $a \in [r-1]$, such that \. $\aP(\Mf,R,a)>0$.
	Then the matrix  $\bC\.(\Mf,R,a)$ satisfies \eqref{eq:Hyp}.
\end{prop}

\begin{proof}
	We prove the claim by induction on the rank $r$ of $\Mf$.
	We start with the base case $r=2$.
	Note that this implies $a=1$.
	Write
\. $ \bigl(\aC_{x\ts y}\bigr)_{x,y \in X} := \bC(\Mf,R, a)$.
It then follows from \eqref{eq:C-1} that
\begin{align}\label{eq:baseC-1}
	\aC_{x,y} \ = \  \begin{cases}
	1 & \text{if }  x\neq y \. \text{ and } \. \{x,y\} \in \cI ,\\
	0 & \text{if } x=y \. \text{ or } \. \{x,y\} \notin \cI.
	\end{cases}
\end{align}
In particular, this shows that the $x$-row (respectively, $x$-column) of $\aC$ is identical to the
$y$-row (respectively, $y$-column) of $\aC$ whenever $x,y$ are non-loops in the same parallel class.
In this case, subtract the $y$-row and
$y$-column of $\aC$ by the $x$-row and $x$-column of $\aC$.
It then follows that the resulting matrix has $y$-row and $y$-column is equal to zero,
and note that \eqref{eq:Hyp} is preserved under this transformation.

Now, apply the above linear transformation repeatedly and remove the zero rows and columns,
and let $\aC'$ be the resulting matrix.  It suffices to prove that $\aC'$ satisfies \eqref{eq:Hyp}.
Note that $\aC'$ is a \ts $p \times p$ \ts matrix (where~$p$ is the number of parallel
classes of $\Mf$), with $0$s at the diagonal entries and $1$s as the non-diagonal entries.
It follows from direct calculations that $(p-1)$ is the only positive eigenvalue of~$\aC'$,
and it follows that $\aC'$ (and thus $\aC$) satisfies \eqref{eq:Hyp}.
This proves the base case of the induction.

We now assume that \ts $r \geq 3$, and that the claim holds for matroids of rank~$(r-1)$.
First, suppose that we have \. $\aP(\Mf,R,a+1)=\aP(\Mf,R,a-1)=0$.
Since $\aP(\Mf,R,a)>0$, the nonvanishing criterion in
Proposition~\ref{p:vanish} gives
\[
\rk_{\Mf}(R)=a
\qquad\text{and}\qquad
\rk_{\Mf}(X- R)=r-a.
\]
Consequently,
%\[
%\Mf=(\Mf|_R)\oplus(\Mf|_{X- R}).
%\]
 \ts $\Mf=\Mf_1\oplus \Mf_2$ \ts is a direct sum of matroids $\Mf_1$ and $\Mf_2\ts$,
where $\Mf_1$ (resp.~$\Mf_2$) is the matroid obtained from $\Mf$ by restricting the
ground set to $R$ (resp.~$X-R$).  It then follows that
\[ \aC_{x,y} \ = \  \begin{cases}
	\iB(\Mf_1/x) \. \iB(\Mf_2/y) & \text{if }  x \in \NL(\Mf) \cap  R \. \text{ and } \. y \in \NL(\Mf) \cap (X-R),\\
	0 & \text{otherwise}.
\end{cases}  \]

Rescale the rows and columns of
\ts $x \in R$ \ts by \ts $\iB(\Mf_1/x)^{-1}$\ts,
and the rows and columns of
\ts $y \in X-R$ \ts by \ts $\iB(\Mf_2/y)^{-1}$, and note that these rescalings preserve hyperbolicity.
Then \. $\bC$ \. becomes a special case of the matrix in \eqref{eq:baseC-1}, which was already shown to satisfy \eqref{eq:Hyp}.
%It then follows that
% \[ \aD_{x,y} \ = \  \begin{cases}
% 	1 & \text{ if } x\in R, y \in S,\\
% 	0 & \text{ if } x,y \in R \ \text{ or } \ x,y \in S.
% \end{cases}  \]
%It follows from direct calculation that the eigenvalues of $\bD$ is a single instance of $1$, a single instance of $-1$, and the rest are $0$'s.
So we can assume that either \. $\aP(\Mf,R,a+1)>0$ \. or \. $\aP(\Mf,R,a-1)>0$\..
By the symmetry, we can without loss of generality assume that \. $\aP(\Mf,R,a-1)>0$.
%Indeed, suppose that \. $\aP(\Mf,R,a+1)>0$\..
%Then  let  \ts $R' := S$\ts, \ts $S' := R$\ts, $a':= r-a$, and we can then apply the same analysis to \. $\Mf,R',a'$ \. instead (note that in this case \. $\aP(\Mf,R',a')>0$ \. and \. $\aP(\Mf,R',a'-1)>0$\.).
%We will show that both
%\. $\bC(\Mf,R,a+1,\varnothing,\varnothing)$ \. and

We split the proof into three parts.
First assume that $a \geq 2$.
Let \ts $\AA(\Mf,R,a)$ \ts be the atlas defined in \S\ref{ss:main-atlas-construction}.
 Its sink
matrices correspond to rank-$(r-1)$ contractions of $\Mf$, while the
non-sink matrices form the interpolation between $\bC(\Mf,R,a)$ and
$\bC(\Mf,R,a-1)$.
It follows from Lemma~\ref{lem:atlas-prop-1} and Lemma~\ref{lem:atlas-prop-2} (note that these lemmas require $a\geq 2$), that
this atlas satisfies \eqref{eq:Inh}, \eqref{eq:Iden}, \eqref{eq:TPInv}, and \eqref{eq:Decsupp}.
We now show that, for every sink vertex $\vf=x \in \Vf^0$,
the matrix $\bM_v$ satisfies \eqref{eq:Hyp}.
If $x$ is a loop of $\Mf$, then \. $\bM_v$ \. is equal to the zero matrix, which satisfies \eqref{eq:Hyp}.
If $x$ is a non-loop of $\Mf$,
then
 by definition \. $\bM_v$ \. is equal to  \. $\bC(\Mf/x,R,a-1)$\.. Also note that \. $\aP(\Mf/x,R,a-1)>0$ \. by Lemma~\ref{lem:C-center}. It then follows from the induction assumption that
 \. $\bM_v$ \.  satisfies \eqref{eq:Hyp}.
%Now let $x \in S$.
%Then by definition \. $\bM_v$ \. is equal to  \. $\bC(\Mf,R,a-1,\varnothing,x)$\..
%Now note that the restriction of this matrix to its support is equal to the matrix \. $\bC(\Mf/x,R,a-1,\varnothing,\varnothing)$\.. Also note that \. $\aP(\Mf/x,R,a-1)>0$ \. by Lemma~\ref{lem:C-center}. It then follows from the induction assumption that
%\. $\bM_v$ \.  satisfies \eqref{eq:Hyp}.

Now  every condition in Theorem~\ref{t:Hyp} has been verified in the paragraph above, so it follows that every non-sink regular vertex in $\Qf$ is hyperbolic.
  Moreover, for every $0<t<1$, the associated vector
$\mathbf h_t$ is strictly positive, while the associated matrix is
irreducible on its support by Lemma~\ref{lem:C-supp}.
Hence the vertex \ts $\vf:=t \in \Vf^1$ \ts is  regular for every \. $t \in (0,1)$\., and
Theorem~\ref{t:Hyp} implies that
\[ \bM_{\vf} \ = \  t \. \bC(\Mf,R,a)  \ +  \ (1-t) \. \bC(\Mf,R,a-1),   \]
satisfies \eqref{eq:Hyp}.
By taking the limit \ts $t \to 0$ \ts and \ts $t\to 1$\ts,  it then follows that \. $\bC(\Mf,R,a)$ \. and \. $\bC(\Mf,R,a-1)$ \. also satisfy \eqref{eq:Hyp}.

For the second case, assume that $a=1$ and  $\aP(\Mf,R,a+1)>0$.
Then let $a':=a+1=2$\.. Note that we have   \. $\aP(\Mf,R,a')>0$, \. $\aP(\Mf,R,a'-1) >  0$\..
By the same argument as before, we  conclude that
\. $\bC(\Mf,R,a) \. = \.  \bC(\Mf,R,a'-1)$ \.
satisfies \eqref{eq:Hyp}.

For the third case, assume that \. $a=1$ \. and \. $\aP(\Mf,R,a+1)=0$.
We would reduce this case to the preceding case, by constructing a matroid $\Mf'$ for which  $\bC(\Mf,R,1)$ and  $\bC(\Mf',R',1)$ have the same principal submatrix,
but with the additional property that \. $\aP(\Mf',R',2)>0$\..
Since \. $\aP(\Mf,R,1)>0$, there exists \. $A \in \cB(\Mf)$ \. such
that \ts $|A \cap R|=1$.  Since  \. $|A|=r\geq 2$, there exists $y \in X-R$
such that $y \in A$. Let $\Mf'$ be the matroid obtained by adding an element
$x'$ that is parallel to $y$, and let \. $R':=R +x'$.  Observe that $\Mf'$
has the same rank as $\Mf$.  Note also that \. $\aP(\Mf',R',2)>0$ \.
because \. $A':=A-y+x'$ \. is a basis of $\Mf'$ and satisfies \ts $|A'\cap R'|=2$.

Finally, note that  \ts $\bC(\Mf,R,1)$ \ts can be obtained from \ts
$\bC(\Mf',R',1)$ \ts by removing the row and column corresponding to~$x'$.
%Indeed, this is because, for every $x_1,x_2 \in X$, we have by definition that
%\begin{align*}
% \big(\bC(\Mf,R,1,\varnothing,\varnothing)\big)_{x_1,x_2}  \ &= \  \big|\Comp(\Mf,R,0,x_1,x_2)\big|  \\
% \ &= \  \big|\Comp(\Mf',R',0,x_1,x_2)\big| \ = \   \big(\bC(\Mf',R',1,\varnothing,\varnothing)\big)_{x_1,x_2}.
%\end{align*}
By the same argument as the second case, we conclude that
\ts $\bC(\Mf',R',1)$ \ts satisfies \eqref{eq:Hyp}.
Since \eqref{eq:Hyp} is a property that is preserved under restricting to principal submatrices,
it then follows that \ts $\bC(\Mf,R,1)$ \ts also satisfies \eqref{eq:Hyp}, and the proof is complete.
\end{proof}

\smallskip

\subsection{Proof of Theorem~\ref{t:main-positive}}\label{ss:main-positive-proof}
Having established the inequality, we now turn to its equality cases.
We will first prove Theorem~\ref{t:main-positive} under the assumption that the rank \ts $r=\rk(\Mf)=2$.
Recall that  \ts $\Par_{\Mf}(x)$ \ts denotes the set of elements  of $\Mf$ that are parallel to~$x$.
\smallskip

\begin{lemma}\label{lem:equ-base}
		Let $\Mf:=(X,\cI)$ be a matroid of rank $2$, and  let $R \subseteq X$
	such that
	\. $\aP(\Mf,R,1)>0$.
	Let $\as>0$ be a positive real number.
	Then
	\begin{equation}\label{eq:base-1}
	 {\aP(\Mf,R,2)} \ = \  \as \. \aP(\Mf,R,1) \ = \ \as^2 \. \aP(\Mf,R,0)
	\end{equation}
	\underline{if and only if}, for every non-loop $x$  of $\Mf$,
	\begin{equation}\label{eq:base-2}
	 |\Par_{\Mf}(x)  \cap R| \ = \  \as \. |\Par_{\Mf}(x)  \cap (X-R)|.
	\end{equation}
\end{lemma}

%\com{SH}{Note that $\as$ and $|S|$ has no connection. Should we change the notation of $\as$ to something else?}

\begin{proof}
	Let \. $\bM:=\bC(\Mf,R,1)$\., and recall that  \. $\fb,\gb \in \rr^n$ \. are the indicator vectors of $R$ and $X-R$ respectively.
	It follows from \eqref{eq:Cfg} that 	
%	Note that
%	\begin{align*}
%		\langle \fb , \bM \fb \rangle \ = \  2 \.\aP(\Mf,R,2), \qquad \langle \fb , \bM \gb \rangle \ = \  2 \.\aP(\Mf,R,1), \qquad  \langle \gb , \bM \gb \rangle \ = \  2 \.\aP(\Mf,R,0).
%	\end{align*}
 \eqref{eq:base-1} is equivalent to
\begin{equation}
	  \langle \fb , \bM \fb \rangle  \ = \  \as \, \langle \fb , \bM \gb \rangle \ = \  \as^2 \, \langle \gb , \bM \gb \rangle.
\end{equation}
i.e.\ \eqref{eq:sEqu} holds.
It then follows from Lemma~\ref{l:equality kernel} that \eqref{eq:base-1} is equivalent to \. $\zb:=\fb-\as \gb$ \. is contained in the kernel of $\bM$.
On the other hand, the matrix $\bM$ is described by
\begin{align*}
	\aM_{x,y} \ = \
	\begin{cases}
	1 & \text{if }  x\neq y \. \text{ and } \. \{x,y\} \in \cI ,\\
0 & \text{if } x=y \. \text{ or } \. \{x,y\} \notin \cI.
	\end{cases}
\end{align*}
It then follows that the kernel of $\bM$ is the set of vectors $\vb \in \rr^n$ such that, for every non-loop $x$ of $\Mf$,
\begin{equation}\label{eq:referee-1}
 \sum_{y \notin \Par_{\Mf}(x)} \av_y \ = \ 0.
\end{equation}
Summed over all distinct parallel classes of $\Mf$, we conclude that
\begin{equation}\label{eq:referee-2}
(p-1) \sum_{y \in X} \av_y \ = \ 0,
\end{equation}
where $p \geq r \geq 2$ is the number of parallel classes of $\Mf$.
Subtracting  \eqref{eq:referee-1} from $\tfrac{1}{p-1}$ times \eqref{eq:referee-2}, this implies
\begin{equation*}
	\sum_{y \in \Par_{\Mf}(x)} \av_y \ = \ 0.
\end{equation*}
Substituting \ts $\vb \gets \zb$\ts, the equation above is equivalent to
	\[ |\Par_{\Mf}(x)  \cap R| \ - \  \as \. |\Par_{\Mf}(x)  \cap (X-R)| \ = \ 0, \]
and the lemma follows.
\end{proof}

\smallskip

We now give an intermediate lemma which takes us halfway towards
Theorem~\ref{t:main-positive}.
The following lemma is the one-step equality-propagation mechanism.  It
shows that equality for $\Mf$ is equivalent to equality, with the same
constant $s$, for all contractions $\Mf/x$ with
$x\in R\cap\NL(\Mf)$.

\smallskip

\begin{lemma}\label{lem:tran}
	Let \ts $\Mf:=(X,\cI)$ \ts be a matroid of rank \. $r\geq 3$, let \ts $R \subseteq X$,
and let \ts $a \in \{2,\ldots, r-1\}$, such that  \. $\aP(\Mf,R,a)>0$.
Finally, let $\as>0$. Then
\begin{equation}\label{eq:tran-1}
	 {\aP(\Mf,R,a+1)} \ = \  \as \. \aP(\Mf,R,a) \ = \ \as^2 \. \aP(\Mf,R,a-1)
\end{equation}
holds \. \underline{if and only if} \. for every  \. $x \in R \cap \NL(\Mf)$, we have:
\begin{equation}\label{eq:tran-2}
	 {\aP(\Mf/x,R,a)} \ = \  \as \. \aP(\Mf/x,R,a-1) \ = \ \as^2 \. \aP(\Mf/x,R,a-2) >0.
\end{equation}
\end{lemma}

\begin{proof}
We first prove the \. $\Leftarrow$ \. direction.
Note that
 \begin{align*}
 	\aP(\Mf,R,a+1) \ &= \
 	\tbinom{r}{a+1}^{-1} \. \big|\big\{ B \in \cB(\Mf) \. : \.  |B \cap R|=a+1 \big\}\big|,\\
 	&= \ \tbinom{r}{a+1}^{-1} \. \tfrac{1}{a+1} \.  \sum_{x \in R  \. \cap \. \NL(\Mf)} \big|\{ B \in \cB(\Mf) \. : \.  |B \cap R|=a+1, x \in B \}\big|\\
 	&= \ \tbinom{r}{a+1}^{-1} \. \tfrac{1}{a+1} \.  \sum_{x \in R  \. \cap \. \NL(\Mf)} \big|\{ B' \in \cB(\Mf/x) \. : \.  |B' \cap R|=a \}\big|\\
 	\ &= \   \frac{1}{r} \. \sum_{x \in R  \. \cap \. \NL(\Mf)} \aP(\Mf/x,R,a).
 \end{align*}
 Applying \eqref{eq:tran-2} to the RHS, we get
 \[	\aP(\Mf,R,a+1) \  = \  \frac{1}{r} \. \sum_{x \in R  \. \cap \. \NL(\Mf)}  \as  \aP(\Mf/x,R,a-1) \ = \  \as  \aP(\Mf,R,a). \]
By the same argument, we also get \. $\aP(\Mf,R,a)  =  \as  \aP(\Mf,R,a-1)$, as desired.

We now prove the \. $\Rightarrow$ \. direction.
Let  \. $\AA(\Mf,R,a)$ \. be the combinatorial atlas defined in \S\ref{ss:main-atlas-construction}.
It follows from the assumption, that \. $\aP(\Mf,R,a)>0$ \. and \. $\aP(\Mf,R,a-1)>0$\..
  It follows from Lemma~\ref{lem:atlas-prop-1} and Lemma~\ref{lem:atlas-prop-2} that
this atlas satisfies \eqref{eq:Inh}, \eqref{eq:Iden}, \eqref{eq:TPInv}, and \eqref{eq:Decsupp}.
It also follows from Proposition~\ref{prop:C-hyp}, that this atlas satisfies \eqref{eq:Hyp}.
Recall that $\fb,\gb \in \rr^n$ is the indicator vector of the subset $R$ and $X-R$, respectively.
It follows from the definition that $\fb,\gb$ form a global pair for this atlas, and that the edge  \. $(1,x)$ \.  is a functional edge for every $x \in R$.

 Let \. $\bM:=\bC(\Mf,R,a)$ \. be the matrix associated to \ts $\vf=1 \in \Vf^1$.
It follows from \eqref{eq:Cfg} that  \eqref{eq:tran-1} is  equivalent to  $\vf=1$ satisfying  \eqref{eq:sEqu}.
At this vertex, we have $\hb_1=\fb$.
Moreover, every edge transformation is the identity map.  Hence
$v=1$ is a functional source, and its functional outgoing edges are
precisely those indexed by
$x\in R\cap\NL(\Mf)$.
By Theorem~\ref{t:equality Hyp}, every functional target $x\in R\cap\NL(\Mf)$ also satisfies \eqref{eq:sEqu}
with the same constant~$\as$. In other words, for every $x\in R\cap\NL(\Mf)$, we have:
\begin{equation}
	\langle \fb , \bM^{\<x\>} \fb \rangle  \ = \  \as \. \langle \fb , \bM^{\<x\>} \gb \rangle \ = \  \as^2 \ts \langle \gb , \bM^{\<x\>} \gb \rangle.
\end{equation}
On the other hand, for every $x \in R\cap\NL(\Mf)$,
we have  \. $\bM^{\<x\>}$ \. is equal to \. $\bC(\Mf/x,R,a-1)$ \. by definition, so it follows from \eqref{eq:Cfg}  that
\begin{align*}
	\langle \fb , \bM^{\<x\>} \fb \rangle \ &= \    (r-1)! \. \aP(\Mf/x,R,a), \qquad 	\langle \fb , \bM^{\<x\>} \gb \rangle \ = \  (r-1)! \. \aP(\Mf/x,R,a-1),\\
			\langle \gb , \bM^{\<x\>} \gb \rangle \ &= \    (r-1)! \. \aP(\Mf/x,R,a-2).
\end{align*}
Finally  note that  \. $\aP(\Mf/x,R,a-1)>0$ \. by Lemma~\ref{lem:C-center}.
Since $\as>0$ and the three quantities satisfy \eqref{eq:tran-2}, the
other two quantities are positive as well.
This completes the proof.
\end{proof}

\smallskip

\begin{proof}[Proof of Theorem~\ref{t:main-positive}]
Note that \eqref{eq:comb-1} is equivalent to
\begin{equation}\label{eq:comb-3}
 {\aP(\Mf,R,a+1)} \ = \  \as \. \aP(\Mf,R,a) \ = \ \as^2 \. \aP(\Mf,R,a-1) \ > \ 0,
\end{equation}
for some positive $\as >0$.
It then suffices to show that \eqref{eq:comb-3} is equivalent to \eqref{eq:comb-2} with the same $\as >0$.

We now iterate the preceding lemma in two stages.
For the first stage,
we iterate   Lemma~\ref{lem:tran} for $a-1$ many times and contract
$a-1$ elements of $R$.
It then follows  that \eqref{eq:comb-3} is equivalent to
\begin{equation}\label{eq:comb-4}
	 {\aP(\Mf/A,R,2)} \ = \  \as \ts \aP(\Mf/A,R,1) \ = \ \as^2 \ts \aP(\Mf/A,R,0) \ > \ 0,
\end{equation}
for every \ts $A \subseteq R$ \ts that is independent in~$\Mf$, and such that \ts $|A|=a-1$.
Now, rewriting \eqref{eq:comb-4} with $X-R$ as the distinguished subset,
we have that  \eqref{eq:comb-4}  is equivalent to
\begin{equation}\label{eq:comb-5}
	{\aP(\Mf/A,X-R,r-a-1)} \ = \  \as \ts \aP(\Mf/A,X-R,r-a) \ = \ \as^2 \ts \aP(\Mf/A,X-R,r-a+1) \ > \ 0,
\end{equation}
for every $A \subseteq R$ that is independent in $\Mf$ and such that \ts $|A|=a-1$.

We now perform the second stage
by iterating  Lemma~\ref{lem:tran} for  $r-a-1$ many times and contract $r-a-1$ elements of $X-R$.
 It then follows that \eqref{eq:comb-5} is equivalent to
\begin{equation}\label{eq:comb-6}
	{\aP(\Mf/B,X-R,0)} \ = \  \as \ts \aP(\Mf/B,X-R,1) \ = \ \as^2 \ts \aP(\Mf/B,X-R,2) >0,
\end{equation}
for every $B \subseteq X$ that is independent in $\Mf$ and such that \ts $|B|=r-2$ \ts and \ts $|B\cap R|=a-1$\ts.
Noting that $\Mf/B$ is a matroid of rank 2 and interchanging  the roles of $R$ and $X-R$, it then follows that  \eqref{eq:comb-6} is equivalent to
\begin{equation}\label{eq:comb-7}
	{\aP(\Mf/B,R,2)} \ = \  \as \ts \aP(\Mf/B,R,1) \ = \ \as^2 \ts \aP(\Mf/B,R,0) >0,
\end{equation}
for every $B \subseteq X$ that is independent in $\Mf$ and such that \ts $|B|=r-2$ \ts and \ts $|B\cap R|=a-1$\ts.
The  theorem now follows by applying Lemma~\ref{lem:equ-base} to \eqref{eq:comb-7}.
\end{proof}

% \smallskip
%
%
% \begin{proof}[Proof of Proposition~\ref{p:vanish}]
% 	The \. $\Rightarrow$ \. direction is straightforward. For the 	 \. $\Leftarrow$ \. direction,
% 	let $A$ be an independent set of size $\rk(R)$. % (note that $A$ exists by the definition of the rank function).
% 		By the exchange property, there exists a basis $B\in \cB(\Mf)$ that contains~$A$.
% 	It follows that \. $|B\cap R|=|A|=\rk(R)$.
%	By an analogous argument, there exists a basis $B'$ of~$\Mf$,
%    such that \. $|B' \cap R|=r-\rk(X-R)$.
%	Now let $a$ be any integer satisfying the condition in the lemma.
%	Applying the bases exchange property consecutively to $B$ and $B'$,
%we get a basis $B''$ of $\Mf$, such that \. $|B'' \cap R|=a$, as desired.
% \end{proof}

\medskip

\section{Total equality cases} \label{s:total}

In this section we present the proof of Corollary~\ref{c:total-Stanley} and Theorem~\ref{t:Yan-conj}.

\medskip

\subsection{Proof of Corollary~\ref{c:total-Stanley}}  \label{ss:total-cor}
To simplify the notation, denote \, $\psc_a := \aP(\Mf,R,a)$.
By Proposition~\ref{p:vanish}, we have \ts $\psc_a>0$ \ts for all \ts $0\le a \le r$.
Writing \eqref{eq:SY} for all \ts $1\le a < r$, we get:
\begin{equation}\label{eq:cons-1}
\frac{\psc_1}{\psc_0} \  \geq \   \frac{\psc_2}{\psc_1}       \ \geq \  \frac{\psc_3}{\psc_2} \ \geq \ \cdots \ \geq \ \frac{\psc_{r}}{\psc_{r-1}} \ > \ 0 \ts.
\end{equation}
This gives:
	 \begin{align}\label{eq:cons-2}
	 	\left(\frac{\psc_1}{\psc_0}\right)^r \ \geq \   \frac{\psc_1 \cdot \psc_2 \. \cdots \. \psc_r}{\psc_0\cdot \psc_1 \. \cdots \. \psc_{r-1} }  \ = \   \frac{\psc_r}{\psc_0} \.,
	 \end{align}
and proves the first part.  For the second part, observe that all inequalities
in \eqref{eq:cons-1} must be equalities, which implies the second part.  \qed

\smallskip

\subsection{Proof of Theorem~\ref{t:Yan-conj}, part~$(1)$}
As described in the introduction, it remains to show the
\. {\small $(ii)$} \ts  $\Rightarrow$ \ts {\small $(iv)$} \. implication.
It follows from Theorem~\ref{t:main-positive}, that
		\begin{equation}\label{eq:2.8-1}
		|\Par_{\Mf/A}(x)  \cap R| \ = \  \as  |\Par_{\Mf/A}(x)  \cap (X-R)|,
	\end{equation}	
	for every independent set \ts $A \in \cI$ \ts of size
	\ts $|A| =  r-2$,  and every \ts $x \in \NL(\Mf/A)$.
	It thus suffices to show that \eqref{eq:2.8-1} implies {\small $(iv)$} for the same value of $\as$.

We use induction on $r$.
For $r=2$, the claim follows immediately by applying \eqref{eq:2.8-1},
since $\Mf$ is loopless and we must have \ts $A=\emp$ \ts in this case.

For $r>2$, suppose the claim holds for all matroids of rank $(r-1)$.
Then, for all $x \in X$, it follows from applying the claim to \ts $\Mf/x$,
that
			\begin{equation}\label{eq:2.8-2}
		|\Par_{\Mf/x}(y)  \cap R| \ = \  \as \. |\Par_{\Mf/x}(y)  \cap (X-R)|,
	\end{equation}	
	for \. $y \in \NL(\Mf/x)$.
%	It follows from \eqref{eq:2.8-1} that \. $\as_x=\as$\..
	In particular, it then follows from \eqref{eq:2.8-2} that
	\begin{equation}\label{eq:2.8-3}
	 |\NL(\Mf/x) \cap R| \ = \ \as \.   |\NL(\Mf/x) \cap (X-R)|.
	\end{equation}
	On the other hand, it follows
	from \ts $\NL(\Mf)=X$, that
	\begin{equation}\label{eq:2.8-4}
	   \NL(\Mf/x) \, = \,  X  \setminus \Par_\Mf(x).
	\end{equation}
By combining \eqref{eq:2.8-3} and \eqref{eq:2.8-4}, we conclude:
\begin{equation}\label{eq:2.8-5}
 |\Par_\Mf(x) \cap R| \. - \. \as  |\Par_\Mf(x) \cap (X-R)| \ = \  |R| \. - \. \as |X-R|,
\end{equation}
for all $x \in X$.  Summing the equation above over all parallel classes of~$\Mf$, we then have
\[   |R| \. - \. \as |X-R| \ = \  p \. \big( |R| \. - \. \as  |X-R| \big), \]
where \ts $p\geq r \geq 3$ \ts is the number of parallel classes of~$\Mf$.
This implies \. $|R|= \as |X-R|$.
Together with \eqref{eq:2.8-5} this gives
\[   |\Par_\Mf(x) \cap R| \, = \, \as   |\Par_\Mf(x) \cap (X-R)|  \quad \text{ for all } \quad x \in X,\]
as desired. \qed

\medskip

\section{Examples and counterexamples} \label{s:ex}
In this section, we give three constructions illustrating the equality
conditions in Theorem~\ref{t:main-positive} and their relation to the total equality
conditions in Theorem~\ref{t:total-SY}.  The first construction gives a broad family
of matroids satisfying total equality.  The second gives another case of equality  through a different mechanism, which  does not satisfy the
total equality condition in Theorem~\ref{t:total-SY}.  The third combines the first
two constructions to produce a  matroid for which
condition~{\small $(iii)$} \ts of Theorem~\ref{t:total-SY} holds while
condition~{\small $(ii)$} \ts fails.

\subsection{Double matroid} \label{ss:ex-double}
Let \ts $\Mf$ \ts be a loopless matroid with a ground set~$X$ that is
given by a representation \ts $\phi: X\to \rr^r$.  Denote by $X'$ the second
copy of~$X$.  Define a matroid  \ts
$\Mf^2$ \ts with the ground set \ts
$Y := X\sqcup X'$, that is given by a representation \ts
$\psi: Y \to \rr^r$, where \ts
$\psi(x):= \phi(x)$ \ts and \ts  $\psi(x'):= -\phi(x)$.

Now let \ts $R\gets X$ \ts and \ts $S\gets X'$.  By the symmetry, observe that
%\begin{equation}\label{eq:par-1}
\begin{equation}\label{eq:ex-par}
	|\Par_{\Mf^2}(y)  \cap R| \ = \  |\Par_{\Mf^2}(y)  \cap S|,
\end{equation}
for all \ts $y \in Y$.  In other words, matroid $\Mf^2$ satisfies
condition \ts {\small $(iv)$} \ts in Theorem~\ref{t:total-SY} with \ts $\as=1$.
By Theorem~\ref{t:Yan-conj}, we conclude that \ts $\Mf^2$ \ts has total equality
in the Stanley--Yan inequality, i.e.\ condition \ts {\small $(ii)$} \ts in
Theorem~\ref{t:total-SY}.

Note that \eqref{eq:ex-par} is a special case of \eqref{eq:comb-2} with \ts $\as=1$.
In fact, it is easy to modify this example to make the ratio to be any rational number.
Indeed, to get the ratio \ts $\as = a/b$, take $a$ copies of~$R$ and $b$ of~$S$.
One can make all vectors to be distinct by taking different multiples (over~$\rr$).
We omit the easy details.

\smallskip

\subsection{Linear matroid} \label{ss:ex-linear}
Fix \ts $r\ge 3$.
Let \ts $X=\ff_2^r$ \ts and let \ts $\Mf$ \ts be a
binary matroid with a ground set~$\ts X$ \ts in its natural representation.
Let \ts $R\ssu X$ \ts be a subspace of dimension \ts $(r-1)$.
Note that $\0$ is the only loop in~$\Mf$.

Take an independent set of vectors \ts $A \subset X$ \ts such that
\. $|A| =  r-2$ \. and \.   $|A \cap R| =  0$.
Since \ts $A\ne \emp$, it is easy to see that for every non-loop \ts
$x \in \NL(\Mf/A)$, we have:
\begin{equation}\label{eq:par-1}
	|\Par_{\Mf/A}(x)  \cap R| \ = \   |\Par_{\Mf/A}(x)  \cap (X-R)|\ts.
\end{equation}	
This is \eqref{eq:comb-2} for \ts $a=1$, with \ts $\as=1$.
By Theorem~\ref{t:main-positive}, this implies that \eqref{eq:comb-1}
also holds for $a=1$.

Finally, note that \ts $\aP(\Mf,R,0)=0$ \ts in this case, which is why
this is not an example of total equality condition \ts {\small $(ii)$} \ts in
Theorem~\ref{t:total-SY} (and why the theorem is inapplicable in any event).

\smallskip

\subsection{Combination matroid} \label{ss:ex-combin}
The previous two examples illustrate different reasons for the equality
\eqref{eq:comb-1} to hold for $a=1$.  The following matroid is a combination
of the two which still gives equality for $a=1$, but not for $a>1$.

Fix \ts $r\ge 3$ \ts and let \ts $V= \ff_2^r$\ts.
Let \ts $R_0\ssu V$ \ts be a subspace of dimension \ts $(r-1)$,
let \ts $S_0 := V\sm R_0$ \ts be the complement.
Let \ts $R_1, S_1\ssu S_0$ \ts be two copies of the same nonempty
set of vectors.  Finally, let \ts $\Mf$ \ts be a matroid
on the ground set \ts $X:= R_0 \sqcup R_1 \sqcup S_0 \sqcup S_1$ \ts and let
\ts $R:=R_0 \sqcup R_1\ts$.
Note that \ts $X -R=S_0 \sqcup S_1$\ts.

Clearly, \ts $\rk(R) = \rk(X\sm R)=r$, so \ts $\aP(\Mf,R,0)>0$ \ts and \ts
\ts $\aP(\Mf,R,r)>0$.  We have \eqref{eq:par-1} by a direct computation.
By Theorem~\ref{t:main-positive}, this again implies that \eqref{eq:comb-1}
holds for $a=1$.  On the other hand,  one
can directly check that \eqref{eq:comb-2} (and thus \eqref{eq:comb-1}) does not hold for \ts $a=r-1$.
We omit the details.

In summary, this gives an example when total equality condition
\ts {\small $(ii)$} \ts in Theorem~\ref{t:total-SY} fails, even though
\ts {\small $(iii)$} \ts holds.  This disproves Conjecture~\ref{conj:Yan}
and proves the second part of Theorem~\ref{t:Yan-conj}.\footnote{To be precise,
matroid~$\Mf$ is not loopless.  To fix this, remove $\0$ from~$R_0$.}

\medskip

\section{Generalized Mason inequality} \label{s:indep}
In this section, we prove a generalized Mason inequality for the number of
independent sets satisfying prescribed intersection-size constraints.

Let \ts $\Mf$ \ts be a matroid of rank \ts $r=\rk(\Mf)$, with a ground set~$X$
of size \ts $|X|=n$.  Denote by \ts $\cI(\Mf)\subseteq 2^X$ \ts the set of
independent sets of~$\Mf$.  Fix integers \ts $k \geq 0$ \ts and
\. $0\le a, c_1,\ldots,c_k \le r$. Additionally, fix disjoint subsets \ts
$S_1,\ldots, S_k \ssu X$, and let \ts $R:= X \sm \cup_i  S_i$\ts.
%
% $\bS:=(S_1,\ldots, S_k)$ \ts be disjoint subsets of $X-R$, and let \ts $\cb:=(c_1,\ldots, c_k) \in \nn^k$\ts.
Define
\[
\cI_{\bSc}(\Mf,a)  \ := \ \big\{ \ts A \in \cI(\Mf)  \, : \,  |A \cap R|=a, \.  |A \cap S_1|=c_1, \. \ldots \., \. |A \cap S_k|=c_k  \ts \big\}\ts,
\]
and let \. $\aI_{\bSc}(\Mf,a):=|\cI_{\bSc}(\Mf,a)|$.  Let \ts $m:=n-c_1-\ldots-c_k$ \ts
and denote by \ts $\mathscr{F}_m$ \ts a free matroid on~$m$ elements.
Substituting the direct sum \. $\Mf \. \gets \. \Mf \oplus \mathscr{F}_m$ \. into the
Stanley--Yan inequality, we obtain a log-concave inequality:
\[
{\aI_{\bSc}(\Mf,a)}^2 \ \geq \   \aI_{\bSc}(\Mf,a+1) \, \aI_{\bSc}(\Mf,a-1)\ts.
\]
This is the argument that was used by Stanley in \cite[Thm~2.9]{Sta-AF}, to obtain
Mason's log-concave inequality \eqref{eq:Mason-1} in the case \ts $k=0$.
The following ultra-log-concave inequality is a natural extension.

\smallskip

\begin{thm}[{\rm \defn{generalized Mason inequality}}{}]\label{t:indep}
	For all \ts $ 1\le a\le \min\{r-1,m-1\}$, we have:
\begin{equation}\label{eq:gen-Mason}
{\aIr_{\bSc}(\Mf,a)}^2 \ \geq \  \big(1+\tfrac{1}{a}\big) \ts
\big(1+\tfrac{1}{m-a}\big) \,  \aIr_{\bSc}(\Mf,a+1) \, \aIr_{\bSc}(\Mf,a-1).
\end{equation}	
\end{thm}

\smallskip

This inequality follows from results proved independently by Anari,
Liu, Oveis Gharan, and Vinzant~\cite{ALOV-III}, and by Br\"and\'en and
Huh~\cite{BH}.  We include a short proof using the
Lorentzian-polynomial formulation of~\cite{BH}.

%This inequality is an easy consequence of the % Lorentzian polynomials technology
%results by Br\"and\'en and Huh~\cite{BH}.  We include a short proof for completeness.

\smallskip

\begin{proof}[Proof of Theorem~\ref{t:indep}]
We assume that \ts $X=[n]$.
Let \. $f_{\Mf} \in \nn[w_0,w_1,\ldots, w_{n}]$ \. be a multivariate polynomial
defined by
\[f_\Mf(w_0, w_1,\ldots, w_n) \ := \ \sum_{A \in \cI(\Mf)} \ts  w_{0}^{n-|A|} \. \prod_{i\in A} \. w_i\..
\]
%where \ts $\bw^A$
It is shown in \cite[Thm~4.10]{BH}, that \. $f_{\Mf}$ \. is Lorentzian.

% Let $x,y,z_1,\ldots, z_k$ be new variables,
% and we substitute the variables \. $w_0,\ldots, w_n$  \. in \. $f_\Mf$ \.  by the rule
Take the following substitution:  \. $w_0\gets y$, \. $w_i\gets x$ \ts for \ts $i\in R$, and
\. $w_i\gets z_j$ \ts for \ts $i\in S_j$\ts, \. $1\le j \le k$.
Let \. $g_{\Mf}(x,y,z_1,\ldots, z_k)$ \. be the resulting polynomial, and let
%Let $h_{\Mf}:=h_{\Mf}(x,y)$ be the polynomial defined by
\[ h_{\Mf}(x,y) \ := \ \frac{\partial^{c_1 \ts + \. \ldots \. + \ts c_k}}{\partial z_1^{c_1} \. \cdots \. \partial z_k^{c_k}} \. \Bigr|_{z_1,\ldots, z_k=0} \ g_{\Mf}(x,y,z_1,\ldots, z_k)\ts.   \]
Since the Lorentzian property is preserved under diagonalization, taking directional derivatives, and zero substitutions, see
% Thm~2.10 and Cor.~2.11 in
\cite[$\S$2.1]{BH}, it follows  that \ts $h_{\Mf}(x,y)$ \ts is  a Lorentzian polynomial with degree $m$.

Now note that the coefficients \. $[x^ay^{m-a}] \ts h_\Mf(x,y)$ \. is equal to  \. $\aI_{\bSc}(\Mf,a) \. c_1! \cdots c_k! $ \. by definition.
Recall now that a bivariate homogeneous polynomial
with nonnegative coefficients is Lorentzian if and only if the sequence of coefficients form an
ultra-log-concave sequence with no internal zeros.  This implies the result.
\end{proof}

\medskip

\section{Final remarks and open problems} \label{s:finrem}

\subsection{Computational complexity ideas}  \label{ss:finrem-quote}
Looking into recent developments, one cannot help but admire
Rota's prescience and keen understanding of mathematical development:

\smallskip

\begin{center}\begin{minipage}{13.cm}%
{{\em ``Anyone who has worked with matroids has come away with the conviction
that the notion of a matroid is one of the richest and most useful concepts
of our day.  Yet, we long, as we always do, for one idea that will allow us
to see through the plethora of disparate points of view.''} \cite{Rota}
}
\end{minipage}\end{center}

\smallskip

\nin
Arguably, the idea of \defna{hyperbolicity} \ts is what unites both the
combinatorial Hodge theory, Lorentzian polynomials and the combinatorial
atlas approaches, even if technical details vary considerably.  On the other
hand, our complexity theoretic approach is as ``disparate'' as one could
imagine, leaving many mathematical and philosophical questions
unanswered.\footnote{Some of these questions related to the nature
of the $\poly$ vs.\ $\NP$ problem are addressed in \cite[$\S\S$1--4]{Aar16}.}

That an open problem in the old school matroid theory was resolved using tools
and ideas from computational complexity might be very surprising
to anyone who had not seen theoretical computer science permeate
even the most distant corners of mathematics.  To those finding
themselves in this predicament, we recommend a recent survey \cite{Wig23},
followed by richly detailed monograph \cite{Wig19}.

\smallskip

\subsection{Negative results for other matroids}  \label{ss:finrem-binary}
One can ask if Theorem~\ref{t:main-negative} extends beyond binary matroids to other families
of matroids given by a succinct presentation.  In fact, our proof is
robust enough, and extends to every family of matroids which satisfies the
following:

{\small $(1)$} \ts computing the number of bases is \ts $\SP$-complete, and

{\small $(2)$} \ts the family includes all (loopless, bridgeless) graphical matroids.

\nin
Notably, matroids realizable over~$\zz$ obviously satisfy {\small $(2)$},
and satisfy {\small $(1)$} by \cite{Snook}.  On the other hand, paving matroids
based on Hamiltonian cycles considered in \cite[$\S$3]{Jer}, easily satisfy
{\small $(1)$}, but are very far from {\small $(2)$}.

For \emph{bicircular matroids}, property {\small $(1)$} was proved in \cite{GN06}.
Unfortunately, not all graphical matroids are bicircular matroids,
see~\cite{Mat77}.  In fact, not all graphical matroids are necessarily
\emph{transversal} (see e.g.\ \cite[Ex.~1.6.3]{Ox}), or even a \emph{gammoid}
(see e.g.\ \cite[Exc.~11(ii)~in~$\S$12.3]{Ox}).
Nevertheless, we believe the following:

\begin{conj}\label{conj:bicircular}
Theorem~\ref{t:main-negative} holds for bicircular matroids.
\end{conj}

Note that not all bicircular matroids are binary, see \cite[Cor.~5.1]{Zas87},
so the conjecture would not imply Theorem~\ref{t:main-negative}.  If one is
to follow the approach in this paper, a starting point would be Conjecture~6.1
in \cite{CP-coinc} which is analogous to Theorem~\ref{t:st} in this case,
and needs to be obtained first.  Afterwards, perhaps the proof can be
extended to basis ratios as in Lemma~\ref{lem:st-ratio}.

\smallskip

\subsection{Generalized Mason inequality}  \label{ss:finrem-Mason}
Denote by \ts $\EMason_k$ \ts the equality of \eqref{eq:gen-Mason}
decision problem.  As we mentioned earlier, \ts \ts $\EMason_0$ \ts
is in \ts $\coNP$.  In fact, it is \ts $\coNP$-complete, see
Corollary~\ref{cor:EMason} below.
By analogy with Theorem~\ref{t:main-negative}, it would be interesting
to see what happens for general~$k:$

\begin{op} \label{op:gen-Mason}
For what \ts $k>0$,  is \ts $\EMason_k$ \ts  in \ts $\PH$?
\end{op}

In particular, any explicit description of equality cases for
 \ts $\EMason_1$ \ts would be a large step forward and potentially
 very difficult.  We note aside that the combinatorial atlas approach
 can also be used to prove \eqref{eq:gen-Mason}.  Unfortunately,
the specific construction we have in mind cannot be used to describe
the equality cases, at least not without major changes.

\smallskip

In a different direction, it is natural to ask whether the constants in the
generalized Mason inequality~\eqref{eq:gen-Mason} can be improved.

\smallskip

\begin{question}
	Let $1 \leq a \leq |R|-1$.
	Is it true that
	\begin{equation}\label{eq:CPM}
	  \aIr_{\bSc}(\Mf,a)^2 \ \geq \  \big(1+\tfrac{1}{a}\big) \ts
\big(1+\tfrac{1}{|R|-a}\big) \, \aIr_{\bSc}(\Mf,a+1) \, \aIr_{\bSc}(\Mf,a-1)?
	\end{equation}
\end{question}

\smallskip
Notice that \eqref{eq:CPM} would be a strengthening of the generalized Mason's inequality~\eqref{eq:gen-Mason}, because $m=|R|+ \sum_{i} (|S_i|-c_i) \ \geq \ |R|$.
The proposed formulation is motivated by a formal duality arising from a
suitable change of parameters, which suggests replacing the quantity $m$
 by $|R|$.
This duality does not, however, provide strong evidence that such a
strengthening should hold.
 Thus,
one possible direction is to search for counterexamples among general
matroids, while another is to seek a proof for more structured classes,
beginning perhaps with graphical matroids.

\subsection{Completeness}\label{ss:finrem-NP}
As evident from this paper, the computational complexity of equality
cases for matroid inequalities is very interesting and remains
largely unexplored.  Of course, for Mason's log-concave inequality
\eqref{eq:Mason-1}, the equality cases are trivial since the
sequence satisfies a stronger inequality \eqref{eq:Mason-2}.
On the other hand, for Mason's ultra-log-concave inequality \eqref{eq:Mason-2},
the equality cases have a simple combinatorial description:
$\girth(\Mf) > a+1$, i.e.\ the size of the minimal circuit in the
matroid has to be at least~$a+2$, see \cite{MNY21} and \cite[$\S$1.6]{CP}
for  proofs using Lorentzian polynomials and combinatorial atlases, respectively.

Now, the decision problem \ts {\sc GIRTH} \. $:= \. \big\{\girth(\Mf) \le^? a+1\big\}$ \.
% decision problem
is in \ts $\NP$ \ts for matroids with concise presentation.  The problem
is easily in~$\poly$ for graphical matroids via taking powers of adjacency
matrix.  Recently, it was shown to be in~$\poly$ for regular matroids
in~\cite{FGLS18}.  Famously, the problem was shown
to be $\NP$-complete for binary matroids by Vardy~\cite{Var97}.
This gives:

\begin{cor}\label{cor:EMason}
$\EMason_0$ \ts is \ts $\coNP$-complete for binary matroids.
\end{cor}

We believe that the upper bound in Corollary~\ref{cor:k=0} is optimal:

\begin{conj} \label{conj:SY-coNP}
$\EBULC_0$ \. is \. $\coNP$-complete for binary and for bicircular matroids.	
\end{conj}

Note that for graphical  matroids, the number of
bases \ts $\iB(\Mf)$ \ts is in~$\FP$ by the \emph{matrix-tree theorem}.
For regular matroids, the same linear algebraic argument applies.
More generally, the number \ts $\iB_{\bSc}(\Mf,R,a)$ \ts can also be computed in
polynomial time via the weighted (multivariate) version of the matrix-tree
theorem (see e.g.\ \cite{GJ83}).  This gives the following observation to
contrast with Conjecture~\ref{conj:SY-coNP}.

\begin{prop} \label{p:SY-P}
$\EBULC_k$ \. is in \. $\poly$ \. for regular matroids and all fixed \ts $k\ge 0$.	
\end{prop}

We conclude with a possible approach to the proof of Conjecture~\ref{conj:SY-coNP}.
The \ts {\sc 2-SPANNING-CIRCUIT} \ts is a problem whether a matroid has a
circuit containing two given elements.  For graphical matroids, this problem is
in~$\poly$ by \emph{Menger's theorem}.  For regular matroids, this problem is in~$\poly$
by a result in \cite{FGLS16} based on
\emph{Seymour's decomposition theorem}.  One can modify examples in
Section~\ref{s:ex} to show the following:

\begin{prop} \label{p:2SC}
{\sc 2-SPANNING-CIRCUIT} \. reduces to \, $\neg\EBULC_0$ \. for binary matroids.	
\end{prop}

By the proposition, the first part of Conjecture~\ref{conj:SY-coNP} follows from
the following natural conjecture that would be analogous to Vardy's result for
the \ts {\sc GIRTH}:

\begin{conj} \label{conj:2SC}
{\sc 2-SPANNING-CIRCUIT} \.  is \ts $\NP$-complete for binary matroids.	
\end{conj}

\smallskip

\subsection{Defect}\label{ss:intro-defect}
Denote by \ts $\vp$ \ts the \defn{defect} of the SY~inequality:
\begin{equation*}% \label{eq:SY-defect}
\vp_{\bSc}(\Mf,R,a) \ := \ \aP_{\bSc}(\Mf,R,a)^2 \, - \, \aP_{\bSc}(\Mf,R,a+1) \. \aP_{\bSc}(\Mf,R,a-1).
\end{equation*}
By definition, \ts $\vp$ \ts is a rational-valued function with denominators in~$\FP$.
If the numerators were also in \ts $\FP$ \ts for binary matroids, we would have
\ts $\EBULC_k \in \poly$, implying that \ts $\PH$ \ts collapses for \ts $k\ge 1$ \ts
by Theorem~\ref{t:main-negative}.  In fact, we have a stronger result:

\begin{prop}\label{p:defect-SP-hard}
\ts $\vp$ \ts is \ts $\SP$-hard for binary matroids and all \ts $k\ge 0$.
\end{prop}

\begin{proof}
For $k=0$, let $\Mf$ be a matroid of rank~$r$, and
let \ts $\Mf':=\Mf\oplus \Mf_1$ be a direct sum of matroid $\Mf$ with a
matroid with a single element~$v$.  Note that \ts $\Mf'$ \ts has rank \ts $r+1$ \ts
and every basis in \ts $\Mf'$ \ts must contain~$v$. Let \ts $R:=\{v\}$.
Observe  that \ts $\iB(\Mf',R,0)=\iB(\Mf',R,2)=0$ \ts and \ts
$\iB(\Mf',R,1)= |\cB(\Mf)|$.  By the definition of \ts $\aP_{\bSc}(\Mf,R,a)$ \ts
and Theorem~\ref{t:KN},
we conclude that \ts $\vp$ \ts is \ts $\SP$-hard. \.  Finally, for  \ts $k\ge 1$,
the result follows from the proof of Lemma~\ref{lem:more}.
\end{proof}

Note that the argument above does not imply Conjecture~\ref{conj:SY-coNP} since, in the case outlined,
$\EBULC_0$ reduces to deciding the
  vanishing of the number of bases of a given matroid, which  is trivially in~$\poly$.
We refer to \cite{Pak-OPAC} for an extensive overview of complexity
of combinatorial inequalities and their defect.

\smallskip

\subsection{Spanning trees}
Note that for simple planar graphs, Stong's Theorem~\ref{t:Stong} is nearly
optimal since the number of spanning trees is at most exponential for planar
graphs with $n$ vertices \cite{BS10}, or even all graph with bounded average
degree, see \cite{Gri76}.  This gives \ts $\al(N) = \Omega(\log N)$\ts, where recall that $\alpha(N)$ is the smallest number of vertices  of a simple planar graph~$G$ with exactly $N$ spanning
trees.
In fact, since the number of unlabeled planar graphs with $n$ vertices
is exponential in~$n$, see e.g.\ \cite[$\S$6.9.2]{Noy15},
proving the corresponding upper bound \ts $\al(N) = O(\log N)$ \ts
is likely to be very difficult.

On the other hand, it follows from
the proof of Theorem~\ref{t:st}, that, for the smallest number of \emph{edges} of a \emph{non-simple}  graph with exactly $N$ spanning trees,
the upper bound of $\log N$  will be implied by
the celebrated \emph{Zaremba's conjecture}, see a discussion and further references
in~\cite{CP-CF} and \cite{CKP24}.

\smallskip

\subsection{Understanding the results}\label{ss:intro-under}
There are several ways to think of our results.  First and most
straightforward, we completely resolve a 1981 open problem by Stanley by
both showing that the equality cases of \eqref{eq:SY} cannot have a satisfactory
description (from a combinatorial point of view) for $k>0$, and by
deriving such a description for $k=0$.

Second, one can think of the results as a showcase for the tools.
This includes both the computational complexity and number theoretic
approach towards the proof of Theorem~\ref{t:main-negative},
and the (rather technical) combinatorial
atlas approach towards the proof of Theorem~\ref{t:main-positive}.
%Let us emphasize that Yan was unable to obtain our Theorem~\ref{t:main-positive}
%using Lorentzian polynomials (cf.~\cite[$\S$3.3]{Yan23}).  While we believe
%that Theorem~\ref{t:main-positive} is not  attainable with Lorentzian
%polynomials, we lack the formal language to make this claim
%rigorous.
Let us note that the Lorentzian-polynomial approach in
\cite[\S3.3]{Yan23} did not yield the characterization in
Theorem~\ref{t:main-positive}.  Our proof relies on the local--global
equality principle for combinatorial atlases, which propagates equality
to selected contractions of the matroid.  At present, no analogous
equality-propagation mechanism is known in the literature in the
Lorentzian-polynomial framework.

Third, one can think of Theorem~\ref{t:main-negative} as an evidence of
the strength of Lorentzian polynomials.  In combinatorics, some of the
most natural combinatorial inequalities are proved by a direct injection.
See e.g.\ \cite{CPP-effective,CPP-KS,DD85,DDP84}
for injective proofs of variations and special cases of \eqref{eq:Sta},
and to \cite{Mani10} for a rare injective proof of a matroid inequality.
% Thinking computationally,
Now, if an injection and its inverse (when defined)
are poly-time computable, this implies that the equality cases are
in~$\ts\coNP$.  Thus, having \ts $\EBULC_1 \notin \PH$ \ts shows
that Lorentzian polynomials are \emph{powerful}, in a sense that they
can prove results beyond elementary combinatorial means.

Finally, this paper gives a rare example of \emph{limits of what
is knowable} about {matroid inequalities}, as opposed to
realizability of matroids where various hardness and undecidability
results are known, see e.g.\ \cite{KY22,Sch13}.  This is especially
in sharp contrast with the equality cases for Mason's inequalities,
which are known to have easy descriptions.

\vskip.4cm

\subsection*{Acknowledgements}
We are grateful to Petter Br\"and\'en, Graham Farr, Fedor Fomin,
Milan Haiman, Jeff Kahn, Noah Kravitz, Jonathan Leake,
Daniel Lokshtanov, Steven Noble, Greta Panova, Yair Shenfeld,
Ilya Shkredov, Richard Stanley, Richard Stong, Ramon van~Handel
and Alan Yan for interesting discussions and helpful remarks.

This work was initiated in July 2023, when both authors were
visiting the American Institute of Mathematics (AIM) at their new
location in Pasadena, CA.  We continued our collaboration during
a workshop at the Institute of Pure and Applied Mathematics (IPAM),
in April 2024.  We are grateful to both AIM and IPAM for the
hospitality, and to workshop organizers for the opportunity to
participate.  Both authors were partially supported by the~NSF.

\vskip1.1cm

%%%%%%%%%%%%%%%%%%%%%%%%%%%%%%%%%%%%%%%%%%%%%%%%%%%%%%%%%%%%%%%%%%%%%%%%
%%%%%%%%%%%%%%%%%%%%%%%%%%%%%%%%%%%%%%%%%%%%%%%%%%%%%%%%%%%%%%%%%%%%%%%%
%%%%%%%%%%%%%%%%%%%%%%%%%%%%%%%%%%%%%%%%%%%%%%%%%%%%%%%%%%%%%%%%%%%%%%%%
%%%%%%%%%%%%%%%%%%%%%%%%%%%%%%%%%%%%%%%%%%%%%%%%%%%%%%%%%%%%%%%%%%%%%%%%
%%%%%%%%%%%%%%%%%%%%%%%%%%%%%%%%%%%%%%%%%%%%%%%%%%%%%%%%%%%%%%%%%%%%%%%%

%\newpage

\end{document}